\documentclass[11pt,letterpaper]{article}

\usepackage[T1]{fontenc}
\usepackage{amsmath,amsthm,mathtools}
\usepackage[charter]{mathdesign}
\usepackage[shortlabels]{enumitem}
\setlist{topsep=.5ex,itemsep=.5ex,parsep=0ex,partopsep=0ex}
\usepackage[svgnames]{xcolor}
\usepackage{tikz}
\usetikzlibrary{calc,arrows.meta,positioning}
\usepackage{needspace}
\usepackage{longtable}
\usepackage[normal]{caption}
\usepackage[bookmarksnumbered,bookmarksopen,bookmarksdepth=2,
  linktocpage=true,linktoc=all,colorlinks=true]{hyperref}
\newcommand{\hidelinks}{\hypersetup{linkcolor=Black}}
\newcommand{\restorelinks}{\hypersetup{allcolors=Navy!75!DodgerBlue,
  citecolor=DarkGreen!75!Green}}
\restorelinks
\hypersetup{pdftitle={The maximum relaxation time of a random walk on regular graphs},
  pdfauthor={Haoran Zhu}}
\allowdisplaybreaks[1]
\newtheorem{theorem}{Theorem}[section]
\newtheorem{lemma}[theorem]{Lemma}
\newtheorem{corollary}[theorem]{Corollary}
\newtheorem{proposition}[theorem]{Proposition}
\theoremstyle{definition}

\theoremstyle{remark}

\numberwithin{equation}{section}

\newcommand{\R}{\mathbb{R}}
\newcommand{\one}{\mathbf{1}}
\newcommand{\cE}{\mathcal E}

\title{The maximum relaxation time of a random walk \\on regular graphs}
\author{Haoran Zhu}
\date{Aug 30, 2026}
\AtEndDocument{\bigskip{\footnotesize
\noindent\textsc{Division of Mathematical Sciences, Nanyang Technological
University, 21 Nanyang Link, Singapore 637371}\par
\noindent E-mail address:
\href{mailto:zhuh0031@e.ntu.edu.sg}{\texttt{zhuh0031@e.ntu.edu.sg}}\par}}

\begin{document}
\maketitle

\begin{abstract}
We establish sharp quadratic bounds on the relaxation time of simple random
walk on connected regular graphs, with leading constants \(3/(2\pi^2)\) and
\(1/\pi^2\) for even and odd orders, respectively. This resolves a
longstanding conjecture known as Aldous--Fill spectral gap conjecture (2002). We also prove
uniqueness and stability theorems for the corresponding cubic and quartic
chains, and settle the quartic uniqueness conjecture posed by Abdi,
Ghorbani, and Imrich (2021) and by Abdi and Ghorbani (2023). We establish
sharp bounds on algebraic connectivity under minimum-degree and regularity
constraints and prove the conjectured chain structure of the minimisers,
identifying their repeating blocks. This resolves a longstanding conjecture
of Guiduli and Mohar (1996) and a structural conjecture of Abdi and
Ghorbani (2024). Our quantitative stability theorem shows that nearly
minimum algebraic connectivity forces nearly maximum diameter, proving their
diameter conjecture. We further obtain sharp relaxation-time bounds in terms
of edge-connectivity. For nonregular graphs, we establish sharp bounds for
the gap between maximum degree and adjacency spectral radius, confirming a
conjecture of Liu (2024). As applications, we prove the sharp bounds on
hitting and commute times conjectured by Aldous and Fill (2002) for regular graphs.
\end{abstract}

{\small
\par}

\setcounter{tocdepth}{2}
\hidelinks
\tableofcontents
\restorelinks

\section{Introduction}

The spectrum of a graph connects its combinatorial structure with the
behaviour of random walk. The relations between eigenvalues,
isoperimetry, and graph distance are fundamental to this connection
\cite{AlonMilman1985,BauerKellerWojciechowski2015,Chung1989,Dodziuk1984,LawlerSokal1988,LiuDualCheeger2015}.
They have played a central role in the study of expander graphs,
spectral gaps of random graphs, and convergence to equilibrium
\cite{HoffmanKahlePaquette2021,LubetzkySly2010,MarcusSpielmanSrivastava2015}.
This leads to extremal questions about the smallest possible spectral
gap and the structure of the graphs attaining it.

Degree constraints restrict the local density of a graph, while its
spectral gap depends on the arrangement of edges throughout the graph.
For fixed degree, the proposed extremisers consist of small dense graphs
joined in long chains.
Conjectures of Guiduli and Mohar (1996), Aldous and Fill (2002), and
Abdi and Ghorbani (2024) predict different aspects of this picture:
the structure of graphs with minimum algebraic connectivity, the
largest relaxation time of random walk on a regular graph, and the
relation between minimum algebraic connectivity and maximum diameter
\cite{GuiduliThesis,AldousFill,AbdiGhorbaniDiameter}.

In this paper, we establish sharp spectral bounds under degree
constraints, determine the structure of the minimisers, and prove
quantitative stability. These results resolve the structure conjecture
of Guiduli and Mohar and the structure and diameter conjectures of
Abdi and Ghorbani. For regular graphs, we determine the maximum
relaxation time and its unique extremiser in both order parities,
proving the Aldous--Fill spectral gap conjecture and the uniqueness
conjecture in degree four. The stability theorems show that nearly
extremal eigenvalues force the same chain structure on almost all
vertices. The degree-dependent constants also give the sharp adjacency
spectral bound conjectured by Liu (2024)
\cite{LiuSpectralRadius2024}. As applications, we prove the hitting-
and commute-time bounds conjectured by Aldous and Fill.

Throughout, graphs are finite, simple, and undirected, and are connected
unless stated otherwise. For a graph \(G\) with vertex set \(V(G)\),
edge set \(E(G)\), and order \(n=|V(G)|\geqslant2\),
let \(\mathsf A(G)\) be its adjacency matrix and \(\mathsf D(G)\)
the diagonal matrix of vertex degrees. Its \emph{algebraic connectivity}
\(\mu(G)\), introduced by Fiedler \cite{Fiedler1973}, is the second
smallest eigenvalue of the Laplacian
\(\mathsf L(G)=\mathsf D(G)-\mathsf A(G)\). A graph is
\(d\)-regular if each vertex has degree \(d\). On such a graph,
simple random walk moves at each step to a uniformly chosen neighbour,
and its spectral gap and relaxation time are
\[
 \lambda(G)=\frac{\mu(G)}d,
 \qquad \tau(G)=\frac1{\lambda(G)}=\frac d{\mu(G)}.
\]
This is also the relaxation time of the continuous-time walk with
independent exponential holding times of mean one; see
\cite{AldousFill,LevinPeresWilmer,Lovasz1996}.

\subsection{Algebraic connectivity under degree constraints}

Write \(\delta(G)\) for the minimum degree of \(G\). For integers
\(d\geqslant3\) and \(n\geqslant d+1\), let \(\mathcal D_{n,d}\)
be the class of connected graphs of order \(n\) with minimum degree
at least \(d\), and let \(\mathcal R_{n,d}\) be the class of
connected \(d\)-regular graphs of order \(n\).
All asymptotic statements at fixed degree are understood as
\(n\to\infty\), through orders for which the relevant class is
nonempty. A subscript \(d\) in \(O_d\) indicates dependence on \(d\).

Without a degree constraint, the path \(P_n\) is the unique graph
of order \(n\) with minimum algebraic connectivity, and
\(\mu(P_n)=2(1-\cos(\pi/n))\sim\pi^2/n^2\)
\cite{Fiedler1973}. The degree-constrained problem has been studied
since the early work on cubic (\(3\)-regular) graphs
\cite{BussemakerEtAl,GuiduliThesis,Guiduli1997}.
Brand, Guiduli, and Imrich \cite{BrandGuiduliImrich} determined the
unique cubic minimiser at each even order. The sharp leading constants
were subsequently obtained for cubic and quartic (\(4\)-regular)
graphs, and for graphs of minimum degree three
\cite{AbdiGhorbaniImrich,AbdiGhorbaniQuartic,AbdiGhorbaniMinDegree}.
For general degree, Abdi and Ghorbani
\cite[Theorem~1.8]{AbdiGhorbaniDiameter} obtained the predicted
constants assuming that the diameter is nearly maximum. Our first
theorem determines these constants from the degree condition alone.
Put
\[
 c_d=\begin{cases}
 d-1,&d\text{ odd},\\
 2(d-2),&d\text{ even}.
 \end{cases}
\]

\begin{theorem}\label{thm:degree-bounds-intro}
Let \(d\geqslant3\) be a fixed integer. Then
\[
 \min_{G\in\mathcal D_{n,d}}\mu(G)
   =(d-1+o(1))\frac{\pi^2}{n^2},
 \qquad
 \min_{G\in\mathcal R_{n,d}}\mu(G)
   =(c_d+o(1))\frac{\pi^2}{n^2}.
\]
\end{theorem}

Thus regularity leaves the leading constant unchanged in odd degree
and increases it in even degree. To explain this distinction, recall
that an \emph{edge cut} consists of all edges between a nonempty
proper vertex set and its complement, and a \emph{bridge} is an edge
whose deletion disconnects the graph. Every edge cut in a regular
graph of even degree has even size, so such a graph has no bridges,
whereas the minimum-degree constructions use them. The finite bound
underlying the first assertion is Theorem~\ref{thm:minimum-degree-bound};
the regular bounds and sharp constructions are given in
Theorem~\ref{thm:fixed-regular-sharp}.

We next describe the graphs that attain the minimum. Let \(K_d\)
denote the complete graph on \(d\) vertices, and let
\(L_d=K_{d+1}-e\), with the endpoints of the deleted edge \(e\)
as its \emph{terminals}. A bridge chain of \(L_d\) is formed from
disjoint copies by joining a terminal of each copy to a terminal of
the next, using each terminal at most once. For even \(d\), let
\(M_d\) consist of four consecutive cliques of orders
\(1,2,d-2,1\), with all edges between consecutive cliques and no
edges between other distinct cliques, and with the two singleton
vertices as its terminals. Identifying the second terminal of each
copy with the first terminal of the next gives a chain of \(M_d\).
In both constructions, connected graphs attached at the two ends
complete the degree requirements; the two types of middle are shown
in Figure~\ref{fig:intro-chains}.

\begin{figure}[htbp]
\centering
\begin{minipage}[t]{.48\textwidth}
\centering
\begin{tikzpicture}[x=.95cm,y=.8cm,
  vertex/.style={circle,fill,inner sep=1.5pt},line width=.5pt]
\foreach \p/\x/\y in {a/0/0,b/.8/.5,c/.8/-.5,d/1.6/0,
  e/2.3/0,f/3.1/.5,g/3.1/-.5,h/3.9/0}
  \node[vertex] (\p) at (\x,\y) {};
\foreach \a/\b in {a/b,a/c,b/c,b/d,c/d,d/e,e/f,e/g,f/g,f/h,g/h}
  \draw (\a)--(\b);
\draw[densely dotted] (-.45,0)--(a);
\draw[densely dotted] (h)--(4.35,0);
\node at (.8,-1) {\small \(L_3\)};
\node at (3.1,-1) {\small \(L_3\)};
\end{tikzpicture}
\end{minipage}\hfill
\begin{minipage}[t]{.48\textwidth}
\centering
\begin{tikzpicture}[x=.95cm,y=.8cm,
  vertex/.style={circle,fill,inner sep=1.5pt},line width=.5pt]
\foreach \p/\x/\y in {a/0/0,b/.65/.5,c/.65/-.5,
  d/1.3/.5,e/1.3/-.5,f/1.95/0,g/2.6/.5,h/2.6/-.5,
  i/3.25/.5,j/3.25/-.5,k/3.9/0}
  \node[vertex] (\p) at (\x,\y) {};
\foreach \a/\b in {a/b,a/c,b/c,b/d,b/e,c/d,c/e,d/e,d/f,e/f,
  f/g,f/h,g/h,g/i,g/j,h/i,h/j,i/j,i/k,j/k}
  \draw (\a)--(\b);
\draw[densely dotted] (-.45,0)--(a);
\draw[densely dotted] (k)--(4.35,0);
\node at (.975,-1) {\small \(M_4\)};
\node at (2.925,-1) {\small \(M_4\)};
\end{tikzpicture}
\end{minipage}
\caption{The middles of the cubic and quartic chains. Successive copies
of \(L_3\) are joined by a bridge; successive copies of \(M_4\) share
a terminal. End subgraphs are omitted.}
\label{fig:intro-chains}
\end{figure}
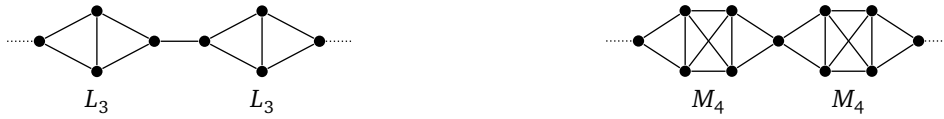

The leading term does not distinguish the possible end subgraphs
of bounded order. The next theorem determines the middle of every
minimiser and confines the remaining choices to its two ends.

A \emph{cut vertex} is a vertex whose deletion disconnects the graph,
and a \emph{block} is a maximal connected subgraph with no cut vertex.
The \emph{block-cut tree} has the blocks and cut vertices as its
vertices, with incidence as adjacency; a graph is \emph{path-like}
if this tree is a path. Guiduli and Mohar conjectured that
every minimum-degree minimiser is path-like, with a chain of \(L_d\)
between two end subgraphs of bounded order. Abdi and Ghorbani proposed
the corresponding regular structure, with \(L_d\) in odd degree
and \(M_d\) in even degree
\cite[Conjectures~1.4 and~1.5]{AbdiGhorbaniDiameter}.
The following theorem determines the structure in both classes.

\Needspace{11\baselineskip}
\begin{theorem}\label{thm:guiduli-mohar-pathlike}
Let \(d\geqslant3\) be a fixed integer. Then, for all sufficiently large
\(n\), every minimiser of \(\mu\) in \(\mathcal D_{n,d}\) or
\(\mathcal R_{n,d}\) is path-like, has two end subgraphs of total
order \(O_d(1)\), and has middle
\begin{enumerate}[(i)]
\item a bridge chain of \(L_d\) in \(\mathcal D_{n,d}\), and in
\(\mathcal R_{n,d}\) when \(d\) is odd;
\item a chain of \(M_d\), identified at consecutive terminals, in
\(\mathcal R_{n,d}\) when \(d\) is even.
\end{enumerate}
The assertion for \(\mathcal D_{n,d}\) also holds when the minimum
degree is required to equal \(d\).
\end{theorem}

Theorem~\ref{thm:guiduli-mohar-pathlike} settles the structure
conjecture of Guiduli and Mohar from 1996 \cite{GuiduliThesis}
and its regular analogue. The degree-two case of the former was
proved by Lal, Patra, and Sahoo \cite{LalPatraSahoo}; the
degree-one case is the classical path theorem. In each class,
only the two end subgraphs remain to be determined, and their orders
are bounded in terms of the degree.

The stability theorems quantify how the chain structure persists
when the eigenvalue is close to the minimum.
For \(G\in\mathcal D_{n,d}\) satisfying
\(n^2\mu(G)\leqslant(d-1)\pi^2+\varepsilon\), where
\(0\leqslant\varepsilon\leqslant1\), all but
\(O_d((\varepsilon+n^{-1})^{1/3}n)\) vertices belong to induced
copies of \(L_d\) joined in chains. The corresponding regular
statement holds with \(c_d\pi^2\) and the chains appropriate to
the degree parity. Theorems~\ref{thm:minimum-degree-stability},
\ref{thm:parity-stability}, and~\ref{thm:fixed-even-stability}
give these quantitative bounds on the number of vertices outside
the chains, with the optimal exponent \(1/3\) in the minimum-degree
class (Proposition~\ref{prop:stability-exponent-sharp}).

One consequence is a sharp relation between minimum algebraic
connectivity and maximum diameter. Write \(\operatorname{diam}(G)\) for the diameter,
the largest distance between two vertices. Abdi and Ghorbani
\cite[Conjecture~5.4]{AbdiGhorbaniDiameter} asked whether
asymptotically minimum algebraic connectivity forces asymptotically
maximum diameter. We prove the following quantitative form.

\begin{corollary}[Abdi--Ghorbani diameter conjecture]
\label{cor:diameter-stability-intro}
Let \(d\geqslant3\) be fixed and \(0\leqslant\varepsilon\leqslant1\).
Suppose that \(G\in\mathcal D_{n,d}\) satisfies
\(n^2\mu(G)\leqslant(d-1)\pi^2+\varepsilon\), or that
\(G\in\mathcal R_{n,d}\) satisfies
\(n^2\mu(G)\leqslant c_d\pi^2+\varepsilon\).
Then, for all sufficiently large \(n\) depending only on \(d\),
\[
 \frac{3n}{d+1}-O_d\bigl((\varepsilon+n^{-1})^{1/3}n\bigr)
 \leqslant\operatorname{diam}(G)
 \leqslant3\left\lfloor\frac n{d+1}\right\rfloor-1.
\]
\end{corollary}

In particular, an eigenvalue within \(o(n^{-2})\) of the minimum
forces diameter \(3n/(d+1)-o(n)\). We also determine when the
converse holds: nearly maximum diameter implies nearly minimum
algebraic connectivity for minimum degree three and for quartic
graphs, hence also for cubic graphs
(Theorem~\ref{thm:diameter-converse}). Together with the
counterexamples of Abdi and Ghorbani in every higher degree
\cite[Theorem~5.3]{AbdiGhorbaniDiameter}, this completes the
answer to their converse question.

\subsection{Extremal regular graphs and relaxation time}

For a regular graph, minimising the spectral gap is equivalent to
maximising the relaxation time of simple random walk. Allowing the
degree to vary leads to the Aldous--Fill spectral gap conjecture.
In 2002, Aldous and Fill predicted that the largest relaxation time
among connected regular graphs of order \(n\) is at most
\((1+o(1))3n^2/(2\pi^2)\), with cubic graphs attaining the constant
\cite[Open Problem~6.14, p.~217]{AldousFill}.
Among unrestricted connected graphs, Aksoy, Chung, Tait, and Tobin
\cite[Theorem~4]{AksoyChungTaitTobin} proved that the maximum is
asymptotic to \(n^3/54\), attained asymptotically by two dense subgraphs
joined by a long path. Regularity imposes a uniform stationary distribution and changes
both the scale and the extremal construction.

We prove the Aldous--Fill conjecture and determine the extremal graph
in each order parity. For \(n\geqslant3\), define
\[
 \tau_{\max}(n)=
 \max\{\tau(G):G\text{ is a connected regular graph of order }n\}.
\]
For even \(n\geqslant4\), let \(X_n\) be the cubic chain
determined by Brand, Guiduli, and Imrich
\cite[Theorem~1]{BrandGuiduliImrich}; its construction is recalled
in Section~\ref{sec:cubic-sharpness}. For \(n\geqslant11\), let
\(\mathcal G_n\) be the quartic chain of Abdi and Ghorbani
\cite[Definition~1.4]{AbdiGhorbaniQuartic}, recalled in
Subsection~\ref{sec:quartic-extremisers}. Its middle consists of
copies of \(M_4\), and its end subgraphs depend on the order
modulo five. The next theorem identifies the extremisers over all
regular degrees and proves the uniqueness of the quartic chain.

\Needspace{14\baselineskip}
\begin{theorem}\label{thm:main}
\label{thm:relaxation-extrema-intro}
Let \(n\) be a sufficiently large integer. Then \(\tau_{\max}(n)\)
is attained uniquely, up to isomorphism, by \(X_n\) if \(n\) is even
and by \(\mathcal G_n\) if \(n\) is odd. Moreover,
\[
 \tau_{\max}(n)=\begin{cases}
 \displaystyle\frac{3n^2}{2\pi^2}+\frac12+O(n^{-1}),&n\text{ even},\\[5pt]
 \displaystyle\frac{n^2}{\pi^2}+\frac{29}{60}+O(n^{-1}),&n\text{ odd}.
 \end{cases}
\]
For every \(n\geqslant11\), the graph \(\mathcal G_n\) is the
unique minimiser of algebraic connectivity in \(\mathcal R_{n,4}\),
up to isomorphism.
\end{theorem}

The first assertion strengthens the Aldous--Fill bound by identifying
the unique extremiser and the next term in each order parity.
In particular, for every \(\varepsilon>0\) and all sufficiently
large \(n\), every connected regular graph of order \(n\) satisfies
\begin{equation}\label{eq:intro-AF}
 \lambda(G)\geqslant(1-\varepsilon)\frac{2\pi^2}{3n^2},
\end{equation}
uniformly over the degree. The quartic assertion settles the
uniqueness conjecture posed by Abdi, Ghorbani, and Imrich (2021)
\cite[Conjecture~3.16]{AbdiGhorbaniImrich} and by Abdi and Ghorbani
(2023) \cite[Conjecture~1.5]{AbdiGhorbaniQuartic}. The latter
authors had determined the possible quartic minimisers up to their
end subgraphs and proved
\begin{equation}\label{eq:quartic-known}
 \min_{G\in\mathcal R_{n,4}}\mu(G)
       =(1+o(1))\frac{4\pi^2}{n^2}.
\end{equation}
Our exact comparison resolves the choice of end subgraphs. The proofs are given in
Theorems~\ref{thm:max-relaxation} and~\ref{thm:quartic-unique}.
For the related cubic bipartite problem, see \cite{LiuXue}.

The same extremal graphs describe all regular graphs whose relaxation
time is close to the maximum. For graphs on the same vertex set,
\(E(G)\mathbin{\triangle}E(H)\) denotes the symmetric difference
of their edge sets.

\begin{corollary}\label{cor:relaxation-stability-intro}
There are absolute constants \(C,\varepsilon_0>0\) such that the
following holds for all sufficiently large \(n\). Let \(G\) be a
connected regular graph of order \(n\), and suppose that
\(0\leqslant\varepsilon\leqslant\varepsilon_0\) and
\(\tau(G)\geqslant(1-\varepsilon)\tau_{\max}(n)\).
Then \(G\) is cubic for even \(n\) and quartic for odd \(n\).
After relabelling the vertices,
\begin{align*}
 |E(G)\mathbin{\triangle}E(X_n)|
 &\leqslant C(\varepsilon+n^{-1})^{1/3}n,
 &&n\text{ even},\\
 |E(G)\mathbin{\triangle}E(\mathcal G_n)|
 &\leqslant C(\varepsilon+n^{-1})^{1/3}n,
 &&n\text{ odd}.
\end{align*}
\end{corollary}

Thus a nearly maximum relaxation time determines the degree exactly
and the graph up to a small number of edge changes. This is
Corollary~\ref{cor:global-stability}, which also gives the diameter
estimate in both order parities.

To compare all regular degrees, we also need bounds when the degree
grows with the order. In this case, we obtain a sharp estimate in
terms of \emph{edge-connectivity}, the minimum size of an edge cut.
The extremal graphs in this range consist of two nearly complete
graphs joined by a small cut.

\begin{theorem}\label{thm:growing-degree-intro}
Let \(\kappa\) be a positive integer and \((G_m)\) a sequence of
connected regular graphs of orders \(n_m\), degrees \(d_m\to\infty\),
and edge-connectivity at least \(\kappa\). Then
\[
 \liminf_{m\to\infty}n_m^2\lambda(G_m)\geqslant8\kappa,
\]
and the constant is best possible. Furthermore,
\(n_m^2\lambda(G_m)\to8\kappa\) if and only if, for all
sufficiently large \(m\), there is an edge cut of size \(\kappa\)
separating \(V(G_m)\) into two sets of \(d_m+o(d_m)\) vertices each.
\end{theorem}

In the equality case, the cut is unique, \(n_m\sim2d_m\), and
each of its two sides is missing only \(o(d_m^2)\) edges from a
complete graph, giving two nearly complete graphs joined by a
smallest possible cut. Taking \(\kappa=1\) gives the sharp constant
eight, while the parity of edge cuts gives sixteen in even regular
degree. These bounds separate growing degree from the cubic and
quartic extremisers and complete the maximisation over all degrees
(Theorems~\ref{thm:edge-connectivity-intro}
and~\ref{thm:growing-stability}).

\subsection{Adjacency spectral radius and hitting times}
\label{sec:intro-applications}

The constants in Theorem~\ref{thm:degree-bounds-intro} also answer
an extremal question for the adjacency matrix. Write \(\Delta(G)\)
for the maximum degree and \(\rho(G)\) for the adjacency spectral
radius. The Perron--Frobenius theorem gives
\(\rho(G)\leqslant\Delta(G)\), with equality precisely when
\(G\) is regular. How close can a nonregular graph come to equality?
This problem has been studied in terms of the order, degree, and
diameter \cite{CioabaGregoryNikiforov2007,Cioaba2007,LiuSpectralRadius2024}.
Liu (2024) determined the sharp bound when the maximum degree is
three or four and conjectured its value in every fixed degree
\cite[Conjecture~7.1]{LiuSpectralRadius2024}.

For \(d\geqslant3\) and \(n\geqslant d+1\), let \(\rho(n,d)\)
be the largest adjacency spectral radius among connected nonregular
graphs of order \(n\) and maximum degree \(d\).

\begin{theorem}[Liu's conjecture]\label{thm:liu-spectral-radius}
Let \(d\geqslant3\) be a fixed integer. Then
\[
 d-\rho(n,d)=(c_d+o(1))\frac{\pi^2}{4n^2}.
\]
\end{theorem}

Liu constructed the sharp families in every degree
\cite[Theorems~6.1--6.3]{LiuSpectralRadius2024}.
Our lower bound follows by completing two copies of an extremal
nonregular graph to a regular graph of order \(2n+O_d(1)\) and
applying Theorem~\ref{thm:degree-bounds-intro}. This gives a direct
connection between the adjacency and Laplacian extremal problems,
including the dependence of their constants on degree parity.
We also prove that the constant one in Cioab\u{a}'s bound
\[
 \Delta(G)-\rho(G)>\frac1{n\operatorname{diam}(G)}
\]
is best possible, even for degree sequence \((3,\ldots,3,2)\)
(Proposition~\ref{prop:sharp-diameter-spectral-radius-constant}).
This answers the related question discussed in
\cite{CioabaGregoryNikiforov2007,Cioaba2007} and raised again in
\cite[Section~7]{LiuSpectralRadius2024}.

The same degree constraints yield sharp bounds on the time needed
for a random walk to reach a specified vertex. For
\(u,v\in V(G)\), let \(T_v\) be the first time simple random walk
visits \(v\), and write \(\mathbb E_u\) for expectation when it
starts at \(u\). Define the maximum hitting and commute times by
\[
 t_{\mathrm{hit}}(G)=\max_{u,v\in V(G)}\mathbb E_uT_v,
 \qquad
 \tau^*(G)=\max_{u,v\in V(G)}
       \bigl(\mathbb E_uT_v+\mathbb E_vT_u\bigr).
\]
Aldous and Fill \cite[Open Problem~6.14]{AldousFill} conjectured
sharp quadratic bounds for both quantities on regular graphs,
alongside their spectral gap conjecture. We prove the following
finite bounds.

\begin{corollary}[Aldous--Fill hitting- and commute-time conjectures]
\label{cor:extremal-times-intro}
Let \(G\) be a connected regular graph of order \(n\geqslant3\). Then
\[
 t_{\mathrm{hit}}(G)\leqslant\frac34(n-1)(n+3),
 \qquad
 \tau^*(G)\leqslant\frac32n^2-3n.
\]
Both leading constants are best possible.
\end{corollary}

Along even orders, the cubic chains are asymptotically extremal for
hitting and commute times as well as relaxation time.
The degree-dependent bounds and proofs are given
in Theorem~\ref{thm:extremal-times}. For the lazy walk, which stays
at its present vertex with probability \(1/2\) and otherwise makes
a simple random-walk step, the relaxation-time bound also yields
mixing time \(O(n^2\log n)\). Here mixing time is the least time
at which the distribution is within total variation distance
\(1/4\) of the stationary distribution for every initial vertex; see
\cite[Theorem~12.4]{LevinPeresWilmer}.

\subsection{Proofs and organisation}

Our proofs begin by ordering the vertices according to an eigenvector
for algebraic connectivity, called a \emph{Fiedler vector}, so that
the degree condition restricts the cuts between consecutive vertices.
Since an edge may cross several cuts, a sharp comparison with the
energy of a path must retain its contribution to each of them.
This comparison gives the leading constants, including the stronger
bound from cut parity in even regular degree, while its error terms
control the number of vertices outside the prescribed chains and
yield stability.

For growing degree, reciprocal cut sizes lead to a weighted path
problem whose sharp bound and equality conditions give
Theorem~\ref{thm:growing-degree-intro}. An exact comparison of the
end subgraphs gives uniqueness in degree four, with the small orders
checked directly. To determine the structure of exact minimisers
in every fixed degree, we compare their eigenvalue equations under
local rearrangements of the blocks, excluding branching and confining
every exceptional block to an end subgraph of bounded order.

In Section~\ref{sec:preliminaries}, we introduce the notation and
record the variational and cut identities used in the proofs.
In Section~\ref{sec:path-estimates}, we prove the fixed-degree
bounds and the stability and diameter theorems.
In Section~\ref{sec:reciprocal-cuts}, we determine the maximum relaxation
time and its unique extremisers, obtain refined asymptotics, and prove
quantitative stability.
In Section~\ref{sec:exact-canonical-middle}, we prove
the structure theorem for exact minimisers in every fixed degree.
In Section~\ref{sec:applications}, we apply the spectral and cut
bounds to adjacency eigenvalues and to hitting and commute times.

\medskip
\noindent\textbf{Independent work.}\quad
The first version of this paper proved the Aldous--Fill spectral gap
conjecture. During the preparation of this expanded version, we became
aware of two preprints by Abdi and Ghorbani
\cite{AbdiGhorbaniPartI,AbdiGhorbaniPartII}, which independently prove
this conjecture and obtain fixed-degree bounds, structural theorems,
and several diameter results overlapping with those presented here.
Their preprints use similar ideas in the fixed-degree setting and do not
include a statement on the use of AI. The AI-assisted proofs in the
expanded version of this paper were developed before these
preprints~\cite{AbdiGhorbaniPartI,AbdiGhorbaniPartII} became publicly
available, and neither preprint was used as a source in developing
those proofs.

The expanded version also proves the quartic uniqueness conjecture
\cite[Conjecture~1.5]{AbdiGhorbaniQuartic} and Liu's adjacency spectral
radius conjecture \cite[Conjecture~7.1]{LiuSpectralRadius2024}, and gives
asymptotically sharp hitting- and commute-time bounds corresponding to
Aldous and Fill's Open Problem~6.14 \cite{AldousFill}. It also includes
results on quartic near-minimisers and converse diameter implications.
These additional conclusions are not stated as results in their
preprints, although the commute-time upper bound follows from a
resistance estimate common to both works
\cite[Theorem~5.3]{AbdiGhorbaniPartI}. Their Part~I also establishes
a sharp degree-uniform asymptotic theorem for graphs of prescribed
minimum degree, including nonregular graphs
\cite[Theorem~1.6]{AbdiGhorbaniPartI}.
Our use of generative AI is described below.

\medskip
\noindent\textbf{Use of generative AI.}\quad
After completing an initial draft, the author used ChatGPT 5.6 and
ChatGPT 6 Plus to explore possible extensions and applications and to
assist in developing the expanded manuscript. These tools were also
used for computations on examples, English language editing, and
identifying references. The author takes full responsibility for the
mathematical content and the writing of the manuscript.

\section{Preliminaries}\label{sec:preliminaries}

In this section, we fix the spectral notation and record the variational
and cut identities used throughout the paper.

\subsection{Laplacian eigenvalues}

Let \(G\) be a connected graph of order \(n\). We write \(\mathsf I\)
for the identity matrix, with dimension determined by the context,
\(\one\) for the constant function one on \(V(G)\), and
\(\|\cdot\|_2\) for the Euclidean norm on \(\R^{V(G)}\).
For \(x:V(G)\to\R\), its Dirichlet energy is
\[
 \cE_G(x)=\sum_{uv\in E(G)}(x(u)-x(v))^2.
\]
The variational characterisation of algebraic connectivity gives
\[
 \mu(G)=
 \min_{\substack{x\in\R^{V(G)}\setminus\{0\}\\
                  \sum_{v\in V(G)}x(v)=0}}
 \frac{\cE_G(x)}{\sum_{v\in V(G)}x(v)^2}.
\]
An eigenvector for \(\mu(G)\) is called a Fiedler vector.
A function \(x:V(G)\to\R\) is \emph{harmonic} at a vertex
\(v\) if \((\mathsf L(G)x)(v)=0\).

\subsection{Fiedler orderings and edge cuts}

In this subsection, \(G\) is \(d\)-regular unless stated otherwise.
Fix a Fiedler vector \(f\), and label the vertices \(v_1,\ldots,v_n\) so that
\[
 f_1\leqslant f_2\leqslant\cdots\leqslant f_n,
 \qquad f_i=f(v_i),\qquad \sum_{i=1}^{n}f_i=0.
\]
\begin{samepage}
We call \(f\), with this labelling, an ordered Fiedler vector.  For
\(1\leqslant i<n\), let \(y_i=f_{i+1}-f_i\), and write
\(\mathbf y=(y_1,\ldots,y_{n-1})^{\mathsf T}\).  Let
\[
 S_i=\{v_1,\ldots,v_i\},\qquad
 q_i=|\partial S_i|,
\]
where \(\partial S_i\) is the set of edges having exactly one endpoint in
\(S_i\). The sets \(S_i\) and \(V(G)\setminus S_i\) are
the two shores of this cut. We also put \(S_0=\varnothing\), \(S_n=V(G)\), and
\(q_0=q_n=0\).
\end{samepage}
For \(S\subseteq V(G)\), let \(e(S)\) denote the number of edges
with both endpoints in \(S\), and for disjoint vertex sets \(X,Y\),
let \(e(X,Y)\) denote the number of edges joining them. The ordering
gives \(y_i\geqslant0\), while connectivity gives \(q_i\geqslant1\)
for \(1\leqslant i<n\); the identity
\begin{equation}
 q_i=di-2e(S_i),
 \label{eq:cut-parity}
\end{equation}
also gives \(q_i\equiv di\pmod 2\).

The following inequality relates the sizes of two such cuts whose indices
differ by at most \(d\).

\begin{lemma}\label{lem:interval-cut}
Let \(G\) be a connected simple graph of order \(n\) and minimum degree
at least \(d\), and let
\(q_i\) be the cut sizes in a Fiedler ordering of \(G\).  Let
\(0\leqslant i<j\leqslant n\), and suppose that \(j-i\leqslant d\).  Then
\(q_i+q_j\geqslant(j-i)(d-j+i+1)\).
\end{lemma}

\begin{proof}
Let \(t=j-i\) and \(T=S_j\setminus S_i=\{v_{i+1},\ldots,v_j\}\).  Since \(G\)
is simple and has minimum degree at least \(d\),
\(|\partial T|\geqslant dt-2e(T)\geqslant
dt-t(t-1)=t(d-t+1)\).
The identity
\[
 q_i+q_j=|\partial T|+2e(S_i,V(G)\setminus S_j),
\]
therefore gives \(q_i+q_j\geqslant t(d-t+1)\).
\end{proof}

For \(1\leqslant i,j<n\), let
\(K_{ij}=|\partial S_i\cap\partial S_j|\) and write
\(\mathsf K=(K_{ij})\). We call \(\mathsf K\) the cut matrix,
and its entries the cut kernel, of this ordering. The matrix \(\mathsf K\) is the Gram matrix
of the cut indicators, and the energy has the following expression
in terms of \(\mathbf y\):

\begin{equation}
 \cE_G(f)=\mathbf y^{\mathsf T}\mathsf K\mathbf y,\qquad
 K_{ii}=q_i,\qquad
 2K_{i,i+1}=q_i+q_{i+1}-d.
 \label{eq:kernel}
\end{equation}
In particular,
\begin{equation}
 \cE_G(f)\geqslant
 \sum_{i=1}^{n-1}q_iy_i^2+
 \sum_{i=1}^{n-2}(q_i+q_{i+1}-d)y_iy_{i+1}.
 \label{eq:tridiagonal-kernel}
\end{equation}

Indeed, each edge \(v_av_b\) with \(a<b\) contributes
\((f_b-f_a)^2=(\sum_{i=a}^{b-1}y_i)^2\) to the energy.
Expanding and summing over all edges proves the first identity in
\eqref{eq:kernel}; its diagonal entries are \(K_{ii}=|\partial S_i|=q_i\).

The symmetric difference
\(\partial S_i\mathbin{\triangle}\partial S_{i+1}\) consists precisely of
the \(d\) edges incident with \(v_{i+1}\), which proves the formula for
the adjacent entries. Since every entry of \(\mathsf K\) and every
coordinate of \(\mathbf y\) is nonnegative, discarding the terms
indexed by nonadjacent pairs in \(\mathbf y^{\mathsf T}\mathsf K\mathbf y\)
gives \eqref{eq:tridiagonal-kernel}.

\subsection{Effective resistance}

Regard each edge of \(G\) as a unit resistor. For distinct
vertices \(s,t\), the \emph{effective resistance} \(R_G(s,t)\) is the
voltage difference produced by passing a unit current from \(s\) to
\(t\). If \(\mathbf e_v\) denotes the indicator of a vertex \(v\),
the potential \(p:V(G)\to\R\), normalised by \(p(t)=0\), satisfies
\[
 \mathsf L(G)p=\mathbf e_s-\mathbf e_t,
 \qquad R_G(s,t)=p(s).
\]
The Dirichlet principle gives
\[
 R_G(s,t)^{-1}=
 \min\{\cE_G(h):h:V(G)\to\R,\ h(s)=1,\ h(t)=0\}.
\]
An \(s\)--\(t\) \emph{unit flow} assigns a real number
\(\theta(u,v)\) to each ordered pair of adjacent vertices, with
\(\theta(u,v)=-\theta(v,u)\), and satisfies
\[
 \sum_{u:uv\in E(G)}\theta(v,u)
   =\mathbf e_s(v)-\mathbf e_t(v)
 \qquad(v\in V(G)).
\]
Thomson's principle states that
\[
 R_G(s,t)=\min_{\theta}\sum_{uv\in E(G)}\theta(u,v)^2,
\]
where the minimum is over all \(s\)--\(t\) unit flows and each
undirected edge is counted once. More generally, if edge \(uv\)
has resistance \(r_{uv}>0\), the same principle holds with
\(\sum_{uv\in E(G)}r_{uv}\theta(u,v)^2\) in place of the
unweighted sum.
We also use the commute-time identity
\[
 \mathbb E_sT_t+\mathbb E_tT_s=2|E(G)|R_G(s,t),
\]
with the hitting-time notation introduced in Section~\ref{sec:intro-applications}. These standard
facts are discussed in \cite{AldousFill,LevinPeresWilmer,Lovasz1996}.
\section{Sharp bounds and stability}\label{sec:path-estimates}

In this section, we prove the sharp bounds for algebraic connectivity
under minimum-degree and regularity constraints. We first establish a
comparison with the path in degree three, then extend the argument to
general minimum degree and to even regular degree. The terms retained
in these comparisons yield the stability theorems and the diameter
characterisation.

\subsection{A path comparison in degree three}

Let \(G\) be a connected cubic graph of order \(n\), and let \(f\)
be an ordered Fiedler vector. We use the cut notation from
Section~\ref{sec:preliminaries}.

An index \(a\) for which \(q_a=1\) will be
called a bridge index.  Lemma~\ref{lem:interval-cut}, applied with
\(t=1,2,3\), gives
\begin{equation}
 q_{a+1}\geqslant2,\qquad q_{a+2}\geqslant3,
 \qquad q_{a+3}\geqslant2.
 \label{eq:cubic-cut-profile}
\end{equation}
It also shows that distinct bridge indices differ by at least four.  The
endpoint bound \(q_i\geqslant i(4-i)\), and its counterpart for the
complementary shore, give \(4\leqslant a\leqslant n-4\).

For a bridge index \(a\), consider the principal block of
\eqref{eq:tridiagonal-kernel} on the indices \(a,\ldots,a+3\).  Replacing the
four cut sizes by the lower bounds \((1,2,3,2)\) in
\eqref{eq:cubic-cut-profile} gives the quadratic form in
\(\mathbf x=(x_0,x_1,x_2,x_3)^{\mathsf T}\)
\[
 Q(\mathbf x)=x_0^2+2x_1^2+3x_2^2+2x_3^2
       +2x_1x_2+2x_2x_3.
\]
We shall use the identity
\begin{equation}
 \begin{split}
 Q(\mathbf x)-\frac12(x_0+x_1+x_2+x_3)^2
  ={}&\frac12(x_0-x_1-x_2-x_3)^2\\
    &+2x_2^2+(x_1-x_3)^2.
 \end{split}
 \label{eq:cubic-local-identity}
\end{equation}

For every bridge index \(a\), let
\(W_a=\{a,a+1,a+2,a+3\}\).  These sets are pairwise disjoint.  Given an
ordered Fiedler vector \(f\), let \(g=(g_1,\ldots,g_n)\) agree with \(f\)
outside the interiors of the intervals \(\{a,\ldots,a+4\}\), and interpolate
affinely between their endpoints:
\begin{equation}
 g_{a+t}=f_a+\frac{t}{4}(f_{a+4}-f_a)
 \qquad(0\leqslant t\leqslant4).
 \label{eq:cubic-interpolant}
\end{equation}

\begin{proposition}\label{prop:cubic-path}
Let \(G\) be a connected simple cubic graph of order \(n\), and let
\(f\) be an ordered Fiedler vector. Then
\begin{equation}
 \cE_G(f)\geqslant2\cE_{P_n}(g),\qquad
 \|f-g\|_2^2\leqslant6\cE_G(f).
 \label{eq:cubic-path-energy}
\end{equation}
\end{proposition}

\begin{proof}
Retain in \eqref{eq:tridiagonal-kernel} the principal block on each
\(W_a\) and the diagonal terms outside their union.  By
\eqref{eq:cubic-local-identity}, the contribution from \(W_a\) is at least
\[
 \frac12(f_{a+4}-f_a)^2
 =2\sum_{t=0}^{3}(g_{a+t+1}-g_{a+t})^2.
\]
If \(i\notin\bigcup_aW_a\), then \(q_i\geqslant2\) and
\(g_{i+1}-g_i=y_i\).  Summing proves the first inequality in
\eqref{eq:cubic-path-energy}.

Both \(f_i\) and \(g_i\) lie between the endpoint values on each interpolated
interval.  Hence
\[
 \|f-g\|_2^2
 \leqslant3\sum_a(f_{a+4}-f_a)^2
 \leqslant6\cE_G(f),
\]
which proves the second inequality in \eqref{eq:cubic-path-energy}.
\end{proof}

The algebraic connectivity of the path \(P_n\) is
\(\mu(P_n)=2(1-\cos(\pi/n))\).  We shall use the following consequence of
its Poincar\'e inequality in both degree ranges.

\begin{lemma}\label{lem:path-transfer}
Let \(G\) be a connected graph of order \(n\), let \(f\) be a unit Fiedler
vector indexed by \(v_1,\ldots,v_n\), and let \(g=(g_1,\ldots,g_n)\).  Suppose
that \(\alpha,\beta>0\) and
\[
 \cE_G(f)\geqslant \alpha\cE_{P_n}(g),
 \qquad
 \|f-g\|_2^2\leqslant \beta\cE_G(f).
\]
Then
\[
 \mu(G)\geqslant
 \frac{\alpha\mu(P_n)}
 {\bigl(1+\sqrt{\alpha\beta\mu(P_n)}\bigr)^2}.
\]
\end{lemma}

\begin{proof}
Put \(\mu=\mu(G)\), and let \(\mathsf P\) be the orthogonal projection onto
\(\one^\perp\).  Since \(f=\mathsf Pf\), the contraction property gives
\(\|\mathsf Pg\|_2\geqslant1-\sqrt{\beta\mu}\).
Subtracting a constant does not change the path energy.  Hence
\[
 \mu\geqslant \alpha\mu(P_n)\max\{1-\sqrt{\beta\mu},0\}^2.
\]
If \(\beta\mu<1\), taking square roots and rearranging gives
\[
 \sqrt\mu\geqslant
 \frac{\sqrt{\alpha\mu(P_n)}}
      {1+\sqrt{\alpha\beta\mu(P_n)}}.
\]
If \(\beta\mu\geqslant1\), this inequality follows from
\(\sqrt\mu\geqslant\beta^{-1/2}\).  Squaring completes the proof.
\end{proof}

Applying Lemma~\ref{lem:path-transfer} to
Proposition~\ref{prop:cubic-path}, with \(\alpha=2\) and \(\beta=6\), gives the
explicit bound for cubic graphs.

\begin{corollary}\label{cor:cubic-explicit}
Let \(G\) be a connected simple cubic graph of order \(n\).  Then
\begin{equation}
 \mu(G)\geqslant
 \frac{2\mu(P_n)}{\bigl(1+\sqrt{12\mu(P_n)}\bigr)^2}.
 \label{eq:cubic-explicit}
\end{equation}
\end{corollary}

\subsection{A sharp bound in terms of the minimum degree}
\label{sec:minimum-degree-bound}

The full cut kernel gives a bound in terms of the minimum degree.
Its leading constant is the one predicted by the Guiduli--Mohar
conjectures.

\begin{theorem}\label{thm:minimum-degree-bound}
Let \(G\) be a connected simple graph of order \(n\) and minimum
degree at least \(d\geqslant3\).  Then
\begin{equation}\label{eq:minimum-degree-explicit}
 \mu(G)\geqslant
 \frac{(d-1)\mu(P_n)}
 {(1+\sqrt{d(d+1)\mu(P_n)})^2}.
\end{equation}
For every fixed \(d\geqslant3\), this gives
\[
 \min_{\substack{|V(G)|=n\\G\text{ connected},\ \delta(G)\geqslant d}}
 \mu(G)=(1+o(1))\frac{(d-1)\pi^2}{n^2}.
\]
For each fixed odd \(d\geqslant3\), the same asymptotic equality holds
when the minimum is restricted to \(d\)-regular graphs and \(n\)
tends to infinity through even integers.
\end{theorem}

Throughout this subsection, \(G\) has minimum degree at least
\(d\geqslant3\). The cut-kernel representation of the energy and
Lemma~\ref{lem:interval-cut} apply under this hypothesis. We first
estimate the principal form beginning at a cut of size less than
\(d-1\).

For a cut of size \(k\), write \(r=d+1\) when \(k=1\), and
\(r=d\) otherwise.

\begin{lemma}\label{lem:general-low-cut}
Suppose that \(q_a=k\), where \(1\leqslant k\leqslant d-2\), and that
\(a+r\leqslant n\). Then the principal cut-kernel form on
\(a,\ldots,a+r-1\) satisfies
\[
 \sum_{t,u=0}^{r-1}K_{a+t,a+u}x_tx_u
 \geqslant\frac{d-1}{r}\left(\sum_{t=0}^{r-1}x_t\right)^2
 \qquad(\mathbf x\in[0,\infty)^r).
\]
For \(k\geqslant2\), the coefficient \(d-1\) can be increased by
a positive quantity depending only on \(d\).
\end{lemma}

\begin{proof}
For \(1\leqslant t\leqslant u\leqslant r-1\), put
\(A=S_a\), \(U=S_{a+t}\setminus S_a\),
\(V=S_{a+u}\setminus S_{a+t}\), and
\(R=V(G)\setminus S_{a+u}\).  Simplicity and the degree bound give
\[
 e(U,R)\geqslant dt-t(t-1)-t(u-t)-e(A,U).
\]
Since \(e(A,U)\leqslant k\), we obtain
\begin{equation}\label{eq:full-low-kernel}
 K_{a+t,a+u}=e(A,R)+e(U,R)
 \geqslant\max\{0,t(d-u+1)-k\}.
\end{equation}

When \(k=1\), the symmetric matrix on the right for
\(1\leqslant t,u\leqslant d\) is the increment kernel \(B\) of
\(K_{d+1}\) with its first-to-last edge removed.  It is positive
definite, and
\[
 B^{-1}\one=\frac1{d-1}(1,0,\ldots,0,1)^{\mathsf T}.
\]
Using \(K_{aa}=1\) and discarding its nonnegative mixed terms,
Cauchy--Schwarz for \(\operatorname{diag}(1,B)\) gives the required
coefficient, because \(1+\one^{\mathsf T}B^{-1}\one=(d+1)/(d-1)\).

Now suppose that \(2\leqslant k\leqslant d-2\).  Let \(A_{d,k}\)
be the \(d\times d\) symmetric matrix whose first diagonal entry is
\(k\), whose other entries in the first row and column are zero,
and whose remaining entries are
\[
 (A_{d,k})_{tu}=
 \max\{0,\min(t,u)(d-\max(t,u)+1)-k\},
 \qquad 1\leqslant t,u\leqslant d-1.
\]
For \(4\leqslant d\leqslant11\), it is checked directly that
\(A_{d,k}\) is positive definite and
\[
 \frac{d}{\one^{\mathsf T}A_{d,k}^{-1}\one}>d-1
 \qquad(2\leqslant k\leqslant d-2).
\]
Cauchy--Schwarz proves the assertion for these degrees.

For larger \(d\), we use the diagonal terms and write
\[
 F_d(k)=\frac1k+
 \sum_{t=1}^{d-1}\frac1{t(d-t+1)-k}.
\]
The denominators are positive, and \(F_d\) is convex on \([2,d-2]\),
so its maximum occurs at an endpoint. Since the terms
\(1/k+1/(d-k)\) agree at the endpoints and every remaining term
increases with \(k\), we have \(F_d(2)\leqslant F_d(d-2)\).  Direct calculation for \(12\leqslant d\leqslant16\)
gives \(F_d(d-2)<d/(d-1)\).  For \(d\geqslant17\), writing
\(H_j=\sum_{i=1}^j1/i\), we have
\[
 \begin{split}
 F_d(d-2)
 &=\frac12+\frac1{d-2}
   +\sum_{u=1}^{d-2}\frac1{u(d-1-u)+2}\\
 &\leqslant\frac12+\frac1{d-2}+\frac{2H_{d-2}}{d-1}<1.
 \end{split}
\]
The last expression decreases with \(d\), and at \(d=17\) it is
\(2829389/2882880<1\).  Weighted Cauchy--Schwarz therefore gives
\[
 \mathbf x^{\mathsf T}A_{d,k}\mathbf x
 \geqslant\frac{(\sum x_t)^2}{F_d(k)}
 >\frac{d-1}{d}\left(\sum x_t\right)^2.
\]
This proves both assertions.  For \(d=3\) there is no nonbridge case.
\end{proof}

\begin{proof}[Proof of Theorem~\ref{thm:minimum-degree-bound}]
Call an index low if \(q_i<d-1\).  For
\(2\leqslant t\leqslant d-1\), the interval-cut inequality gives
\(q_i+q_{i+t}\geqslant t(d-t+1)\geqslant2d-2\).  Two low indices
therefore cannot be at these distances.  Their maximal consecutive
clusters have length at most two, and the first indices of distinct
clusters differ by at least \(d\).  If \(q_a=1\), the same
inequality for \(t=1,d\) shows that every other low index is at
distance at least \(d+1\) from \(a\).

At every bridge index use the \(d+1\) increment indices beginning
there, and at the first index of each other low cluster use the
\(d\) increment indices beginning there.  These blocks are disjoint
and cover all low indices.  The endpoint bounds
\(q_i\geqslant i(d-i+1)\geqslant d\), for \(1\leqslant i\leqslant d\),
and their complementary versions ensure that the blocks fit inside
\(\{1,\ldots,n-1\}\).

Given an ordered Fiedler vector \(f\), interpolate affinely across
these blocks to obtain \(g\), leaving all other increments unchanged.
In the full cut-kernel identity, retain the principal forms on the
chosen blocks and the other diagonal entries.  All omitted terms
are nonnegative.  Lemma~\ref{lem:general-low-cut} and the inequality
\(q_i\geqslant d-1\) elsewhere give
\[
 \cE_G(f)\geqslant(d-1)\cE_{P_n}(g),\qquad
 \|f-g\|_2^2\leqslant\frac{d(d+1)}{d-1}\cE_G(f).
\]
For the interpolation error, a block of length \(r\leqslant d+1\)
contributes at most \((r-1)(f_{a+r}-f_a)^2\), while its retained
energy is at least \((d-1)(f_{a+r}-f_a)^2/r\).
Lemma~\ref{lem:path-transfer} proves \eqref{eq:minimum-degree-explicit}.

The matching upper bound is given by chains of copies of \(K_{d+1}-e\)
joined by bridges, with bounded end graphs ensuring minimum degree
at least \(d\).  Such graphs exist at every sufficiently large
order: one may use complete end graphs of orders \(d+1\) and
\(d+1+r\), where \(0\leqslant r\leqslant d\) is chosen to match
the order modulo \(d+1\).  Their asymptotic algebraic connectivity
is \((d-1)\pi^2/n^2\), by
\cite[Corollary~3.10]{AbdiGhorbaniDiameter}.  For odd \(d\), regular
end graphs can be chosen at every sufficiently large even order, as
in \cite[Table~1]{AbdiGhorbaniDiameter}.
\end{proof}

\subsection{The sharp constant for every fixed even degree}

In an even-regular graph, every cut has even size. The smallest possible cuts
therefore have size two, and the sharp coefficient is determined by
the intervals between them. Together with the minimum-degree bound,
the resulting estimate gives the sharp formula for every regular degree.

\begin{theorem}\label{thm:fixed-even-sharp}\label{thm:fixed-regular-sharp}
For every fixed integer \(d\geqslant3\),
\[
 \min_{\substack{|V(G)|=n\\G\text{ connected and }d\text{-regular}}}
 \mu(G)=(1+o(1))\frac{c_d\pi^2}{n^2},
\]
where \(n\) tends to infinity through orders for which the class is
nonempty.
\end{theorem}

Our aim is the comparison in
Proposition~\ref{prop:even-strict-compression}: the graph energy
dominates the Dirichlet energy of a piecewise affine function on an
interval of length \(n+O_d(1)\), with a strict improvement away from
copies of \(K_{d+1}-P_3\). We first estimate the intervals adjacent
to cuts of size two, then the increments between them. The strict
part of the comparison will also prove stability in
Section~\ref{sec:fixed-even-stability}.

For the even-degree argument, let \(d\geqslant6\) be even and let
\(G=(V,E)\) be a connected \(d\)-regular graph of order \(n\). Write
\(c=2(d-2)\) and \(\gamma=(d-2)/(2(d+2))\). For an ordered vector
$f_1\leq\cdots\leq f_n$, put $y_i=f_{i+1}-f_i$ and
$q_i=|\partial S_i|$, so that
\[
 \cE_G(f)=\mathbf y^{\mathsf T}K\mathbf y,\qquad
 K_{ij}=e(S_{\min(i,j)},V\setminus S_{\max(i,j)}).
\]
Every $q_i$ is a positive even integer. The degree sum formula on
an intervening set of $t$ vertices gives
\[
\begin{aligned}
 q_i+q_{i+t}&\geq t(d+1-t),\\
 K_{ij}&\geq\frac{q_i+q_j-d|i-j|}{2},\qquad
 K_{i,i+1}=\frac{q_i+q_{i+1}-d}{2}.
\end{aligned}
\]
We also use the full low-cut bound of
Lemma~\ref{lem:general-low-cut}. In particular, size-two cuts are
separated by at least $d+1$ indices and do not occur among the
first or last $d$ increments.

We use resistance bounds both for electrical networks and for their
associated quadratic forms. A resistance bound \(R>0\) for a
quadratic form \(Q\) means that
\(Q(\mathbf y)\geqslant R^{-1}(\sum_i y_i)^2\) for
nonnegative \(\mathbf y\). For an interval of coordinate length
\(\ell\), its resistance deficit is \(\ell/c-R\), and its
resistance excess is \(R-\ell/c\).

We also decompose the energy according to the values of \(f\).
For an edge \(e=uv\), orient its endpoints so that
\(f(u)\leqslant f(v)\), and put \(D_e=f(v)-f(u)\).
For real numbers \(s\leqslant t\), let
\(d_e(s,t)=|[f(u),f(v)]\cap[s,t]|\), where the bars denote
interval length, and define
\[
 \mathcal E_{[s,t]}(f)=\sum_{e\in E(G)}D_e d_e(s,t).
\]
The clipped function
\(v\mapsto\min\{\max\{f(v),s\},t\}\) has edge drops
\(d_e(s,t)\). These energies add over adjacent potential intervals.
Splitting an
edge at the levels \(s,t\) assigns resistance \(d_e(s,t)/D_e\)
to the segment between them when \(D_e>0\); that segment has
energy \(D_e d_e(s,t)\).

For each size-two cut $a$, let $z_a=(f_a+f_{a+1})/2$.
A shore is concentrated if the two crossing edges have a common
endpoint on that shore, and split otherwise. Set $\delta_a=\gamma$
if the left shore is concentrated, $\delta_a=-\gamma$ if the right
shore is concentrated, and $\delta_a=0$ otherwise. Define
\[
 h_a=\frac14\sum_{uv\in\partial S_a}
       \frac{z_a-f_u}{f_v-f_u},\qquad
 x_a=a+\frac12+\delta_a+ch_a,
\]
where $u<v$ and a ratio with zero denominator is interpreted as
$1/2$. For consecutive size-two cuts $a<b$, write
$T=S_b\setminus S_a$ and let $z$ count the edges bypassing $T$.
After subdivision at the cut levels, let $R_{\rm int}$ be the
minimum interior flow energy when current $1/2$ passes through
each boundary vertex.

\Needspace{18\baselineskip}
\begin{lemma}\label{lem:even-level-split}
Let $d\geq6$ be even, let $G$ be a connected simple $d$-regular
graph of order $n$, and let $f$ be a real-valued function on its
vertices, ordered so that $f_1\leq\cdots\leq f_n$.
Then the edges admit a simultaneous subdivision at the levels $z_a$
in which segment resistances are the fractions of the total voltage
drop. The subdivision preserves the total energy, and every new
boundary vertex at cut $a$ has value $z_a$. The coordinates $x_a$
are increasing and differ from their ranks by a bounded amount
depending only on $d$.

For consecutive size-two cuts $a<b$, a flow carrying current $1/2$
through each boundary vertex has total boundary and bypass energy
$(2-z)/4+h_b-h_a$. Thus the resulting unit flow has energy
\[
 R_{\rm int}+\frac{2-z}{4}+h_b-h_a,
\]
which is an upper bound for the effective resistance between the
two cut levels.
\end{lemma}

\begin{proof}
For an edge with voltage drop $D>0$, a segment of drop $b$ has
resistance $b/D$ and energy $Db$. These resistances are nonnegative
and sum to one, and the energies sum to $D^2$. Zero-drop edges also
have zero energy with the stated convention. The boundary formula
follows by summing one quarter of the segment resistances: an entering
edge contributes $(1-r_e(a))/4$, a leaving edge contributes $r_e(b)/4$,
and a bypass contributes $(r_e(b)-r_e(a))/4$, where
$r_e(a)=(z_a-f_u)/(f_v-f_u)$, with the stated convention when
$f_u=f_v$.

Simplicity prevents both shores of a cut from being concentrated.
Moreover, two size-two cuts have distance at least $d+1$, since an
intervening set of $t$ vertices gives $4\geq t(d+1-t)$, which is
impossible for $1\leq t\leq d$. Since $0\leq h_a\leq1/2$ and
$|\delta_a|\leq\gamma<1/2$, the possible coordinate displacement
varies by less than $d-1$. This proves the assertions about $x_a$.
\end{proof}

For a size-two cut $a$ and $j=a+r+1$, call the potential interval
$[z_a,f_j]$ associated with $v_{a+1},\ldots,v_{a+r}$ a boundary
interval, and give its right endpoint coordinate $y_j=j+c/4$.
For $2\leq r\leq d$ and $1\leq r\leq d-1$, respectively, set
\[
 R_S(d,r)=\frac14+\frac{d-r+3}{2\{d(d-r+1)-2\}},\qquad
 R_C(d,r)=\frac14+\frac{d-r+2}{(d-1)(d-r+1)-2}.
\]
We also set
\[
 R_C(d,d)=\frac14+\frac{2(d+1)}{(d-2)(d+2)},\qquad
 B_d=\frac{d^2-2d-4}{4(d-2)(d+2)},\qquad C_d=2B_d.
\]

\Needspace{9\baselineskip}
\begin{lemma}\label{lem:even-collar}
Let $a$ be a size-two cut. Then the boundary interval $[z_a,f_j]$
has resistance bound $R_S(d,r)+1/4-h_a$ when the right shore is
split, and $R_C(d,r)+1/4-h_a$ when it is concentrated, in the
corresponding ranges of $r$. Its energy is at least the squared
difference between its endpoint values divided by this bound.
For a reversed boundary interval ending at a size-two cut $b$,
the respective bounds are $R_S(d,r)-1/4+h_b$ and
$R_C(d,r)-1/4+h_b$.

For $r=d$, the resistance deficit for either shore type is at least
$B_d>0$. Thus the resistance is at most the coordinate length
divided by $c$, minus $B_d$, and two such boundary intervals have
total resistance deficit at least $C_d$.
\end{lemma}

\begin{proof}
Introduce current $1/2$ at each left boundary vertex. If an incoming
edge bypasses $v_{a+1},\ldots,v_{a+r}$, send its current directly
to the level $f_j$. The incoming segments have total flow energy at
most $1/2-h_a$. Replace each remaining outgoing segment by an edge
of resistance one to ground, which can only increase resistance.
Let $l_i$ count the incoming edges ending at $v_{a+i}$, and write
$l=(l_1,\ldots,l_r)^{\mathsf T}$. The grounded Laplacian has
diagonal $d-l_i$ and off-diagonal $-1$ on internal edges, and its
source vector is $l/2$.

For nonnegative sources the grounded potentials are nonnegative.
Adding a missing internal edge and removing one ground edge at each
endpoint reduces the energy of a nonnegative vector by $2v_iv_j$,
so it can only increase the minimum flow energy for these sources. Thus the largest resistance is attained by completing the induced
graph to $K_r$, except that when $r=d$ and both incoming edges
meet one vertex it is $K_d$ with one edge missing at that vertex.
The completion is always possible: vertices not receiving both
incoming edges retain a ground edge until all their internal edges
have been added. If only one edge of a split shore enters this vertex set, adding
the second at a distinct vertex and removing a ground edge there
also increases resistance. In the concentrated case, either both
incoming edges enter the set or both bypass it.

Writing $I_r$ and $J_r$ for the identity and all-ones matrices,
respectively, and $e_1,\ldots,e_r$ for the standard basis vectors,
the grounded Laplacian of the completed graph is
\[
 A=(d+1)I_r-J_r-\operatorname{diag}(l),
\]
where $l=e_1+e_2$ or $l=2e_1$, except that $A_{12}$ and $A_{21}$
are increased by one in the concentrated case with $r=d$.
The displayed formulas follow from $1/4+l^{\mathsf T}A^{-1}l/4$
and are checked directly from the Laplacian equations; applying
Cauchy--Schwarz to this flow and the actual potentials then gives
the energy inequality.

Finally, let $R_\ast$ be $R_S(d,r)$ or $R_C(d,r)$ according
to the shore type. Then
\[
 \frac{y_j-x_a}{c}-(R_\ast+1/4-h_a)
   =\frac{r+1/2-\delta_a}{c}-R_\ast.
\]
For a split right shore $\delta_a\leq\gamma$, whereas for a
concentrated right shore $\delta_a=-\gamma$. With $r=d$,
\[
 R_S(d,d)=\frac14+\frac3{2(d-2)},\qquad
 R_C(d,d)-R_S(d,d)=\frac1{2(d+2)}.
\]
The choice $\gamma=c/(4(d+2))$ makes the two lower bounds on the
resistance deficit equal to $B_d$. Reversal proves the corresponding right-hand
statement.
\end{proof}

We next estimate the increments between two boundary intervals.
Cuts of sizes $4,6,\ldots,d-2$ require the full cut kernel;
cuts of size at least $d$ admit a tridiagonal estimate. The following
three lemmas give the bounds for these two cases and for a truncated
interval at the end.

For an increment index $a$, write $B=\{a,\ldots,a+d-1\}$, and set
$P_d=(d-4)/(2d(d-2))$.

\begin{lemma}\label{lem:even-low-block-reserve}
Let $d\geqslant6$ be even, and let $G$ be a simple $d$-regular graph.
Suppose that $q_a=k\in\{4,6,\ldots,d-2\}$ and that
$q_i\geqslant4$ for every $i\in B$.
Then there is a number $\rho_d$, depending only on $d$, such that
\[
 \rho_d<\frac d c-P_d,\qquad
 \rho_d<\frac d c-2P_d+\frac{2}{(d-2)(d+2)},
\]
and
\[
 \mathbf y_B^{\mathsf T}K_{BB}\mathbf y_B
 \geqslant \rho_d^{-1}
 \left(\sum_{i\in B}y_i\right)^2
 \qquad(\mathbf y_B\geqslant0).
\]
\end{lemma}

\begin{proof}
Set $p_0=k$ and
$p_t=\max\{4,t(d+1-t)-k\}$ for $1\leqslant t\leqslant d-1$.
The full low-cut estimate and the degree identity on an intervening
interval show that the kernel on $B$ dominates entrywise the
symmetric matrix
\[
\begin{aligned}
 A_{tt}&=p_t,\\
 A_{tu}&=\max\Bigl\{0,\,
 \min(t,u)(d+1-\max(t,u))-k,\\
 &\hspace{4.5em}\frac{p_t+p_u-d|t-u|}{2}\Bigr\}
 \qquad(t\ne u).
\end{aligned}
\]
For $6\leqslant d\leqslant38$, direct calculation gives
$A\succ0$ and
\[
 \one^{\mathsf T}A^{-1}\one
 \leqslant \min\left\{\frac d c-P_d,\,
       \frac d c-2P_d+\frac{2}{(d-2)(d+2)}\right\}
       -\frac2{1335}.
\]
For $d\geqslant40$ the diagonal entries alone suffice.
For $4\leqslant k\leqslant d-4$, convexity gives
$1/k+1/(d-k)\leqslant1/4+1/(d-4)$.
The remaining value $k=d-2$ satisfies the same bound after replacing
$d-k$ by $\max\{4,d-k\}$.
Furthermore, for $2\leqslant t\leqslant d-1$,
\[
 p_t\geqslant(t-1)(d-t)+2\geqslant(t-1)(d-t).
\]
Writing $H_m=\sum_{j=1}^m1/j$, we obtain
\[
\begin{aligned}
 \sum_{t=0}^{d-1}\frac1{p_t}
 &\leqslant U_d:=\frac14+\frac1{d-4}
                   +\frac{2H_{d-2}}{d-1},\\
 U_d&<\frac12<\min\left\{\frac d c-P_d,\,
       \frac d c-2P_d+\frac{2}{(d-2)(d+2)}\right\}.
\end{aligned}
\]
The strict inequality holds at $d=40$ by direct calculation and
persists thereafter, since $H_m/(m+1)$ is decreasing for $m\geqslant2$.
Cauchy--Schwarz for the positive definite matrix $A$, or for the
diagonal matrix in the latter case, proves the assertion after taking
the maximum over the finitely many possible $k$.
\end{proof}

Write $T_d=1/2-1/(d-2)$.

\begin{lemma}\label{lem:even-truncated-low-block}
Let $B$ be a block satisfying the hypotheses of
Lemma~\ref{lem:even-low-block-reserve}, and let $1\leqslant r<d$.
Then the principal cut-kernel form on the first $r$ increment
indices of $B$ has a resistance bound
\[
 R_r\leqslant\frac r c+T_d.
\]
\end{lemma}

\begin{proof}
Use the leading $r$-by-$r$ submatrix of the matrix in
Lemma~\ref{lem:even-low-block-reserve}.
For even $6\leqslant d\leqslant38$, the stated bound follows by
direct calculation for $k=4,6,\ldots,d-2$ and $1\leqslant r<d$.
For $d\geqslant40$, the reciprocal sum of the first $r$ diagonal
entries is at most $U_d<1/2$. If $r\geqslant2$, subtracting
$r/c\geqslant2/c=1/(d-2)$ proves the claim.
For $r=1$, use $R_1=1/k\leqslant1/4$ directly.
\end{proof}

\begin{lemma}\label{lem:even-high-run}
Let an increment interval have length $r$, and suppose that
$q_i\geqslant d$ throughout it. Then its principal cut-kernel form
has a resistance bound
\[
 R_r\leqslant\frac r c+P_d.
\]
The bounds supplied by the tridiagonal matrix with diagonal $d$
and off-diagonal $d/2$ are
\[
 R_{2m-1}=\frac m d,\qquad
 R_{2m}=\frac{2m(m+1)}{d(2m+1)}.
\]
In particular, $R_r\leqslant r/c$ for $r\geqslant d-1$.
If the interval follows a full low-cut block, it can be partitioned
into intervals of length bounded in terms of $d$ so that, after
combining the first with that block, every resulting interval has
a resistance bound at most its length divided by $c$.
\end{lemma}

\begin{proof}
The adjacent-cut identity gives the stated tridiagonal lower matrix,
and solving the corresponding second-order recurrence gives the
displayed sums of the entries of its inverse. Thus
$R_r\leqslant(r+1)/(2d)$ in both parities, with equality for odd
$r$, and hence
\[
 R_r-\frac r c
 \leqslant\frac{d-2-2r}{2d(d-2)}\leqslant P_d.
\]
To obtain bounded pieces, write the interval length as $s+md$,
where $0\leqslant s<d$. Combine its first $s$ increments with the
preceding low-cut block, and partition the remainder into blocks of
length $d$. The strict bound in
Lemma~\ref{lem:even-low-block-reserve} gives a resistance deficit
larger than the excess of the first interval. Combining the two inequalities by
Cauchy--Schwarz gives coefficient at least $c$. The combined block has length at most $2d-1$.
\end{proof}

The constants above satisfy $C_d=T_d+2/((d-2)(d+2))$.

\begin{lemma}\label{lem:even-boundary-charges}
An initial high-cut block of length $0\leqslant s<d$, followed
immediately by a truncated low-cut block of length $1\leqslant r<d$,
has total resistance excess strictly less than $C_d$.
\end{lemma}

\begin{proof}
Let $k$ be the size of the first cut in the truncated block.
The two blocks may be estimated separately except when
$(d,s,k,r)=(8,1,4,2)$.
The high-cut resistance excess is at most $P_d$.
For even $6\leqslant d\leqslant38$, direct calculation for the
leading low-cut matrices gives
\[
 R_r-r/c+P_d<C_d
\]
except when $(d,k,r)=(8,4,2)$.
For $d\geqslant40$ and $r\geqslant3$, use $R_r<U_d<1/2$ and
\[
 C_d-P_d+\frac3c
 =\frac12+\frac{4(d+1)}{d(d^2-4)}>\frac12.
\]
For $r=1$, use $R_1\leqslant1/4$.
For $r=2$, the first two diagonal entries give
$R_2\leqslant1/4+1/(d-4)$, and for $d\geqslant10$,
\[
 C_d-\left(\frac14+\frac1{d-4}-\frac1{d-2}+P_d\right)
 =\frac{d^4-10d^3+16d^2-64}
        {4d(d-4)(d-2)(d+2)}>0.
\]
For the exceptional low-cut prefix at $d=8$, every high-cut
block with $s\geqslant2$ has nonpositive resistance excess, so only $s=1$
requires a combined estimate. Its kernel dominates
\[
 \begin{pmatrix}8&2&0\\2&4&0\\0&0&4\end{pmatrix},
\]
whose inverse has entries summing to $15/28$. Its resistance excess is
$15/28-3/12=2/7<C_8=11/30$.
The case $s=0$ is immediate from the truncated-block bounds.
\end{proof}

We now combine the boundary and interior estimates. For a long
interval between cuts of size two, they give a uniform improvement
over the coefficient $c$.

\begin{lemma}\label{lem:even-central-allocation}
Let $a<b$ be consecutive cuts of size two with $b-a\geqslant2d+1$.
Then there is a constant $\nu_d>0$, depending only on $d$, and a
continuous nondecreasing piecewise affine interpolation $g$
between the two cut levels such that
\[
 (c+\nu_d)\int_{x_a}^{x_b}|g'|^2
 \leqslant\mathcal E_{[z_a,z_b]}(f).
\]
Every affine piece spans at most $2d$ rank increments, its coordinate
length is bounded above and below by positive constants depending
only on $d$, and every knot has bounded displacement from its rank.
\end{lemma}

\begin{proof}
Apply the two boundary-interval estimates on $d$ vertices, with
coordinate $y_j=j+c/4$ at vertex rank $j$.
The two boundary intervals end at ranks $j_L=a+d+1$ and $j_R=b-d$.
Their total resistance deficit is at least $C_d$, and the stated
coordinate corrections give a deficit of at least $C_d/2$ on each
one. All cuts in the remaining interval have size at least four.

Indices with $q_i<d$ form clusters of one or two consecutive
indices, and different cluster starts have distance at least $d$.
Indeed, two such indices at distance $2\leqslant t\leqslant d-1$
would contradict $q_i+q_{i+t}\geqslant t(d+1-t)\geqslant2d-2$.
Starting at the first such cut, take full low-cut blocks of length
$d$ as long as possible, leaving at most one final truncated block.
Between full low-cut blocks, write each intervening high-cut run
as $s+md$, with $0\leqslant s<d$. Keep the first $s$ increments
with the preceding low-cut block and use blocks of length $d$
for the rest. Partition the initial high-cut run into $d$-blocks
followed by its remainder of length less than $d$.

Each full low-cut block, together with the following high-cut
remainder, has positive resistance deficit: the deficit of the block
exceeds $P_d$, which bounds the excess of the remainder. Every
high-cut $d$-block also has positive deficit. The initial high-cut
remainder has excess at most $P_d$, and the final truncated block
has excess at most $T_d$.
If a full low-cut block occurs, combine it with these two remainders
and the boundary intervals. The stronger bound in
Lemma~\ref{lem:even-low-block-reserve} gives
\[
 C_d+\bigl(d/c-\rho_d\bigr)-2P_d-T_d>0.
\]
Thus the combined intervals have positive deficit when the high-cut
remainder following this full low-cut block is included. If no full
low-cut block occurs, the initial remainder and the truncated block
are adjacent, and Lemma~\ref{lem:even-boundary-charges} gives
positive deficit after adding the boundary intervals. If there is
no truncated block, their deficit exceeds the excess of the initial
high-cut remainder.

We choose interpolation coordinates as follows.
Set the coordinate length equal to the rank length on
every interior $d$-block, or on a full low-cut block together with
its high-cut remainder. On the bounded number of remaining pieces,
including the boundary intervals and, if needed, the designated first full
low-cut block and its remainder, divide the available coordinate
length in proportion to their resistance bounds.
The strict inequalities give coefficient $c+\nu_d$ on these
intervals for some fixed $\nu_d>0$.
The other interior pieces already permit such a coefficient, after
decreasing $\nu_d$ if necessary. A combined low-cut block and its
high-cut remainder may be split into its two affine pieces by the
same proportional division.

Only a bounded number of groups have total coordinate lengths
different from their rank lengths. Splitting any other combined
group changes its internal knot by a bounded amount and preserves
its total length. The cumulative coordinate displacement is
therefore bounded.
All resistance bounds are positive and bounded above and below
by constants depending only on $d$, so every coordinate length is
positive and uniformly bounded.

Finally, split the edges at the interpolation levels as in
Lemma~\ref{lem:even-level-split}. On a potential interval $[s,t]$,
the clipped edge drop $d_e=d_e(s,t)$ satisfies
$d_e^2\leqslant D_ed_e$. Thus the cut-kernel estimate for the
clipped vector and the flow estimate on each boundary interval are
bounded by the energy on their respective potential intervals.
These energies add to the original energy, proving the assertion.
\end{proof}

\Needspace{14\baselineskip}
\begin{lemma}\label{lem:even-short-cores}
Let \(d\geqslant6\) be even, and let \(a<b\) be consecutive
indices with \(q_a=q_b=2\) and \(d+1\leqslant b-a\leqslant2d+1\).
Then the potential interval $[z_a,z_b]$ has a resistance bound
\[
 R\leqslant\frac{x_b-x_a}{c}.
\]
Moreover, its resistance deficit is bounded below by a positive
constant depending only on \(d\), except when \(b-a=d+1\),
\(G[S_b\setminus S_a]=K_{d+1}-P_3\), its inner shores are
concentrated and split, and \(\delta_b=\delta_a\).
\end{lemma}

\begin{proof}
Put \(t=b-a\), and write \(S\) for a split inner shore and \(C\) for a concentrated
inner shore. A one-sided boundary interval on \(r\) vertices has the following
uncorrected resistance bounds:
\[
 S_d(r)=\frac14+
 \frac{d-r+3}{2\{d(d-r+1)-2\}},\qquad
 C_d(r)=\frac14+
 \frac{d-r+2}{(d-1)(d-r+1)-2}.
\]
The first is valid for \(r\leqslant d\), and the second for
\(r\leqslant d-1\). At \(r=d\), the concentrated bound is
\[
 C_d^*=\frac14+\frac{2(d+1)}{(d-2)(d+2)}\leqslant C_d(d).
\]
These are the bounds of Lemma~\ref{lem:even-collar}.
Take \(r_L+r_R=t-1\), so that the two rank endpoints coincide.
The sum of their uncorrected bounds, plus \(h_b-h_a\), is a
resistance bound for the whole potential interval. This remains valid for bypass
edges, since the two estimates use disjoint potential intervals.

For shore types \(SS,SC,CC\), respectively, the least possible
value of \(\delta_b-\delta_a\) is
\(-2\gamma,0,2\gamma\). Choose the boundary interval lengths as equally
as possible, assigning the smaller length to \(C\) in the mixed
case. For all \(d+2\leqslant t\leqslant2d+1\), their uncorrected
sum is strictly below
\[
 \frac{t-2\gamma}{c},\qquad
 \frac t c,\qquad
 \frac{t+2\gamma}{c},
\]
respectively, except for the case \(d=6,t=8,CC\), which is
handled below. The \(SS\) inequality also holds for \(t=d+1\).

To verify these inequalities, observe that both \(S_d(r)\) and \(C_d(r)\)
are increasing convex functions of \(r\) on the relevant range.
For each parity of \(t\), the difference between its displayed
linear bound and the sum of the two boundary interval bounds is therefore
concave. It suffices to check the endpoints: for odd \(t\),
the lower endpoint is \(d+1\) for \(SS\), and \(d+3\) for
the other two types, while the upper endpoint is \(2d-1\).
For even \(t\), the endpoints are \(d+2\) and \(2d\).
The largest value \(t=2d+1\) uses \(C_d^*\) at each concentrated
shore. These endpoint inequalities are checked directly after
clearing positive denominators. In the \(CC\) even case use
\(d\geqslant8\); when \(d=6\), its lower endpoint is instead
\(t=10\). For example, the \(SS\) difference at \(t=d+1\) is
\[
 \frac{8(d+1)}{(d-2)(d+2)(d^2+2d-4)}>0,
\]
and all three differences at \(t=2d+1\) equal
\[
 \frac{d^2-2d-4}{2(d-2)(d+2)}>0.
\]

It remains to examine the shortest intervening subgraphs. Put
\(H=G[S_b\setminus S_a]\). A bypass edge forces both inner
shores to be split: otherwise both edges of a concentrated cut
would bypass \(H\), disconnecting it. Thus a concentrated shore
excludes bypass edges. For \(t=d+1\), the complement of \(H\)
has exactly two edges. Two concentrated shores are impossible; a mixed shore
pair forces the complement to be \(P_3\). Its concentrated terminal
is the centre of this path and its split terminals are the leaves.
The uncorrected resistance is
\[
 \frac12+\frac{3}{2(d-2)}=\frac{d+1}{c}.
\]
The correction \(\delta_b-\delta_a\geqslant0\) is zero precisely
in the case stated in the lemma. Otherwise it is at least
\(\gamma>0\).

Consider now \(t=d+2\) with two concentrated shores. If the terminals coincide, the interior
source vanishes and the uncorrected resistance is \(1/2\).
If they are distinct, the complement degree sequence is
\((3,3,1^d)\). According as its centres are nonadjacent or adjacent,
it is two three-leaf stars together with a matching, or two adjacent
centres with two leaves each together with a matching. The terminal
resistances in \(H\) are, respectively,
\[
 \frac{2(d+1)}{(d+2)(d-2)},\qquad
 \frac{2(d+1)}{d^2-d-4}.
\]
These values are checked directly from the Laplacian equations. The latter is larger, and
\[
\begin{aligned}
 \frac{d+2+2\gamma}{c}
 -\left(\frac12+\frac{2(d+1)}{d^2-d-4}\right)
 &=\frac{d^3-3d^2-10d-8}
 {2(d-2)(d+2)(d^2-d-4)}\\
 &>0\qquad(d\geqslant6).
\end{aligned}
\]
This includes \(d=6,t=8,CC\). Since \(d\) is fixed and there are only finitely many
possible boundary-interval lengths and shore types, the positive
resistance deficits have a common lower bound depending only on
\(d\).
\end{proof}

For an increasingly ordered vector $f$, let $F$ take value $f_i$
on the unit interval centred at $i+c/4$. On any interpolation
interval, extend or restrict $F$ by its endpoint values as needed.
Call the interval between consecutive size-two cuts $a<b$ canonical
when
\[
 b-a=d+1,\qquad G[S_b\setminus S_a]=K_{d+1}-P_3,
 \qquad\delta_b=\delta_a,
\]
and its inner shores are concentrated and split. The concentrated
terminal is the centre of the missing path and the split terminals
are its leaves.

\begin{proposition}
\label{prop:even-strict-compression}
Let $G$ be a connected simple $d$-regular graph of order $n$, and
let $f$ be an increasingly ordered vector on its vertices.
Then there are constants $C_d<\infty$ and $\nu_d>0$, depending
only on $d$, a continuous nondecreasing piecewise affine function
$g$ on an interval $[A,B]$ of length $n+O_d(1)$, and a set
$\mathcal B\subseteq[A,B]$ consisting of affine pieces of the
interpolation such that
\begin{equation}\label{eq:even-strict-energy}
 \cE_G(f)\geq c\int_A^B|g'|^2
           +\nu_d\int_{\mathcal B}|g'|^2.
\end{equation}
Moreover,
\begin{equation}\label{eq:even-step-comparison}
 \int_A^B|g-F|^2\leq C_d\sum_{i=1}^{n-1}(f_{i+1}-f_i)^2
                       \leq C_d\cE_G(f).
\end{equation}
Every affine piece spans a bounded number of consecutive vertex
indices, its coordinate length is bounded above and below by positive
constants depending only on $d$, and all knots have bounded
displacement from their ranks. Every piece belongs to $\mathcal B$
except for a bounded number of end pieces and the canonical intervals.
\end{proposition}

\begin{proof}
At each size-two cut use the point $(x_a,z_a)$ of
Lemma~\ref{lem:even-level-split}. For consecutive size-two cuts,
use Lemma~\ref{lem:even-short-cores} if their distance is at most
$2d+1$, and Lemma~\ref{lem:even-central-allocation} otherwise.
The short interval is represented by one affine piece; its
resistance bound proves the desired energy estimate. The long interval is partitioned into boundary and interior
intervals of bounded length. Both arguments
allow bypass edges, including an edge that bypasses many such
intervals. Their strict estimates apply uniformly except on the
canonical intervals defined above.

We specify the ends. Before the first size-two cut $a$, take a
reversed boundary interval on $d$ vertices ending at rank
$j=a-d\geq1$. Partition the increments $1,\ldots,j-1$ as in the proof of
Lemma~\ref{lem:even-central-allocation}: full low-cut blocks and
their following high-cut remainders have their rank lengths,
and all remaining long high-cut runs are partitioned into $d$-blocks.
Only the initial high-cut remainder and the final truncated low-cut block
may have positive resistance excess. Give each of these pieces
length equal to the larger of its rank length and $c$ times its
resistance bound. The total additional length $D_L$ is at most
$c(P_d+T_d)$, a constant depending only on $d$. Start at coordinate
$A=1+c/4-D_L$ with value $f_1$. Then the right endpoint of this
partition is $j+c/4$, the endpoint coordinate of the boundary interval.
Do the same after the last size-two cut, ending at
$B=n+c/4+D_R$ with $D_R\leq c(P_d+T_d)$.
If there is no size-two cut, use the same partition on all $n-1$
increments and put its bounded extra length at an outer endpoint.

These constructions have positive bounded piece lengths. The full
low-cut blocks, high-cut $d$-blocks, and boundary intervals at the
ends all have positive resistance deficit. The at most four end
remainders have coefficient at least $c$. The finitely many local strict estimates,
together with the strictly positive bounds in
Lemma~\ref{lem:even-central-allocation}, have a common positive
coefficient increment $\nu_d$.

The energy decomposition allows us to sum these estimates. Split each
original edge at every interpolation level, as in
Lemma~\ref{lem:even-level-split}. If the edge has total drop $D_e$
and drop $b_e$ across one potential interval, its contribution to
the energy on that interval is $D_eb_e$.
The corresponding energy in the clipped vector, with unit edge
resistance, is
$b_e^2\leq D_eb_e$. Thus both the cut-kernel estimates and the
flow estimates use at most the energy of their respective potential
intervals, whose energies sum to the original energy. This proves
\eqref{eq:even-strict-energy}.

The coordinate adjustments affect a bounded number of pieces
by a bounded amount, or preserve the total length of a bounded
combined block. Consequently every knot stays at bounded distance
from its associated rank. Its value is a vertex value or the mean
of two consecutive vertex values. On each affine piece, comparison
with the adjacent values of the monotone step function and
Cauchy--Schwarz over a bounded number of increments bound its
squared error by a constant times the sum of squared nearby
increments. Each increment occurs in a bounded number of these bounds,
since the rank spans are ordered and disjoint and the coordinate
displacements are bounded. This proves
\eqref{eq:even-step-comparison}.
\end{proof}

\begin{proof}[Proof of Theorem~\ref{thm:fixed-even-sharp}]
For odd degrees the assertion follows from
Theorem~\ref{thm:minimum-degree-bound}; for degree four it follows
from \eqref{eq:quartic-known}. We prove the remaining even cases.
The matching upper bound is supplied by the regular chain families
in \cite[Theorem~1.8(ii) and Table~1]{AbdiGhorbaniDiameter}.
It suffices to prove the lower bound along any sequence with
$n^2\mu(G)$ bounded. Order a unit, mean-zero Fiedler vector
increasingly and put $E=\cE_G(f)=\mu(G)=O_d(n^{-2})$.
Use the interpolation of
Proposition~\ref{prop:even-strict-compression}.

Since $\sum_i(f_{i+1}-f_i)^2\leq E$ and $\sum_i f_i=0$,
\[
 \max_i|f_i|\leq\sqrt{nE}=O_d(n^{-1/2}).
\]
Thus the step function $F$ in that proposition satisfies
\[
 \int_A^B F^2=1+O_d(n^{-1}),\qquad
 \int_A^B F=O_d(n^{-1/2}),\qquad
 \|g-F\|_2=O_d(n^{-1}).
\]
Writing $\overline g=(B-A)^{-1}\int_A^B g$ for the mean, we obtain
\[
 \int_A^B|g-\overline g|^2=1+o(1).
\]
The Neumann Poincar\'e inequality on $[A,B]$ therefore gives
\[
 E\geq c\int_A^B|g'|^2
 \geq\frac{c\pi^2}{(B-A)^2}
       \int_A^B|g-\overline g|^2
 =(2(d-2)-o(1))\frac{\pi^2}{n^2},
\]
as required.
\end{proof}

\subsection{Cubic sharpness and stability}\label{sec:cubic-sharpness}

The graph \(L_3=K_4-e\) is called a diamond. The cubic chains
described in the introduction attain the leading constant in
Corollary~\ref{cor:cubic-explicit} for every sufficiently large even
order. We now describe their end graphs explicitly. Let \(C_5\) be obtained by
subdividing one edge of \(K_4\), and distinguish the new vertex
\(\xi_5\).  Let \(C_7\) have vertex set \(\{0,1,\ldots,6\}\),
distinguished vertex \(\xi_7=0\), and edge set
\[
 \{01,02,12,13,24,35,46,36,45,56\}.
\]
In each \(C_r\), the distinguished vertex has degree two and all other
vertices have degree three.

Take \(m\) disjoint diamonds, and label the terminals of the \(i\)th diamond
by \(a_i,b_i\).  Add the edges \(b_ia_{i+1}\) for \(1\leqslant i<m\), and attach
disjoint copies of \(C_r\) and \(C_s\) at the two ends by joining their
distinguished vertices to \(a_1\) and \(b_m\), respectively.  Denote the
resulting diamond-chain graph by \(N_m^{r,s}\).  It is connected, simple,
and cubic, and has \(4m+r+s\) vertices.  The choices \((r,s)=(5,5)\) and \((5,7)\) give the cubic minimisers
of \cite{BrandGuiduliImrich} for sufficiently large even orders.  Thus
\[
 X_n=
 \begin{cases}
 N_{(n-10)/4}^{5,5},&n\equiv2\pmod4,\\
 N_{(n-12)/4}^{5,7},&n\equiv0\pmod4.
 \end{cases}
\]

\begin{proposition}\label{prop:cubic-sharp}
The graphs \(X_n\) satisfy
\[
 n^2\mu(X_n)\longrightarrow2\pi^2
 \qquad(n\to\infty,\ n\text{ even}).
\]
\end{proposition}

\begin{proof}
Fix \(r,s\in\{5,7\}\), put \(\vartheta=\pi/(2m)\), and let
\[
 x_j=\cos\bigl((j+\tfrac12)\vartheta\bigr)
 \qquad(0\leqslant j<2m).
\]
On the \(i\)th diamond, give the terminals the values \(x_{2i-2}\) and \(x_{2i-1}\),
and give each of the other two vertices their average.  Give every vertex in
the left end graph the value \(x_0\), and every vertex in the right end graph
the value \(x_{2m-1}=-x_0\).  Denote the resulting function by \(F\).

The vector \((x_0,\ldots,x_{2m-1})\) is a Fiedler vector of
\(P_{2m}\).  In particular,
\begin{equation}
 \sum_{j=0}^{2m-1}x_j=0,
 \qquad
 \sum_{j=0}^{2m-1}x_j^2=m.
 \label{eq:cubic-trial-path}
\end{equation}
The end graphs and their attaching edges contribute no energy.  A diamond
contributes the square of the difference between its terminal values, while
the edges joining successive diamonds supply the remaining differences.
Consequently,
\begin{equation}
 \cE_{N_m^{r,s}}(F)
 =\sum_{j=0}^{2m-2}(x_{j+1}-x_j)^2
 =2m(1-\cos\vartheta).
 \label{eq:cubic-trial-energy}
\end{equation}

For \(m\geqslant2\), the trial values satisfy
\[
 \sum_{i=1}^{m}x_{2i-2}x_{2i-1}=\frac m2\cos\vartheta,
 \qquad \sum_vF(v)=(r-s)x_0.
\]
Let \(F_0\) be obtained from \(F\) by
subtracting its mean.  Using \eqref{eq:cubic-trial-path}, we obtain
\[
 \|F_0\|_2^2
 =\frac m2(3+\cos\vartheta)
  +\left(r+s-\frac{(r-s)^2}{4m+r+s}\right)\cos^2\frac{\vartheta}{2}.
\]
Since centring preserves the energy, the Rayleigh principle and
\eqref{eq:cubic-trial-energy} give
\[
 \mu(N_m^{r,s})\leqslant\frac{\pi^2}{8m^2}+O(m^{-3}).
\]

For both choices defining \(X_n\), we have \(n=4m+O(1)\).  The preceding
upper bound and Corollary~\ref{cor:cubic-explicit} prove the assertion.
\end{proof}

To prove Theorem~\ref{thm:cubic-stability}, we retain the terms discarded in
Proposition~\ref{prop:cubic-path}.  Near equality forces almost all vertices to lie in chains of
diamonds.

\begin{samepage}
Let \(f\) be an ordered unit Fiedler vector of a connected cubic graph,
and let \(g\) be given by \eqref{eq:cubic-interpolant}.  Let \(\mathcal B\)
be the set of bridge indices, write
\(\mathcal W=\bigcup_{a\in\mathcal B}W_a\), and let \(\mathcal D\) denote the deficit
\(\cE_G(f)-2\cE_{P_n}(g)\).  The proof of
Proposition~\ref{prop:cubic-path}, with the discarded terms retained, gives
\begin{equation}
 \mathcal D\geqslant
 \sum_{i\notin\mathcal W}(q_i-2)y_i^2\geqslant0.
 \label{eq:cubic-deficit}
\end{equation}
\end{samepage}

For the two extremal degrees, write \(a_3=2\), \(a_4=4\),
\(Z_{3,n}=X_n\), and \(Z_{4,n}=\mathcal G_n\). The following theorem
gives quantitative stability in both degrees.

\Needspace{14\baselineskip}
\begin{theorem}\label{thm:parity-stability}
\label{thm:cubic-stability}\label{thm:quartic-stability}
There are an absolute constant \(C>0\) and an integer \(n_0\) for
which the following holds.  Let \(d\in\{3,4\}\), let
\(0\leqslant\varepsilon\leqslant1\), and let \(G\) be a connected simple
\(d\)-regular graph of order \(n\geqslant n_0\) satisfying
\[
 n^2\mu(G)\leqslant a_d\pi^2+\varepsilon.
\]
Then the vertices can be relabelled so that
\begin{align*}
 |E(G)\mathbin{\triangle}E(Z_{d,n})|
 &\leqslant C(\varepsilon+n^{-1})^{1/3}n,\\
 \frac{3n}{d+1}-C(\varepsilon+n^{-1})^{1/3}n
 &\leqslant\operatorname{diam}(G)\leqslant\frac{3n}{d+1}.
\end{align*}
\end{theorem}

\begin{proof}[Proof of Theorem~\ref{thm:parity-stability} in degree three]
Put \(\eta=(\varepsilon+n^{-1})^{1/3}\).  All constants implicit in this proof
are absolute.  Let \(\mathsf P\) be the
orthogonal projection onto \(\one^\perp\).  Proposition~\ref{prop:cubic-path} gives
\[
 \|\mathsf Pg-f\|_2\leqslant\|g-f\|_2
 \leqslant\sqrt{6\mu(G)}=O(n^{-1}),
\]
so \(\|\mathsf Pg\|_2=1+O(n^{-1})\).  The Poincar\'e inequality for the path,
the hypothesis, and \(\mu(P_n)=\pi^2/n^2+O(n^{-4})\) give
\begin{equation}
 \begin{split}
 0\leqslant{}&\mathcal D
 +2\bigl(\cE_{P_n}(g)-\mu(P_n)\|\mathsf Pg\|_2^2\bigr)\\
 ={}&\mu(G)-2\mu(P_n)\|\mathsf Pg\|_2^2
 =O\left(\frac{\eta^3}{n^2}\right).
 \end{split}
 \label{eq:deficit-upper}
\end{equation}

We first show that the increments of \(g\) are close to those of a unit
Fiedler vector of \(P_n\).  Let \(\boldsymbol\varphi=(\varphi_i)_{i=1}^n\) be the
increasing unit Fiedler vector of \(P_n\), where
\[
 \varphi_i=-\sqrt{\frac2n}
 \cos\frac{(2i-1)\pi}{2n}
 \qquad(1\leqslant i\leqslant n).
\]
The next positive eigenvalue of \(P_n\) is \(2(1-\cos(2\pi/n))\), and its
difference from \(\mu(P_n)\) is at least \(12/n^2\) for \(n\geqslant4\).

Decompose
\[
 \mathsf Pg=\rho\boldsymbol\varphi+\boldsymbol\psi,
 \qquad
 \boldsymbol\psi\perp\one,\boldsymbol\varphi.
\]
Here \(\rho>0\).  Indeed, \(\mathsf Pg\) and \(\boldsymbol\varphi\) are
nonconstant, nondecreasing, and centred, whence
\[
 \langle\mathsf Pg,\boldsymbol\varphi\rangle
 =\frac1{2n}\sum_{i,j=1}^{n}
   (g_i-g_j)(\varphi_i-\varphi_j)>0.
\]
The spectral separation on the path and \eqref{eq:deficit-upper} give
\[
 \|\boldsymbol\psi\|_2^2=O(\eta^3),
 \qquad
 \cE_{P_n}(\boldsymbol\psi)=O\left(\frac{\eta^3}{n^2}\right).
\]
Since \(\rho^2+\|\boldsymbol\psi\|_2^2=\|\mathsf Pg\|_2^2=1+O(n^{-1})\)
and \(n^{-1}\leqslant\eta^3\), we have \(\rho=1+O(\eta^3)\).  It follows that
\begin{equation}
 \sum_{i=1}^{n-1}
 \left(g_{i+1}-g_i-(\varphi_{i+1}-\varphi_i)\right)^2
 =O\left(\frac{\eta^3}{n^2}\right).
 \label{eq:path-mode-stability}
\end{equation}

\begin{samepage}
We next bound the number of indices outside the bridge blocks.
By increasing \(C\), we may assume that \(\eta\leqslant1/4\).  Let \(J_\eta\)
be the set of indices between \(\lceil\eta n\rceil\) and
\(\lfloor(1-\eta)n\rfloor\).  For \(i\in J_\eta\),
\begin{equation}
 \varphi_{i+1}-\varphi_i
 =2\sqrt{\frac2n}\sin\frac{\pi}{2n}\sin\frac{\pi i}{n}
 \geqslant4\sqrt2\,\eta n^{-3/2}.
 \label{eq:path-increment-lower}
\end{equation}
\end{samepage}
Let \(\mathcal O\) be the set of odd indices not contained in \(\mathcal W\).
The parity
identity \eqref{eq:cut-parity} gives \(q_i\equiv i\pmod2\).  Thus
\(q_i\geqslant3\) for \(i\in\mathcal O\), and \eqref{eq:cubic-deficit} yields
\[
 \mathcal D\geqslant\sum_{i\in\mathcal O}y_i^2.
\]
For \(i\in\mathcal O\cap J_\eta\),
\[
 (\varphi_{i+1}-\varphi_i)^2
 \leqslant2\bigl(y_i-(\varphi_{i+1}-\varphi_i)\bigr)^2+2y_i^2.
\]
It follows from \eqref{eq:deficit-upper},
\eqref{eq:path-mode-stability}, and \eqref{eq:path-increment-lower} that
\[
 |\mathcal O\cap J_\eta|=O(\eta n).
\]
There are at most \(2\eta n+2\) indices outside \(J_\eta\), and hence
\begin{equation}
 |\mathcal O|=O(\eta n).
 \label{eq:uncovered-odd-indices}
\end{equation}

Every bridge index is odd, and each set \(W_a\) contains two odd indices.
There are \(n/2\) odd indices in \(\{1,\ldots,n-1\}\), since \(n\) is even.
Let \(k=|\mathcal B|\), and let \(u=n-1-|\mathcal W|\) be the number of
indices outside \(\mathcal W\).  Then
\[
 \frac n2=|\mathcal O|+2k,
 \qquad u=n-1-4k=2|\mathcal O|-1.
\]
Together with \eqref{eq:uncovered-odd-indices}, this gives
\begin{equation}
 u=O(\eta n).
 \label{eq:indices-outside-blocks}
\end{equation}

If \(k=0\), then \(u=n-1=O(\eta n)\), and both stability
conclusions follow by increasing the absolute constant.  We may therefore
assume \(k\geqslant1\).  We next determine the graph away from the
exceptional indices.  List the
bridge indices as
\(a_1<\cdots<a_k\).  Since the blocks are disjoint and each has four indices,
\(4k=n-1-u\).
Moreover,
\[
 u=(a_1-1)+\sum_{j=1}^{k-1}(a_{j+1}-a_j-4)+(n-a_k-4),
\]
so at most \(u\) consecutive pairs of bridge indices have distance greater
than four.

Suppose that \(a_{j+1}=a_j+4\), and let
\(T=S_{a_j+4}\setminus S_{a_j}\).  Then
\[
 q_{a_j}+q_{a_j+4}
 =|\partial T|+2e(S_{a_j},V(G)\setminus S_{a_j+4})=2.
\]
The set \(T\) has four vertices, and \(|\partial T|\) is a positive even
integer.  It follows that \(|\partial T|=2\), that there is no edge from
\(S_{a_j}\) to \(V(G)\setminus S_{a_j+4}\), and that
\[
 e(T)=\frac{3|T|-|\partial T|}{2}=5.
\]
Thus \(G[T]\) is a diamond, and its two boundary edges meet the two terminals.
The diamonds arising from all four-step gaps form at most
\(u+1\) disjoint chains.  If \(r\) denotes the number of these diamonds,
then
\[
 r\geqslant k-1-u,
 \qquad n-4r\leqslant5u+5.
\]

Every bridge index lies between \(5\) and \(n-5\), since it is odd.  It
follows that \(r\leqslant\lfloor(n-10)/4\rfloor\), which is precisely the
number of diamonds in \(X_n\).  Let \(X\) contain all vertices not belonging
to the \(r\) identified diamonds.  Then
\begin{equation}
 |X|=n-4r\leqslant5u+5.
 \label{eq:exceptional-vertices}
\end{equation}
Choose a bijection from \(V(G)\) to \(V(X_n)\) which identifies each of
these chains with a consecutive string of diamonds in \(X_n\), preserves
all edges within the chains, and maps \(X\) to the remaining vertices.  If the
identified diamonds form several chains, their number is at most \(u+1\).
Under this identification, only edges incident with \(X\) or with an end of
one of these chains can differ.  There are at most \(3|X|\) edges of each
graph incident with \(X\), and at most \(4(u+1)\) further edges at the chain
ends.  Hence
\[
 |E(G)\mathbin{\triangle}E(X_n)|
 \leqslant6|X|+4(u+1)=O(\eta n)
\]
by \eqref{eq:indices-outside-blocks} and
\eqref{eq:exceptional-vertices}.  This proves the edit-distance assertion.

Finally, let \(a<b\) be bridge indices.  Since \(b-a\) is even, applying the
identity above to \(S_b\setminus S_a\) shows that the corresponding bridge
edges are distinct.
Every path from \(S_{a_1}\) to \(V(G)\setminus S_{a_k}\) traverses all \(k\)
bridges.  Each of the \(r\) identified diamonds between successive bridges
contributes two further edges, and hence
\[
 \operatorname{diam}(G)\geqslant k+2r
 \geqslant\frac{3n}{4}-O(u+1).
\]
Together with \eqref{eq:indices-outside-blocks}, this proves the lower bound
in the theorem.  For the upper bound, a diametral path is induced.  If it has
length \(L\), exactly \(L+3\) edges join it to its complement, and each vertex
outside the path is incident with at most three of them.  Thus
\[
 n-(L+1)\geqslant\frac{L+3}{3},
\]
which is equivalent to \(L\leqslant(3n-6)/4\).
\end{proof}

\subsection{Quartic sharpness and stability}\label{sec:quartic-stability}

For quartic graphs, the cuts of size two play the role of the bridge cuts
in the cubic argument.  Let \(M_0\) be obtained from \(K_4\) by adding
two nonadjacent terminals, each adjacent to two vertices of the clique,
with their two neighbourhoods partitioning the clique.  This is the
middle block in the construction of \(\mathcal G_n\).
We first obtain a sharp path inequality while
retaining a uniform deficit outside the extremal cut pattern.

For an ordering of a connected simple quartic graph and a consecutive
increment block $W$ of length $r$, a lower profile is a vector
$\mathbf p=(p_1,\ldots,p_r)$ such that the cut sizes on $W$ are
at least the corresponding $p_i$. Let $Q_W$ be the principal
quadratic form obtained by restricting
\eqref{eq:tridiagonal-kernel} to $W$. We use the following profiles
and coefficients, with the same coefficient for a reversed profile:
\[
\begin{array}{c|c}
\mathbf p & \alpha(\mathbf p)\\ \hline
(2,4,4,4,2)&4\\
(2,4,4,4)&52/11\\
(2,4,4,4,4)&5\\
(2,4,4,4,4,2)&102/23\\
(2,4,4,4,4,4,2)&14/3\\
(4,4)&6\\
(4,4,4)&6
\end{array}
\]

\Needspace{13\baselineskip}
\begin{lemma}\label{lem:quartic-blocks}
Let $q_1,\ldots,q_{n-1}$ be the cut sizes in an ordering of a connected
simple quartic graph. Then the increment indices can be partitioned
into consecutive blocks of lengths at most seven, each admitting a
lower profile from the table above or its reversal. For every block
$W$ of length $r$ with its assigned lower profile $\mathbf p$,
\[
 Q_W(\mathbf x)\geqslant
 \frac{\alpha(\mathbf p)}{r}
 \left(\sum_{i=1}^{r}x_i\right)^2
 \qquad(\mathbf x\in[0,\infty)^r).
\]
Moreover, a block with profile $(2,4,4,4,2)$ starts at the second
of two consecutive cuts of size two and ends at the next cut of
size two.
\end{lemma}

\begin{proof}
Every cut size is a positive even integer.  Lemma~\ref{lem:interval-cut}
with \(t=2,3\) shows that two indices at which the cut size is two cannot
have distance two or three.  Thus maximal clusters of such indices have
length one or two, and consecutive clusters are separated by at least
three indices with cut size at least four.  The endpoint bounds
\(q_i\geqslant i(5-i)\), for \(1\leqslant i\leqslant4\), and their
counterparts at the other end show that there are at least four such
indices before the first cluster and after the last cluster.

Assign the first index of every cluster to the interval on its left,
and assign the second index, when present, to the interval on its right.
Between successive clusters the resulting lower profile is
\((2,4^h,2)\) or \((4^h,2)\), where \(h\geqslant3\) and \(4^h\)
denotes a string of \(h\) entries equal to four.  A profile
\((2,4^h,2)\) with \(h=3,4,5\) is already in the table.  For
\(h\geqslant6\), use a block \((2,4,4,4)\) at each end and partition
the remaining string of fours into blocks of length two or three.  If
there is precisely one remaining four, absorb it into either end block.
A profile \((4^h,2)\) is treated by taking the whole interval when
\(h=3,4\), and otherwise taking \((4,4,4,2)\) at the right end and
partitioning the remaining fours into blocks of length two or three.
The initial and final intervals are treated in the same way; each has
at least four fours and at most one entry equal to two.  If there are
no cuts of size two, partition the entire string into blocks of length
two or three.  This is possible since a simple quartic graph has
\(n\geqslant5\).

Let \(A_{\mathbf p}\) be the symmetric tridiagonal matrix with diagonal
\(p_i\) and adjacent entries \((p_i+p_{i+1}-4)/2\).  It is checked directly that each \(A_{\mathbf p}\) is positive
definite and
\[
 \one^{\mathsf T}A_{\mathbf p}^{-1}\one=\frac{r}{\alpha(\mathbf p)}.
\]
Cauchy--Schwarz in the quadratic form defined by \(A_{\mathbf p}\)
gives the listed coefficients.
Since the increments are nonnegative, replacing the actual cut sizes
by the lower profile can only decrease the tridiagonal form.
\end{proof}

Call the blocks with lower profile \((2,4,4,4,2)\) regular, and call
all the other blocks exceptional.  These terms refer only to the
partition in Lemma~\ref{lem:quartic-blocks}.  Given an ordered Fiedler
vector \(f\), define \(g\) by affine interpolation over each block:
if \(W=\{a,\ldots,a+r-1\}\), put
\[
 g_{a+t}=f_a+\frac{t}{r}(f_{a+r}-f_a)
 \qquad(0\leqslant t\leqslant r).
\]
The definitions agree at endpoints of adjacent blocks.  Let
\(\mathcal X\) be the union of the exceptional blocks.

\Needspace{12\baselineskip}
\begin{proposition}\label{prop:quartic-path}
Let $G$ be a connected simple quartic graph of order $n$, and let
$f$ be an ordered Fiedler vector. Then
\begin{align}
 \cE_G(f)-4\cE_{P_n}(g)
 &\geqslant\frac{10}{23}
   \sum_{i\in\mathcal X}(g_{i+1}-g_i)^2,
 \label{eq:quartic-deficit}\\
 \|f-g\|_2^2&\leqslant\frac{21}{2}\cE_G(f).
 \label{eq:quartic-interpolation-error}
\end{align}
Consequently,
\[
 \mu(G)\geqslant
 \frac{4\mu(P_n)}{(1+\sqrt{42\mu(P_n)})^2}.
\]
\end{proposition}

\begin{proof}
In \eqref{eq:tridiagonal-kernel}, retain the principal forms on the
blocks from Lemma~\ref{lem:quartic-blocks}.  The discarded terms are
nonnegative.  Affine interpolation over a block of length \(r\) has
path energy \((f_{a+r}-f_a)^2/r\).  Every regular block has coefficient
four, whereas every exceptional block has coefficient at least
\(102/23=4+10/23\).  Summing proves
\eqref{eq:quartic-deficit}.

On a block of length \(r\), both vectors lie between the endpoint
values, so its contribution to \(\|f-g\|_2^2\) is at most
\((r-1)(f_{a+r}-f_a)^2\).  The retained graph energy on that block
is at least \(4(f_{a+r}-f_a)^2/r\).  Since \(r\leqslant7\), summing
gives \eqref{eq:quartic-interpolation-error}.  The last assertion
follows from Lemma~\ref{lem:path-transfer}.
\end{proof}

\begin{proof}[Proof of Theorem~\ref{thm:parity-stability} in degree four]
Put \(\eta=(\varepsilon+n^{-1})^{1/3}\), take an ordered unit Fiedler
vector \(f\), and form \(g\) as above.  All implicit constants are
absolute.  Let \(\mathsf P\) be the orthogonal projection onto
\(\one^\perp\), and put
\(\mathcal D=\mu(G)-4\cE_{P_n}(g)\).  Proposition~\ref{prop:quartic-path}
gives \(\|\mathsf Pg-f\|_2=O(n^{-1})\), and therefore
\[
 0\leqslant\mathcal D+
 4\bigl(\cE_{P_n}(g)-\mu(P_n)\|\mathsf Pg\|_2^2\bigr)
 =O(\eta^3n^{-2}).
\]
Let \(\boldsymbol\varphi\) be the increasing unit Fiedler vector of
\(P_n\), as in the proof of Theorem~\ref{thm:cubic-stability}.
Decomposing \(\mathsf Pg\) into its component along
\(\boldsymbol\varphi\) and an orthogonal remainder, and using the
separation between the first two positive path eigenvalues, gives
\[
 \sum_{i=1}^{n-1}
 \bigl(g_{i+1}-g_i-(\varphi_{i+1}-\varphi_i)\bigr)^2
 =O(\eta^3n^{-2}).
\]
The coefficient along \(\boldsymbol\varphi\) is positive because
both vectors are nonconstant, nondecreasing and centred. This is
the same estimate for the increments used in the cubic proof.

We claim that \(b=|\mathcal X|=O(\eta n)\).  We may assume
\(\eta\leqslant1/4\), since otherwise all the conclusions follow by
increasing \(C\).  On the indices between \(\lceil\eta n\rceil\)
and \(\lfloor(1-\eta)n\rfloor\), the path increments satisfy
\(\varphi_{i+1}-\varphi_i\geqslant4\sqrt2\eta n^{-3/2}\).
For \(i\in\mathcal X\), use
\[
 (\varphi_{i+1}-\varphi_i)^2
 \leqslant2\bigl(g_{i+1}-g_i-(\varphi_{i+1}-\varphi_i)\bigr)^2
       +2(g_{i+1}-g_i)^2.
\]
Summing and applying \eqref{eq:quartic-deficit} bounds the number of
such central indices by \(O(\eta n)\).  At most \(2\eta n+2\)
indices lie outside this interval, proving the claim.

We next recover the graph on almost all vertices.  A regular block
has indices \(a,\ldots,a+4\), with \(q_{a-1}=q_a=2\) and
\(q_{a+4}=2\).  Suppose also that \(q_{a+5}=2\), so that the
cluster at its right end has two indices.  Put
\(T=S_{a+4}\setminus S_a\).  Since \(|T|=4\), simplicity gives
\(|\partial T|\geqslant4\), while
\[
 4=q_a+q_{a+4}
  =|\partial T|+2e(S_a,V(G)\setminus S_{a+4}).
\]
Thus \(|\partial T|=4\), there is no edge bypassing \(T\), and
\(G[T]=K_4\).  Each vertex of \(T\) has exactly one edge leaving
\(T\).  Two consecutive cuts of size two are disjoint, since their
intersection has size \((2+2-4)/2=0\).  Their symmetric difference
is the four edges at the intervening vertex.  Consequently the two
edges of \(\partial S_a\) have common endpoint \(v_a\), and the
two edges of \(\partial S_{a+4}\) have common endpoint
\(v_{a+5}\).  The subgraph on \(\{v_a\}\cup T\cup\{v_{a+5}\}\)
is therefore \(M_0\), with terminals \(v_a,v_{a+5}\).

Every singleton cluster is followed by an interval all of whose
blocks are exceptional, including when it is the last cluster.
These intervals are nonempty and disjoint.  Hence the number of
singleton clusters is at most \(b\).  There are
\((n-1-b)/5\) regular blocks, and at most \(b\) of them end at
a singleton cluster.  If \(r\) is the number of copies of \(M_0\)
just identified, then
\[
 n-5r\leqslant6b+1.
\]
Consecutive copies share one terminal and form chains; nonconsecutive
copies have disjoint vertex sets.  If there are \(c\) nonempty chains,
their union has \(5r+c\) vertices, whence \(c\leqslant6b+1\).

The graph \(\mathcal G_n\) consists of
\(m=\lfloor(n-11)/5\rfloor\) copies of \(M_0\) between two end
blocks of bounded order.  If necessary, delete copies from the ends
of the identified chains until the retained \(r'\) copies in \(c'\)
nonempty chains satisfy \(r'+c'-1\leqslant m\).  At most \(c+2\)
copies need be deleted, since \(r\leqslant n/5\) and
\(m\geqslant(n-15)/5\).  If all copies are deleted, use the empty
collection.  Thus \(n-5r'=O(b+1)\).

Map these chains to strings of copies of \(M_0\) in
\(\mathcal G_n\), leaving at least one unused copy between strings.
The strings then have disjoint vertex sets, and the map preserves
every edge within a retained chain.  Extend it to a bijection on
all vertices.  Let \(Y\) consist of the vertices outside the retained
chains and the two end terminals of each retained chain.  Then
\(|Y|=O(b+1)\), and every
edge which can differ under this bijection is incident with \(Y\).
Both graphs are quartic, so the symmetric difference has size at most
\(8|Y|=O(b+1)=O(\eta n)\).

Finally, if \(r=0\), the lower diameter bound follows immediately
from \(n\leqslant6b+1=O(\eta n)\), after increasing \(C\).
If \(r\geqslant1\), the terminal vertices of each identified copy of \(M_0\)
separate its interior from the graph on either side.  Any path from
the first left terminal to the last right terminal must traverse
every one of the \(r\) copies, using at least three edges in each.
Thus \(\operatorname{diam}(G)\geqslant3r\geqslant3n/5-O(\eta n)\).
For the upper bound, take every third vertex on a geodesic of length
\(D\).  Their closed neighbourhoods are pairwise disjoint and each
contains five vertices.  Hence
\(5(\lfloor D/3\rfloor+1)\leqslant n\), which gives
\(D\leqslant3\lfloor n/5\rfloor-1\).
\end{proof}

\subsection{Stability at fixed minimum degree}
\label{sec:minimum-degree-stability}

For \(d\geqslant3\), let \(L_d=K_{d+1}-e\), with the endpoints of
the missing edge as terminals.  Each terminal has degree \(d-1\).
Chains of these graphs joined by bridges attain the bound in
Theorem~\ref{thm:minimum-degree-bound}.  The next theorem shows that
every asymptotic minimiser has this structure on almost all vertices.

For $n\geqslant1$ and $\varepsilon\geqslant0$, write
$\eta=(\varepsilon+n^{-1})^{1/3}$.

\begin{theorem}\label{thm:minimum-degree-stability}
For every fixed integer \(d\geqslant3\), there are constants
\(C=C(d)>0\) and \(n_0=n_0(d)\) with the following property.
Let \(0\leqslant\varepsilon\leqslant1\), and let \(G\) be a connected
simple graph of order \(n\geqslant n_0\) and minimum degree at least
\(d\), satisfying
\[
 n^2\mu(G)\leqslant(d-1)\pi^2+\varepsilon.
\]
Then all but at most
\(C\eta n\) vertices of \(G\) belong to vertex-disjoint induced
copies of \(L_d\).  Each copy has exactly two boundary edges, both
bridges, one at each terminal, and the copies form at most
\(C\eta n\) chains.  Moreover,
\[
 \frac{3n}{d+1}-C\eta n
 \leqslant\operatorname{diam}(G)
 \leqslant3\left\lfloor\frac{n}{d+1}\right\rfloor-1.
\]
\end{theorem}

\begin{proof}[Proof of Theorem~\ref{thm:minimum-degree-stability}]
Put \(c=d-1\), and start with the disjoint low-cut blocks used in
the proof of Theorem~\ref{thm:minimum-degree-bound}: a block of length
\(d+1\) at each bridge index, and a block of length \(d\) at the
first index of every other cluster with \(q_i<c\).  We call these
the initial blocks.  Lemma~\ref{lem:general-low-cut} gives a fixed
strict improvement over \(c\) for each nonbridge initial block.
The remaining increment indices have cut size at least \(c\).
We refine this partition and retain a deficit on suitable blocks.

First observe that, if \(q_i=q_{i+1}=c\) and \(K_{i,i+1}=0\),
all edges in these two cuts are incident with the intervening vertex.
By simplicity, the next vertex has at most one neighbour in
\(S_{i+1}\).  The minimum-degree condition therefore gives
\begin{equation}\label{eq:neutral-cut-neighbours}
 q_{i+2}\geqslant c+d-2=2d-3,
 \qquad q_{i-1}\geqslant2d-3.
\end{equation}
The second inequality follows by reversing the order.  The endpoint
cut sizes are at least \(d\), so the neighbouring indices needed
here exist whenever two consecutive cuts have size \(c\).

Make every remaining index with \(q_i\geqslant d\) a strict
singleton block.  Consider the maximal runs of the remaining
indices with \(q_i=c\).  In a run of length at least three,
\eqref{eq:neutral-cut-neighbours} implies that every adjacent kernel
entry is at least one.  Partition the run into blocks of length two
or three.  Their diagonal entries are \(c\), and their adjacent
entries are at least one.  Direct calculation gives path coefficients
\[
 c+1,\qquad
 \frac{3(c^2-2)}{3c-4}
 =c+\frac{4c-6}{3c-4},
\]
respectively.  Both are at least \(c+1\), since \(c\geqslant2\).
Declare these blocks strict.  A run of length two is treated in the
same way if its adjacent kernel entry is positive.  Otherwise retain
its diagonal form, whose path coefficient is \(c\), and call this
block of length two neutral.  By \eqref{eq:neutral-cut-neighbours}, the index
immediately following it is a high-cut singleton.  It is uncovered
by the initial blocks, since the first cut in an initial block has
size less than \(c\).  Associate the neutral block with that strict singleton.

A run of length one is also neutral.  If it has an uncovered neighbour, that neighbour is a strict
high-cut singleton; associate the neutral index with this singleton.  Otherwise its whole uncovered interval consists of
this one index, and it lies between two initial blocks.  Indeed,
a nonempty initial or final uncovered interval contains an endpoint
cut, whose size is at least \(d\).  If one of the neighbouring
initial blocks is nonbridge, associate the neutral singleton with
that strict block.

It remains to treat a neutral singleton between two bridge blocks.
If the first bridge index is \(a\), the singleton has index
\(a+d+1\), and the next bridge has index \(a+d+2\).  The two
latter cuts have sizes \(c\) and one.  Since their sum is \(d\),
the intervening vertex has degree \(d\), their intersection is
empty, and all \(c\) edges of the earlier cut end at that vertex.
At most one starts at \(v_{a+d+1}\).  Hence
\[
 K_{a+d,a+d+1}\geqslant d-2\geqslant1.
\]
Absorb this singleton into the bridge block on its left.  This does
not reach the next bridge block, so the enlarged blocks remain
disjoint.

We record the strict coefficient for this enlargement.  Write
\(A=\operatorname{diag}(1,B)\) for the bridge matrix in
Lemma~\ref{lem:general-low-cut}.  If \(T\) is the \(d\times d\)
tridiagonal matrix with diagonal two and adjacent entries minus one,
and \(e=e_1+e_d\), then
\[
 B^{-1}=\frac{T}{d+1}+
 \frac{ee^{\mathsf T}}{(d+1)(d-1)}.
\]
This follows from the rank-one inverse formula applied to
\(B=(d+1)T^{-1}-\one\one^{\mathsf T}\).
In particular,
\[
 B^{-1}\one=\frac{e}{d-1},\qquad
 (B^{-1})_{dd}=\frac{2d-1}{d^2-1}.
\]
Append to \(A\) a diagonal entry \(c\) and an adjacent entry one.
The resulting matrix is positive definite, and its Schur complement
gives the path coefficient
\[
 c+\frac{2d+1}{(d+1)(d+3)}>c.
\]
Thus the enlarged block is strict.  All block lengths are at most
\(d+2\).

There is one further source of strictness.  Suppose that bridge
indices \(a\) and \(a+d+1\) occur, and put
\(U=S_{a+d+1}\setminus S_a\).  The cut identity and connectedness
give \(|\partial U|=2\) and no edge bypassing \(U\).  Since
\(|U|=d+1\), the minimum-degree condition implies that \(G[U]\)
is either \(K_{d+1}-e\) or \(K_{d+1}\).  In the complete case,
the principal energy of the bridge block at \(a\) is at least
\[
 y_a^2+\frac{d+1}{2}
 \left(\sum_{t=1}^{d}y_{a+t}\right)^2
 \geqslant\frac{d+1}{d+3}
 \left(\sum_{t=0}^{d}y_{a+t}\right)^2.
\]
Here we used the effective conductance \((d+1)/2\) between two
vertices of \(K_{d+1}\).  The corresponding path coefficient is
\((d+1)^2/(d+3)=c+4/(d+3)\).  Declare this bridge block strict
as well.  This reclassification does not change the partition.

Let \(\mathcal S\) be the union of all strict blocks, and let
\(\mathcal T\) contain the indices in the remaining neutral blocks of lengths one and two.  The only other blocks are ordinary bridge blocks of
length \(d+1\).  Each neutral block is associated with a high-cut singleton or a
nonbridge initial block. A high-cut singleton has at most two
associated neutral runs, each of length at most two. A nonbridge
initial block has at most one associated singleton at each end.  Therefore
\begin{equation}\label{eq:neutral-block-charging}
 |\mathcal T|\leqslant4|\mathcal S|.
\end{equation}
We keep the neutral blocks separate in the energy estimate.
If there are no low cuts, the same partition and counting argument
apply to the whole increment interval.

Affinely interpolate an ordered unit Fiedler vector \(f\) over this
partition to obtain \(g\).  The strict coefficients above and
Lemma~\ref{lem:general-low-cut} give constants \(\gamma_d>0\) and
\(\beta_d=(d+1)(d+2)/(d-1)\) for which
\begin{align}
 \cE_G(f)-c\cE_{P_n}(g)
 &\geqslant\gamma_d
   \sum_{i\in\mathcal S}(g_{i+1}-g_i)^2,
 \label{eq:min-degree-deficit}\\
 \|f-g\|_2^2&\leqslant\beta_d\cE_G(f).
 \label{eq:min-degree-interpolation-error}
\end{align}
Indeed, retain the principal cut-kernel forms on the blocks and
discard the nonnegative mixed terms between them.  The interpolation
error follows from the block-length bound just as in the proof of
Theorem~\ref{thm:minimum-degree-bound}.

The comparison with the first nonconstant path eigenvector in the
proof of Theorem~\ref{thm:cubic-stability} applies with coefficient \(c\).
For the increasing unit Fiedler vector \(\boldsymbol\varphi\) of
\(P_n\), it gives
\[
 \cE_G(f)-c\cE_{P_n}(g)=O_d(\eta^3n^{-2}),
\]
\[
 \sum_{i=1}^{n-1}
 \bigl(g_{i+1}-g_i-(\varphi_{i+1}-\varphi_i)\bigr)^2
 =O_d(\eta^3n^{-2}).
\]
The increments of \(\boldsymbol\varphi\) are bounded below by
\(4\sqrt2\eta n^{-3/2}\) on the central indices from
\(\lceil\eta n\rceil\) to \(\lfloor(1-\eta)n\rfloor\).
Using \eqref{eq:min-degree-deficit} and the preceding estimate there,
and counting the at most \(2\eta n+2\) remaining indices, yields
\(|\mathcal S|=O_d(\eta n)\). For \(\eta>1/4\), this bound
follows by increasing the final constant.  Consequently, the number
of indices outside ordinary bridge blocks satisfies
\[
 u=|\mathcal S|+|\mathcal T|=O_d(\eta n).
\]

List the ordinary bridge indices as \(a_1<\cdots<a_k\).
If \(k=0\), then \(n-1=u=O_d(\eta n)\), and the claimed lower
bounds and coverage are immediate after increasing \(C\).
Otherwise \((d+1)k=n-1-u\), and
\[
 u=(a_1-1)+\sum_{j=1}^{k-1}(a_{j+1}-a_j-d-1)
       +(n-a_k-d-1).
\]
All terms are nonnegative.  At most \(u\) consecutive pairs have
separation larger than \(d+1\).  Each pair with separation
\(d+1\) gives an induced \(L_d\): the complete alternative would
have made its left block strict.  Its two boundary edges must meet
the two missing-edge terminals, since all vertices have degree at
least \(d\).  Each of these vertices therefore has degree exactly
\(d\) in \(G\).

If \(r\) is the number of identified copies, then
\[
 r\geqslant k-1-u,
 \qquad
 n-(d+1)r\leqslant(d+2)u+d+2=O_d(\eta n).
\]
Their vertex sets are disjoint, and they form at most \(u+1\)
chains joined by the corresponding bridges.  A path from before the
first ordinary bridge to after the last must cross all \(k\)
bridges and use two further edges in every identified copy.  Hence
\[
 \operatorname{diam}(G)\geqslant k+2r\geqslant3r
 \geqslant\frac{3n}{d+1}-O_d(\eta n).
\]
For the upper bound, closed neighbourhoods of every third vertex
on a geodesic are disjoint, and each contains at least \(d+1\)
vertices.  This proves the stated bound and completes the proof.
\end{proof}

Under a bound on the maximum degree, the vertex conclusion also
gives stability in edge distance.  For sufficiently large \(n\), let
\(\mathcal L_{d,n}\) be a chain of \(L_d\)'s with complete end
graphs of orders \(d+1\) and \(d+1+s\), where
\(0\leqslant s\leqslant d\) makes the total order \(n\).
Every joining edge uses a terminal, and its end in a complete end
graph may be chosen arbitrarily.

\begin{corollary}\label{cor:bounded-degree-edit-stability}
Fix integers \(\Delta\geqslant d\geqslant3\).  Under the hypotheses
of Theorem~\ref{thm:minimum-degree-stability}, suppose also that
\(\Delta(G)\leqslant\Delta\).  Then the vertices can be relabelled
so that
\[
 |E(G)\mathbin{\triangle}E(\mathcal L_{d,n})|
 \leqslant C(d,\Delta)(\varepsilon+n^{-1})^{1/3}n.
\]
For odd \(d\) and \(d\)-regular \(G\), the reference graphs may
instead be chosen from any fixed family of \(d\)-regular chains
with the same middle blocks and end graphs of bounded order.
\end{corollary}

\begin{proof}
There are \(n/(d+1)+O_d(1)\) copies of \(L_d\) in
\(\mathcal L_{d,n}\).  Discarding at most \(O_d(1)\) copies from
the ends of the identified chains if necessary, map each retained
chain to a consecutive string of copies in \(\mathcal L_{d,n}\),
and extend this map to a vertex bijection.  Copies of \(L_d\) have
disjoint vertex sets, so these strings can be placed consecutively.
All edges within the retained chains are preserved.  Every edge
which may differ is incident with an uncovered vertex or one of
the two end terminals of a chain.  There are \(O_d(\eta n)\)
such vertices, the degrees of \(G\) are at most \(\Delta\), and
those of \(\mathcal L_{d,n}\) are at most \(2d+1\).  The asserted
bound follows.  Changing the end graphs of the reference chain affects
only \(O_d(1)\) vertices and edges, which proves the final assertion.
\end{proof}

The exponent in Theorem~\ref{thm:minimum-degree-stability} cannot be
increased, either in the diameter estimate or in the number of
exceptional vertices.

\begin{proposition}\label{prop:stability-exponent-sharp}
Fix \(d\geqslant3\). For every \(p>1/3\), there is a sequence
\((G_n)\) of connected simple graphs of minimum degree \(d\), with orders
\(n\to\infty\) and parameters \(\varepsilon_n\to0\), such that
\[
 n^2\mu(G_n)\leqslant(d-1)\pi^2+\varepsilon_n,
\]
but
\[
 \frac{3n/(d+1)-\operatorname{diam}(G_n)}
 {n(\varepsilon_n+n^{-1})^p}\longrightarrow\infty.
\]
The same assertion holds with the numerator replaced by the minimum
number of vertices outside copies of \(L_d\) having the boundary
properties in Theorem~\ref{thm:minimum-degree-stability}.
\end{proposition}

\begin{proof}
Put \(c=d-1\) and \(w=d+1\).  Join a chain of \(m\) copies of
\(L_d\) by bridges, and attach complete end graphs \(K_h\) and
\(K_w\), where \(h\geqslant w\).  Denote the resulting graph by
\(G_{m,h}\).  It has minimum degree \(d\), order
\(n=mw+h+w\), and diameter
\[
 \operatorname{diam}(G_{m,h})=3m+3,
 \qquad
 \frac{3n}{w}-\operatorname{diam}(G_{m,h})=\frac{3h}{w}.
\]
We show that, uniformly when \(h=o(n)\),
\begin{equation}\label{eq:large-cap-trial}
 n^2\mu(G_{m,h})\leqslant
 c\pi^2+O_d\left((h/n)^3+n^{-1}\right).
\end{equation}

Let \(F(s)=\cos(\pi s/n)\).  Give all vertices in the left end graph
the value \(F(h)\).  Starting at coordinate \(h\), give each bridge
coordinate length \(c\), and give the two-terminal interior of
each copy of \(L_d\) coordinate length two.  At each terminal use
the value of \(F\) at its coordinate, and give the other \(c\)
vertices of the copy the average of its terminal values.  The final
bridge ends at coordinate \(n-2\); give every vertex in the right
end graph the value \(F(n-2)\).  Write \(f\) for this graph function.

A bridge has conductance one and coordinate length \(c\), while
the effective conductance between the terminals of \(L_d\) is
\(c/2\), and its coordinate length is two.  Cauchy--Schwarz on
each coordinate interval therefore gives
\[
 \cE_{G_{m,h}}(f)
 \leqslant c\int_h^{n-2}F'(s)^2\,ds
 \leqslant c\int_h^nF'(s)^2\,ds.
\]
Group each copy of \(L_d\) with the bridge immediately preceding
it, giving a coordinate interval of length \(w\). Since \(d\) is
fixed and \(|F'|\leqslant\pi/n\), comparison with \(F\) on
these intervals gives
\begin{align*}
 \sum_v f(v)^2
 &=hF(h)^2+\int_h^nF(s)^2\,ds+O_d(1),\\
 \sum_v f(v)
 &=hF(h)+\int_h^nF(s)\,ds+O_d(1).
\end{align*}
The error is uniform in \(m,h\): each copy has \(w\) vertices,
whose assigned values differ by \(O_d(n^{-1})\) from the values
of \(F\) on the associated interval. The right end graph has
bounded order and satisfies the same estimate.

Writing \(t=h/n\), elementary integration yields
\begin{align*}
 c\int_h^n F'(s)^2\,ds
 &=\frac{c\pi^2}{2n}
   -\frac{c\pi^4h^3}{3n^4}+O_d(h^5n^{-6}),\\
 hF(h)^2+\int_h^n F(s)^2\,ds
 &=\frac n2-\frac{2\pi^2h^3}{3n^2}+O(h^5n^{-4}),\\
 hF(h)+\int_h^n F(s)\,ds
 &=O(h^3n^{-2}).
\end{align*}
Subtracting the mean from \(f\) therefore leaves squared norm
\(n/2+O_d(nt^3+1)\).  The Rayleigh quotient proves
\eqref{eq:large-cap-trial}.

Now take \(h=\lfloor m^{3/4}\rfloor\), so that
\(h\asymp_d n^{3/4}\), and choose
\(\varepsilon_n=C_d((h/n)^3+n^{-1})\) with a sufficiently large
fixed constant.  Then \(\varepsilon_n=O_d(n^{-3/4})\), whereas
the diameter deficit is of order \(n^{3/4}\).  For \(p>1/3\),
\[
 \frac{n^{3/4}}{n(\varepsilon_n+n^{-1})^p}
 \gg_d n^{(3p-1)/4}\longrightarrow\infty.
\]
Finally, every vertex in a copy of \(L_d\) with exactly its two
terminal boundary edges has degree \(d\) in the ambient graph.
All \(h\) vertices in the large end graph have degree greater than \(d\)
for sufficiently large \(m\), so all must be exceptional.  This
proves the second assertion.
\end{proof}

\subsection{Stability at fixed even degree}
\label{sec:fixed-even-stability}

Let $Q_d=K_{d+1}-P_3$, where the two edges of a path on three
vertices are deleted. Call the centre of the missing path its
concentrated terminal, and its two leaves the split terminals. A chain
of these graphs is formed by joining the two split terminals of each
copy to the concentrated terminal of the next copy. After regrouping
at the cutvertices, these are the chains with middle blocks
$M_d=K_1+K_2+K_{d-2}+K_1$ in
\cite[Figure~4 and Table~1]{AbdiGhorbaniDiameter}.
For each sufficiently large $n$, fix any $d$-regular reference
chain $\mathcal M_{d,n}$ of order $n$ with these middle blocks
and end graphs of bounded order.

As above, write $\eta=(\varepsilon+n^{-1})^{1/3}$ for the stability
error.

\Needspace{14\baselineskip}
\begin{theorem}\label{thm:fixed-even-stability}
For every fixed even $d\geqslant6$, there are constants $C_d$ and
$n_d$ with the following property. Let $0\leqslant\varepsilon\leqslant1$
and let $G$ be a connected simple $d$-regular graph of order
$n\geqslant n_d$ with
\[
 n^2\mu(G)\leqslant2(d-2)\pi^2+\varepsilon.
\]
Then all but at most $C_d\eta n$ vertices belong to vertex-disjoint
induced copies of $Q_d$, forming at most $C_d\eta n$ chains.
Each copy has exactly four boundary edges, two at its concentrated
terminal and one at each split terminal. Moreover,
\[
 \frac{3n}{d+1}-C_d\eta n
 \leqslant\operatorname{diam}(G)
 \leqslant3\left\lfloor\frac n{d+1}\right\rfloor-1.
\]
After relabelling the vertices,
\[
 |E(G)\mathbin{\triangle}E(\mathcal M_{d,n})|
 \leqslant C_d\eta n.
\]
\end{theorem}

\begin{proof}
Set $c=2(d-2)$, $\omega=\varepsilon+n^{-1}$, and order a unit,
mean-zero Fiedler vector $f$ increasingly.
Apply Proposition~\ref{prop:even-strict-compression}, writing
$L=B-A=n+O_d(1)$ for the interpolation interval length.
The step comparison in that proposition gives
\begin{equation}\label{eq:even-stability-norm}
 \int_A^B|g-\overline g|^2=1+O_d(n^{-1}),\qquad
 \int_A^B g=O_d(n^{-1/2}).
\end{equation}
Indeed, $\max_i|f_i|\leqslant\sqrt{n\mu(G)}=O_d(n^{-1/2})$,
so extending or restricting the domain of the step function by a
bounded amount changes its squared norm by $O_d(n^{-1})$ and its
integral by $O_d(n^{-1/2})$.
The interpolation error has $L^2$ norm $O_d(n^{-1})$.

Let $\varphi(s)=-\sqrt{2/L}\cos(\pi(s-A)/L)$ be the increasing
unit first Neumann eigenfunction on $[A,B]$.
Expansion in the Neumann cosine basis, using
\eqref{eq:even-stability-norm} and the spectral separation between
the first and second nonconstant modes, gives
\begin{align}
 \mu(G)-c\int_A^B|g'|^2&=O_d(\omega n^{-2}),
 \label{eq:even-stability-deficit}\\
 \int_A^B|g'-a\varphi'|^2&=O_d(\omega n^{-2}),
 \qquad a^2=1+O_d(\omega),
 \label{eq:even-stability-mode}
\end{align}
where $a=\int_A^B(g-\overline g)\varphi\geqslant0$.
Indeed, writing $a_j$ for the coefficient of the $j$th normalised
nonconstant Neumann eigenfunction, the excess is
\[
 \int_A^B|g'|^2-\frac{\pi^2}{L^2}
                         \int_A^B|g-\overline g|^2
 =\sum_{j\geqslant2}\frac{(j^2-1)\pi^2}{L^2}a_j^2
 =O_d(\omega n^{-2}),
\]
which gives the two estimates. The sign of $a$ follows
from the monotonicity of $g$ and $\varphi$.
It suffices to consider $\omega$ below a fixed positive constant;
otherwise all conclusions follow by enlarging $C_d$.
In particular we may assume $a\geqslant1/2$.

Let $\mathcal B$ be the union of the strict pieces in
Proposition~\ref{prop:even-strict-compression}.
Its energy deficit and \eqref{eq:even-stability-mode} imply
\[
 \int_{\mathcal B}|\varphi'|^2=O_d(\omega n^{-2}).
\]
On $[A+\eta L,B-\eta L]$, the squared derivative of $\varphi$
is at least a fixed multiple of $\eta^2n^{-3}$, provided
$\eta\leqslant1/4$. Consequently
\begin{equation}\label{eq:even-bad-length}
 |\mathcal B|\leqslant2\eta L+
             O_d(\omega n/\eta^2)=O_d(\eta n).
\end{equation}
The case $\eta>1/4$ is again absorbed into the constant.

Every interpolation piece has coordinate length bounded below and
rank span bounded above by positive constants depending only on
$d$. Therefore \eqref{eq:even-bad-length} bounds both the number
of strict pieces and the number of vertex ranks they cover by
$O_d(\eta n)$. Only a bounded number of noncanonical end pieces may lack a strict
improvement.
It follows that all but $O_d(\eta n)$ vertices lie in canonical
intervals
\[
 q_a=q_b=2,\qquad b-a=d+1,\qquad
 G[S_b\setminus S_a]=Q_d,\qquad \delta_b=\delta_a.
\]
Their vertex sets are disjoint, and the number of intervening
exceptional intervals is $O_d(\eta n)$.

We check that consecutive canonical intervals form the asserted
chains. If a copy of $Q_d$ is oriented from its concentrated terminal to
its split terminals, its
left cut has $\delta=-\gamma$; equality of its two corrections
requires the same value at its right cut. If it is oriented from
split to concentrated, both corrections are $+\gamma$.
Thus adjacent canonical intervals have the same orientation.
Equivalently, changing from concentrated--split to
split--concentrated produces a split--split joining cut and a
strict correction in its neighbours, while the opposite change
would require both shores of one cut to be concentrated, contrary
to simplicity. The two joining edges are therefore exactly those
from the split terminals of one copy to the concentrated terminal of the
next, after reversing a chain if necessary.
There are no edges bypassing a canonical copy of $Q_d$.
Its internal degrees also show that its four boundary edges have
exactly the stated incidences.

Let $r$ be the number of retained copies. The rank count gives
\[
 n-(d+1)r=O_d(\eta n),
\]
and they form $O_d(\eta n)$ chains. A path from before the first retained copy to after the last
must pass through every retained copy in rank order. Within each one it must use at least two edges
between the concentrated and split terminals, and between successive
copies it uses at least one further edge. Hence
\[
 \operatorname{diam}(G)\geqslant3r-1
 \geqslant\frac{3n}{d+1}-O_d(\eta n).
\]
If $r=0$, the lower bound follows from the rank count after
increasing the constant. For the upper bound, the closed
neighbourhoods of every third vertex of a geodesic are disjoint,
and each contains at least $d+1$ vertices.

Finally, $\mathcal M_{d,n}$ contains $n/(d+1)+O_d(1)$ disjoint
copies of $Q_d$. Discard at most $O_d(1)$ copies from the ends of
the retained chains if needed, place each remaining chain on a
consecutive string in $\mathcal M_{d,n}$, and extend the map to a
vertex bijection. All edges within these strings are preserved.
Every potentially different edge is incident with an uncovered
vertex or one of the three outer terminals of a retained chain.
There are $O_d(\eta n)$ such vertices, and both graphs have degree
$d$. This proves the edge-distance bound.
\end{proof}

\subsection{Maximum diameter and minimum algebraic connectivity}
\label{sec:diameter-converse}

Combining the sharp bounds with the stability theorems proves the
diameter conjecture.

\begin{corollary}[Abdi--Ghorbani diameter conjecture]
\label{cor:abdi-ghorbani-diameter}
Fix \(d\geqslant3\). Let \((G_n)\) be a sequence of connected
graphs of order \(n\), each of minimum degree \(d\), or each
\(d\)-regular.
If their algebraic connectivity is asymptotically minimum in the
respective class, then
\[
 \operatorname{diam}(G_n)=(1+o(1))\frac{3n}{d+1}.
\]
\end{corollary}

The corollary follows from Theorems~\ref{thm:minimum-degree-bound},
\ref{thm:minimum-degree-stability}, \ref{thm:fixed-regular-sharp},
and~\ref{thm:fixed-even-stability}, together with the degree-three
and degree-four stability theorem.

The reverse implication in Corollary~\ref{cor:abdi-ghorbani-diameter}
holds in precisely the low-degree classes left open in
\cite[Section~5, question~(b)]{AbdiGhorbaniDiameter}.

\begin{theorem}\label{thm:diameter-converse}
Let \((G_n)\) be a sequence of connected graphs of order \(n\).
If \(\delta(G_n)\geqslant3\), then
\[
 \operatorname{diam}(G_n)=(1+o(1))\frac{3n}{4}
 \quad\Longleftrightarrow\quad
 \mu(G_n)=(1+o(1))\frac{2\pi^2}{n^2}.
\]
For a sequence of connected quartic graphs,
\[
 \operatorname{diam}(G_n)=(1+o(1))\frac{3n}{5}
 \quad\Longleftrightarrow\quad
 \mu(G_n)=(1+o(1))\frac{4\pi^2}{n^2}.
\]
In particular, the first equivalence holds for cubic graphs.
\end{theorem}

For regular degrees at least five, and for prescribed minimum
degrees at least four, the implication from diameter to algebraic
connectivity is false by
\cite[Theorem~5.3(iii)--(iv)]{AbdiGhorbaniDiameter}.
Thus Theorem~\ref{thm:diameter-converse} completes the answer to
that question. Their Theorem~5.3(ii) proves the cubic and quartic
assertions when the diameter has additive error \(O(1)\).
The minimum-degree-three case with this additive error was proved
in \cite[Theorem~5.4]{AbdiGhorbaniMinDegree}.
The following proof allows an arbitrary \(o(n)\) error.

\begin{proof}
The implications from algebraic connectivity to diameter follow
from Theorems~\ref{thm:minimum-degree-stability}
and~\ref{thm:parity-stability}. We prove the converse.
Set \((d,c)=(3,2)\) in the minimum-degree case and
\((d,c)=(4,4)\) in the quartic case, and put \(w=d+1\).
Write \(D=\operatorname{diam}(G)\). Take the breadth-first
layers \(V_0,\ldots,V_D\) from an endpoint of a diametral path,
and set \(a_i=|V_i|\). For \(1\leqslant i\leqslant D-1\),
\[
 e_i=a_{i-1}+a_i+a_{i+1}-w\geqslant0.
\]
The total excess satisfies the exact identity
\begin{equation}\label{eq:bfs-diameter-excess}
 h:=3n-wD+w
   =\sum_{i=1}^{D-1}e_i+2a_0+a_1+a_{D-1}+2a_D.
\end{equation}
Our hypothesis gives \(h=o(n)\).

Discard the layers within distance four of an index with
\(e_i>0\), and of either end. The discarded layers contain
\(O_d(h)\) vertices. Indeed, their number is \(O(h)\),
\(a_i\leqslant w+e_i\) for every interior index, and
\eqref{eq:bfs-diameter-excess} controls both the sum of the
excesses and the end layers. There are \(O(h)\) remaining
intervals of layers. Discarding bounded intervals, and shortening
the others by a bounded number of layers, discards another
\(O_d(h)\) vertices.

On every remaining interval the triples have size exactly \(w\).
A vertex in the middle layer of such a triple has all its
neighbours in that triple and at least \(d\) neighbours. Thus
the layers are cliques, adjacent layers are completely joined,
and \(a_{i+3}=a_i\). For \(d=3\) the period is a permutation
of \((1,1,2)\). For \(d=4\), the possibilities are
\((1,1,3)\) and \((1,2,2)\). The first gives two consecutive
singleton layers, whose unique joining edge is a bridge of
\(G\). An even regular graph has no bridge, so only
\((1,2,2)\) occurs in the quartic case.

Trim each interval to full periods with singleton end layers.
In the minimum-degree-three case, a period is \(K_4-e\) followed by its next
bridge; in the quartic case it is
\(K_1+K_2+K_2+K_1\), with the last singleton shared with
the next period. In either case a period has \(w\) vertices
when its right endpoint is excluded, and resistance \(w/c\)
between its endpoints. Let \(m\) be the total number of
periods, in their order along the layers. Then
\[
 L:=wm=n-O_d(h).
\]

Give the successive period endpoints the values
\(u_j=\cos(\pi j/m)\), \(0\leqslant j\leqslant m\),
and extend harmonically inside each period. Across each
discarded interval use the common value at the endpoints of
the two neighbouring retained intervals; use the first or last
endpoint value on the discarded end intervals. This defines
a function \(f\) on all vertices. Edges join only the same
or adjacent breadth-first layers, so the discarded intervals
contribute no energy. Hence
\[
 \sum_{uv\in E(G)}(f(u)-f(v))^2
  =\frac{c}{w}\sum_{j=1}^{m}(u_j-u_{j-1})^2
  \leqslant\frac{c\pi^2}{2L}.
\]
Each value in a period lies between its two endpoint values.
Comparing with the left endpoint therefore gives a total error
\(O_d(1)\) in both the sum and the squared norm over all
retained periods. The remaining \(O_d(h)\) vertices have
values in \([-1,1]\). The elementary cosine sums now give
\[
 \sum_v f(v)=O_d(h),\qquad
 \sum_v f(v)^2=\frac n2+O_d(h).
\]
Subtracting the mean and applying the variational principle yields
\[
 \mu(G)\leqslant
 \frac{c\pi^2}{n^2}\bigl(1+O_d(h/n)\bigr).
\]
The matching lower bounds follow from
Theorems~\ref{thm:minimum-degree-bound}
and~\ref{thm:fixed-regular-sharp}.
\end{proof}
\section{The maximum relaxation time}\label{sec:reciprocal-cuts}

In this section, we determine the maximum relaxation time among regular
graphs, identify its unique extremisers, and prove quantitative stability.
We first obtain sharp spectral bounds in terms of edge-connectivity when
the degree tends to infinity, together with a characterisation of
asymptotic equality. These bounds and the fixed-degree results of
Section~\ref{sec:path-estimates} reduce the maximisation to the cubic
and quartic cases. We then determine the unique quartic minimiser and
derive the first correction terms for the cubic and quartic extremisers.
These results give Theorem~\ref{thm:max-relaxation} and
Corollary~\ref{cor:global-stability}.

\subsection{A Dirichlet eigenvalue bound}

In this subsection and the next two, let \(G\) be a connected
\(d\)-regular graph of order \(n\), where
\(2\leqslant d\leqslant n-1\), and let \(f\) be a Fiedler vector
with \(f_1\leqslant\cdots\leqslant f_n\). Write
\(y_i=f_{i+1}-f_i\), \(S_i=\{v_1,\ldots,v_i\}\), and
\(q_i=|\partial S_i|\).

For the ordered Fiedler vector \(f\), let
\[
 h_i=-\sum_{a=1}^{i}f_a\quad(0\leqslant i\leqslant n),
 \qquad h_0=h_n=0.
\]
These partial sums are nonnegative.  Indeed, if
\(\sum_{a=1}^{i}f_a>0\) for some \(i<n\), its average is positive and does not
exceed \(f_i\).  Every later entry is then positive, contradicting
\(\sum_{a=1}^{n} f_a=0\).  Moreover, \(f_i=h_{i-1}-h_i\) and
\(y_i=2h_i-h_{i-1}-h_{i+1}\).

Every edge \(v_av_b\) crossing \(S_i\), with \(a\leqslant i<b\), satisfies
\(f_b-f_a=\sum_{t=a}^{b-1}y_t\geqslant y_i\).  Summing the eigenvalue
equation \(\mathsf Lf=\mu(G)f\) over \(S_i\) therefore gives
\begin{equation}
 \mu(G)h_i
 =\sum_{\substack{a\leqslant i<b\\v_av_b\in E(G)}}(f_b-f_a)
 \geqslant q_iy_i.
 \label{eq:prefix-bound}
\end{equation}
Let \(\mathsf R=\operatorname{diag}(q_1^{-1},\ldots,q_{n-1}^{-1})\), and write
\(\mathbf h=(h_1,\ldots,h_{n-1})^{\mathsf T}\).  Multiplying
\eqref{eq:prefix-bound} by \(h_i/q_i\), summing over \(i\), and using discrete
summation by parts gives
\begin{equation}
 \mu(G)\mathbf h^{\mathsf T}\mathsf R\mathbf h
 \geqslant\sum_{i=1}^{n-1}h_iy_i
 =\sum_{i=1}^{n}(h_i-h_{i-1})^2
 =\sum_{i=1}^{n}f_i^2.
 \label{eq:summation-by-parts}
\end{equation}

Let \(\mathsf L_D\) be the \((n-1)\)-by-\((n-1)\) Dirichlet Laplacian on
the path, with diagonal entries \(2\) and adjacent entries \(-1\).
For \(\mathbf u=(u_1,\ldots,u_{n-1})^{\mathsf T}\), set \(u_0=u_n=0\).
The least weighted Dirichlet eigenvalue is
\[
 \lambda_D(\mathsf R)=
 \min_{\mathbf u\in\R^{n-1}\setminus\{\mathbf 0\}}
 \frac{\mathbf u^{\mathsf T}\mathsf L_D\mathbf u}
      {\mathbf u^{\mathsf T}\mathsf R\mathbf u}
 =\min_{\mathbf u\in\R^{n-1}\setminus\{\mathbf 0\}}
  \frac{\sum_{i=1}^{n}(u_i-u_{i-1})^2}
       {\sum_{i=1}^{n-1}u_i^2/q_i}.
\]
The vector \(\mathbf h\) is nonzero because \(f\) is nonzero, and
\(\mathsf R\) is positive definite because \(G\) is connected.
Equation~\eqref{eq:summation-by-parts} now gives
\begin{equation}
 \mu(G)\geqslant
 \frac{\mathbf h^{\mathsf T}\mathsf L_D\mathbf h}
      {\mathbf h^{\mathsf T}\mathsf R\mathbf h}
 \geqslant\lambda_D(\mathsf R).
 \label{eq:weighted-dirichlet}
\end{equation}

Let \(\mathsf B=\mathsf L_D^{-1}\).  Its entries are
\begin{equation}
 B_{ij}=
 \frac{\min\{i,j\}\bigl(n-\max\{i,j\}\bigr)}{n}
 \qquad(1\leqslant i,j<n).
 \label{eq:green}
\end{equation}
The function \((i,j)\mapsto B_{ij}\) is the discrete Dirichlet Green function
for the path with zero boundary values at \(0\) and \(n\); the formula is
verified by taking second differences in either index.
The reciprocal of \(\lambda_D(\mathsf R)\) is the spectral radius of the
nonnegative matrix \(\mathsf R^{1/2}\mathsf B\mathsf R^{1/2}\).  Applying the
Collatz--Wielandt bound
\cite[Theorem~8.1.26]{HornJohnson} with the positive vector
\((q_i^{-1/2})_{i=1}^{n-1}\) gives
\begin{equation}
 \lambda_D(\mathsf R)^{-1}
 \leqslant\max_{1\leqslant i<n}
 \sum_{j=1}^{n-1}\frac{B_{ij}}{q_j}.
 \label{eq:green-row-sum}
\end{equation}

To bound the row sums in \eqref{eq:green-row-sum}, we first consider the
reciprocal cut sum over \(d+1\) consecutive indices.
For integers \(1\leqslant\kappa<d\), let \(H_d=\sum_{t=1}^{d}1/t\), and write
\begin{equation}
 \theta_{d,\kappa}=\max\left\{
 \frac1\kappa+\frac{2dH_d}{(d-\kappa)(d+1)},
 \frac1{\kappa+1}+\frac{4H_d}{d+1}
 \right\}.
 \label{eq:theta-d-kappa}
\end{equation}

\begin{lemma}\label{lem:local-reciprocal-sum}
\begin{samepage}
Let \(1\leqslant\kappa<d\) be integers, and let \(G\) be a connected simple
\(d\)-regular graph of order \(n\) and edge-connectivity at least
\(\kappa\).  Let \(q_1,\ldots,q_{n-1}\) be the cut sizes of a Fiedler
ordering of \(G\), and let \(I\) consist of \(d+1\) consecutive indices in
\(\{1,\ldots,n-1\}\).  Then
\[
 \sum_{j\in I}\frac1{q_j}\leqslant\theta_{d,\kappa}.
\]
\end{samepage}
\end{lemma}

\begin{proof}
Choose \(j_\ast\in I\) so that
\(q_\ast=q_{j_\ast}=\min_{j\in I}q_j\). For \(j\in I\),
Lemma~\ref{lem:interval-cut}, with \(t=|j-j_\ast|\), gives
\[
 q_j\geqslant\max\{q_\ast,t(d+1-t)-q_\ast\}.
\]
Suppose that \(j_\ast\) lies \(a\) positions from the left endpoint of \(I\).
The multiset of distances from \(j_\ast\) is
\[
 \{0,1,\ldots,a\}\mathbin{\uplus}\{1,\ldots,d-a\}.
\]
Since \(t(d+1-t)\) is unchanged when \(t\) is replaced by \(d+1-t\),
the corresponding multiset of these values is precisely that obtained from
\(t=0,1,\ldots,d\).  Consequently,
\begin{equation}
 \sum_{j\in I}\frac1{q_j}
 \leqslant\frac1{q_\ast}+
 \sum_{t=1}^{d}\frac1{\max\{q_\ast,t(d+1-t)-q_\ast\}}.
 \label{eq:local-reciprocal-sum}
\end{equation}

If \(q_\ast=\kappa\), then \(t(d+1-t)\geqslant d\), and hence
\[
 t(d+1-t)-\kappa\geqslant
 \frac{d-\kappa}{d}\,t(d+1-t).
\]
The identity
\[
 \sum_{t=1}^{d}\frac1{t(d+1-t)}
 =\frac{2H_d}{d+1}
\]
shows that the right-hand side of \eqref{eq:local-reciprocal-sum} is at most the first
quantity in \eqref{eq:theta-d-kappa}.  If \(q_\ast\geqslant\kappa+1\), then
\(1/q_\ast\leqslant1/(\kappa+1)\)
and
\[
 \max\{q_\ast,t(d+1-t)-q_\ast\}\geqslant\frac{t(d+1-t)}2.
\]
The second quantity in \eqref{eq:theta-d-kappa} now bounds
\eqref{eq:local-reciprocal-sum}.
\end{proof}

Since \(H_d\leqslant1+\log d\),
\begin{equation}
 \theta_{d,\kappa}=\frac1\kappa
 +O_\kappa\left(\frac{\log d}{d}\right)
 \qquad(d\to\infty).
 \label{eq:theta-kappa-asymptotic}
\end{equation}

Assume that \(G\) has edge-connectivity at least \(\kappa\) and that
\(d<\lfloor n/2\rfloor\).  Put \(\ell=d+1\), so that
\(n\geqslant2\ell\), and partition the cut indices into
\[
 J_-=\{1,\ldots,d\},\qquad
 J_0=\{\ell,\ldots,n-\ell\},\qquad
 J_+=\{n-d,\ldots,n-1\}.
\]
On \(J_0\), let \(\mathbf w=(w_j)_{j\in J_0}\), where
\(w_j=(\theta_{d,\kappa}q_j)^{-1}\).  Lemma~\ref{lem:local-reciprocal-sum} gives
\(\sum_{j\in I}w_j\leqslant1\) for every consecutive interval
\(I\subseteq J_0\) containing at most
\(\ell\) indices.  Indeed, such an interval can be enlarged, if necessary,
to an \(\ell\)-term interval in \(\{1,\ldots,n-1\}\).

For an interval \(J\) of integers, call \(S\subseteq J\)
\emph{\(\ell\)-separated} if \(|s-t|\geqslant\ell\) whenever \(s,t\in S\)
are distinct.

For a finite interval \(J\) of integers and \(S\subseteq J\), write
\(\one_S\in\R^J\) for the indicator vector of \(S\).

\begin{lemma}\label{lem:fractional-packing}
Let \(J\) be a finite interval of integers, and let \(\ell\) be a positive
integer.  Suppose that \(\mathbf w=(w_j)_{j\in J}\) is nonnegative and that
\(\sum_{j\in I}w_j\leqslant1\) for every consecutive interval
\(I\subseteq J\) with \(|I|\leqslant\ell\).  Then \(\mathbf w\) is a convex
combination of the vectors \(\one_S\) for \(\ell\)-separated subsets
\(S\subseteq J\).
\end{lemma}

\begin{proof}
Let \(\mathsf M\) be the incidence matrix whose rows are the consecutive intervals
of \(J\) of length at most \(\ell\) and whose columns are the points of
\(J\).  Its rows have consecutive ones, so \(\mathsf M\) is totally unimodular
\cite[Chapter~19, Example~7, p.~279]{Schrijver}.  The Hoffman--Kruskal theorem
\cite[Corollary~19.2a, p.~268]{Schrijver} therefore shows that every vertex of
\(\{\mathbf w:\mathsf M\mathbf w\leqslant\one,\ \mathbf w\geqslant0\}\) is integral.
The singleton constraints give \(0\leqslant w_j\leqslant1\), so every vertex
is the indicator vector of an \(\ell\)-separated set.  Conversely, each such
indicator vector satisfies the interval constraints.  The polytope is bounded
and hence is the convex hull of its vertices.
\end{proof}

Lemma~\ref{lem:fractional-packing} reduces the sum over \(J_0\) in
\eqref{eq:green-row-sum} to the following sum over an \(\ell\)-separated set.

For \(n\geqslant2\), let \(\mathsf B\) be the Dirichlet Green matrix
with entries given by \eqref{eq:green}.

\begin{lemma}\label{lem:green-packing}
Let \(n\) and \(\ell\) be positive integers with \(n\geqslant2\ell\).
Let \(S\subseteq\{\ell,\ldots,n-\ell\}\) be \(\ell\)-separated.  Then, for every
\(1\leqslant i<n\),
\begin{equation}
 \sum_{j\in S}B_{ij}\leqslant\frac{n^2}{8\ell}.
 \label{eq:green-packing}
\end{equation}
\end{lemma}

\begin{proof}
Let \(F\) be the piecewise-affine interpolation on \([0,n]\) of
\(F(i)=\sum_{j\in S}B_{ij}\) for \(1\leqslant i<n\), with
\(F(0)=F(n)=0\). Each summand
is a tent function whose slope decreases at \(j\), so \(F\) is concave
and piecewise affine. If \(S\ne\varnothing\), its maximum is attained
at some \(x\in S\): when the maximum is attained on an interval, an
endpoint of a maximal affine segment is a breakpoint in \(S\).
The case \(S=\varnothing\) is immediate.

Fix such an \(x\), and put
\[
 m=|S\cap[\ell,x]|,\qquad
 s=|S\cap(x,n-\ell]|.
\]
The separation and endpoint restrictions imply
\[
 m\leqslant\frac{x}{\ell},\qquad
 s\leqslant\frac{n-x}{\ell}-1.
\]
By the separation condition, the left-hand sum is maximised when the points
are placed as far to the right as possible.  Thus
\begin{equation}
 \sum_{\substack{j\in S\\j\leqslant x}}j
 \leqslant mx-\frac{\ell m(m-1)}2
 \leqslant\frac{x^2}{2\ell}+\frac{x}{2}.
 \label{eq:left-packing}
\end{equation}
Similarly, the right-hand sum is maximised when the points are placed as far
to the left as possible, and
\begin{equation}
 \sum_{\substack{j\in S\\j>x}}(n-j)
 \leqslant s(n-x)-\frac{\ell s(s+1)}2
 \leqslant\frac{(n-x)^2}{2\ell}-\frac{n-x}{2}.
 \label{eq:right-packing}
\end{equation}
The first quadratic is increasing for \(0\leqslant m\leqslant x/\ell\), and
the second for \(0\leqslant s\leqslant(n-x)/\ell-1\).  The stated bounds
follow by substituting the corresponding real upper limits.

Using \eqref{eq:green} to express \(F(x)\) in terms of these two sums, we obtain
\[
 \begin{aligned}
 F(x)
 &=\frac{n-x}{n}
   \sum_{\substack{j\in S\\j\leqslant x}}j
  +\frac{x}{n}
   \sum_{\substack{j\in S\\j>x}}(n-j)\leqslant\frac{x(n-x)}{2\ell}
 \leqslant\frac{n^2}{8\ell}.
 \end{aligned}
\]
This proves \eqref{eq:green-packing}.
\end{proof}

For the boundary indices \(1\leqslant j\leqslant d\), simplicity gives
\(q_j\geqslant j(d-j+1)\), while \eqref{eq:green} gives
\(B_{ij}\leqslant j\).  Therefore
\begin{equation}
 \sum_{j\in J_-}\frac{B_{ij}}{q_j}
 \leqslant\sum_{j=1}^{d}\frac1{d-j+1}=H_d.
 \label{eq:left-endpoint}
\end{equation}
Applying the same argument to complementary shores yields the same
bound on \(J_+\).

On \(J_0\), Lemma~\ref{lem:fractional-packing} expresses \(\mathbf w\) as a convex
combination of indicator vectors of \(\ell\)-separated sets.  Lemma~\ref{lem:green-packing}
then gives
\begin{equation}
 \sum_{j\in J_0}\frac{B_{ij}}{q_j}
 \leqslant\theta_{d,\kappa}\frac{n^2}{8(d+1)}.
 \label{eq:central-green}
\end{equation}
\begin{samepage}
Substituting these bounds for \(J_-\), \(J_0\), and \(J_+\) into
\eqref{eq:green-row-sum}, and using \eqref{eq:weighted-dirichlet}, gives
the following result.
\nopagebreak[4]

\begin{proposition}\label{prop:green-finite}
Let \(1\leqslant\kappa<d\) be integers, and let \(G\) be a connected simple
\(d\)-regular graph of order \(n\) and edge-connectivity at least \(\kappa\).
Suppose that \(d<\lfloor n/2\rfloor\).  Then
\begin{align}
 \mu(G)&\geqslant
 \left(\theta_{d,\kappa}\frac{n^2}{8(d+1)}+2H_d\right)^{-1},\notag\\
 \frac{n^2\mu(G)}d&\geqslant
 \frac{8(d+1)/d}
 {\theta_{d,\kappa}+16(d+1)H_d/n^2}.
 \label{eq:green-bound-normalised}
\end{align}
\end{proposition}
\end{samepage}

\begin{corollary}\label{cor:edge-connectivity-growing}
Let \(\kappa\) be a positive integer, and let \((G_m)\) be a sequence of
connected simple \(d_m\)-regular graphs of orders \(n_m\).  Suppose that
every \(G_m\) has edge-connectivity at least \(\kappa\), that
\(d_m\to\infty\), and that \(d_m<\lfloor n_m/2\rfloor\).  Then
\[
 \liminf_{m\to\infty}\frac{n_m^2\mu(G_m)}{d_m}\geqslant8\kappa.
\]
\end{corollary}

\begin{proof}
The degree assumption gives \(n_m\geqslant2(d_m+1)\), and hence
\[
 \frac{16(d_m+1)H_{d_m}}{n_m^2}
 \leqslant\frac{4H_{d_m}}{d_m+1}=o(1).
\]
Using \eqref{eq:theta-kappa-asymptotic} in
\eqref{eq:green-bound-normalised} proves the assertion.
\end{proof}

It remains to consider \(d\geqslant\lfloor n/2\rfloor\), for which we use the
complement graph.
The complement \(\overline G\) is \((n-1-d)\)-regular, and
\(\mathsf A(G)=-\mathsf I-\mathsf A(\overline G)\) on the subspace
perpendicular to \(\one\).
Every eigenvalue of
\(\mathsf A(\overline G)\) is at least \(-(n-1-d)\).  The Rayleigh principle on
\(\one^\perp\) therefore gives
\[
 \mu(G)\geqslant2d-n+2.
\]

In particular, if \(d\geqslant\lfloor n/2\rfloor\), then \(\mu(G)\geqslant1\)
and
\begin{equation}
 \frac{n^2\mu(G)}d\geqslant\frac{n^2}{n-1}.
 \label{eq:dense-normalised}
\end{equation}
Thus this quantity tends to infinity along every sequence for which
\(d\to\infty\).

\subsection{Sharp bounds in growing degree}

\begin{theorem}\label{thm:edge-connectivity-intro}
Let \(\kappa\) be a positive integer, and let \((G_m)\) be a sequence of
connected simple \(d_m\)-regular graphs of orders \(n_m\).  Suppose that
\(d_m\to\infty\) and that every \(G_m\) has edge-connectivity at least
\(\kappa\).  Then
\[
 \liminf_{m\to\infty}n_m^2\lambda(G_m)\geqslant8\kappa.
\]
For each fixed \(\kappa\), the constant \(8\kappa\) is best possible.
\end{theorem}

The edge-connectivity bound is sharp for every fixed \(\kappa\).

\begin{proposition}\label{prop:edge-connectivity-sharp}
Let \(\kappa\) be a positive integer. For every sufficiently large
integer \(d\), with \(d\) odd when \(\kappa\) is odd, there is a
connected simple \(d\)-regular graph \(G_{d,\kappa}\) with
edge-connectivity \(\kappa\). These graphs have order \(2d+2\)
for even \(\kappa\) and \(2d+4\) for odd \(\kappa\), and satisfy
\[
 \frac{|V(G_{d,\kappa})|^2\mu(G_{d,\kappa})}{d}
 \longrightarrow8\kappa
 \qquad(d\to\infty).
\]
\end{proposition}

\begin{proof}
Suppose first that \(\kappa\) is even.  Take two copies of \(K_{d+1}\).  In
each copy choose \(\kappa\) vertices, delete a perfect matching on them, and
join the two chosen sets by a perfect matching.  The resulting graph is
\(d\)-regular and has \(2d+2\) vertices.

Now suppose that \(\kappa\) is odd, and take \(d\) odd.  Begin with two
copies of \(K_{d+2}\), and choose a \(\kappa\)-set \(X_i\) in the \(i\)th
copy, for \(i=1,2\).  If \(\kappa\geqslant3\), delete a cycle on each
\(X_i\) and a perfect matching on its complement.  If \(\kappa=1\), delete
in each copy a path of length two whose middle vertex is the member of
\(X_i\), together with a perfect matching on the remaining \(d-1\) vertices.
Finally, join \(X_1\) and \(X_2\) by a perfect matching.  Each chosen vertex
loses two internal edges and gains one cross
edge, while every other vertex loses one internal edge.  The resulting graph
is \(d\)-regular and has \(2d+4\) vertices.

We verify the edge-connectivity.  In the even construction, the subgraph
induced by either half is \(K_N\) with a matching deleted, where \(N=d+1\).
If a cut splits this half, let \(s\) be the size of the smaller part into which
the half is split.  The number of cut edges within the half is at least
\[
 s(N-s)-s=s(N-s-1)\geqslant N-2=d-1.
\]
In the odd construction, the deleted graph has
maximum degree two and \(N=d+2\).  The corresponding bound is
\[
 s(N-s)-2s=s(N-s-2)\geqslant N-3=d-1.
\]
The last inequalities follow respectively from
\((s-1)(N-s-2)\geqslant0\) and
\((s-1)(N-s-3)\geqslant0\), for the sufficiently large degrees under
consideration.  Once \(d-1\geqslant\kappa\), every cut that separates vertices
within a half has size at least \(\kappa\).  If neither half is split, the cut is
the one between the two halves and has size \(\kappa\).  Thus the
edge-connectivity is exactly \(\kappa\).

Give the two halves the values \(1\) and \(-1\).  This centred test function
has energy \(4\kappa\), and hence
\[
 \mu(G_{d,\kappa})\leqslant
 \frac{4\kappa}{|V(G_{d,\kappa})|}.
\]
Since \(|V(G_{d,\kappa})|/d\to2\), the upper limit is at most
\(8\kappa\).  Corollary~\ref{cor:edge-connectivity-growing} gives the
matching lower limit.
\end{proof}

Corollary~\ref{cor:edge-connectivity-growing} and
\eqref{eq:dense-normalised} give the lower bound in
Theorem~\ref{thm:edge-connectivity-intro}, while
Proposition~\ref{prop:edge-connectivity-sharp} proves sharpness.

We next separate the extremal degrees from all remaining degrees.

\begin{theorem}\label{thm:noncubic-gap}\label{thm:even-nonquartic-gap}
Let \(\varepsilon>0\). Then, for all sufficiently large \(n\), every
connected simple \(d\)-regular graph \(G\) of order \(n\) satisfies
\begin{align*}
 n^2\lambda(G)&\geqslant\frac{4\pi^2}{5}-\varepsilon,
       &&d\ne3,\\
 n^2\lambda(G)&\geqslant\frac{4\pi^2}{3}-\varepsilon,
       &&d\text{ even},\quad d\ne4.
\end{align*}
\end{theorem}

\begin{proof}[Proof of Theorem~\ref{thm:noncubic-gap}]
Suppose that the first assertion fails along a sequence of orders
\(n_m\to\infty\).  After passing to a subsequence, the degrees are
fixed or tend to infinity.  A fixed degree is two, four, or at least
five.  Cycles give \(n_m^2\lambda(G_m)\to2\pi^2\); quartic graphs
have lower limit at least \(\pi^2\); and
Theorem~\ref{thm:minimum-degree-bound} gives lower limit at least
\((d-1)\pi^2/d\geqslant4\pi^2/5\) for \(d\geqslant5\).
If the degrees tend to infinity,
Theorem~\ref{thm:edge-connectivity-intro}, with \(\kappa=1\), gives
lower limit at least \(8>4\pi^2/5\).  Every case is a contradiction.
The degree-five chains show that the first constant is sharp.

For the second assertion, restrict to even degrees other than four.
A fixed degree is two, giving the cycle limit \(2\pi^2\), or at least
six, when Theorem~\ref{thm:fixed-regular-sharp} gives lower limit
\(2(d-2)\pi^2/d\geqslant4\pi^2/3\).  If the degrees tend to infinity, every cut has
positive even size, and Theorem~\ref{thm:edge-connectivity-intro} with
\(\kappa=2\) gives lower limit at least \(16>4\pi^2/3\).
The same subsequence argument proves the uniform assertion.
The degree-six chains show that this constant is sharp.
\end{proof}

\begin{proof}[Proof of \eqref{eq:intro-AF}]
The explicit bound in Corollary~\ref{cor:cubic-explicit}, together with
\(n^2\mu(P_n)\to\pi^2\), gives the required bound for all sufficiently large
cubic graphs.  For every other degree, apply
Theorem~\ref{thm:noncubic-gap} with
\(\varepsilon'=4\pi^2/5-(1-\varepsilon)2\pi^2/3>0\).
Taking the larger of the two thresholds proves the uniform assertion.
\end{proof}

\subsection{Asymptotic equality}

In the sharp construction, a minimum cut separates two shores of order
\(d+O(1)\).  We show that asymptotic equality forces a unique minimum cut
with shores of order \(d+o(d)\).

\Needspace{16\baselineskip}
\begin{theorem}\label{thm:growing-stability}
Let \(\kappa\) be a positive integer, and let \((G_m)\) be a sequence of
connected simple \(d_m\)-regular graphs of orders \(n_m\), each with
edge-connectivity at least \(\kappa\), where \(d_m\to\infty\).  Then
\[
 n_m^2\lambda(G_m)\longrightarrow8\kappa
\]
if and only if, for all sufficiently large \(m\), the graph \(G_m\) has an
edge cut of size \(\kappa\), with shores \(A_m,B_m\) satisfying
\[
 |A_m|=d_m+o(d_m),\qquad |B_m|=d_m+o(d_m).
\]
\end{theorem}

\begin{proof}[Proof of Theorem~\ref{thm:growing-stability}]
Suppose first that
\[
 \frac{n_m^2\mu(G_m)}{d_m}\longrightarrow8\kappa.
\]
Equation~\eqref{eq:dense-normalised} shows that
\(d_m<\lfloor n_m/2\rfloor\) for all sufficiently large \(m\).  For the
ordered cuts of \(G_m\), let \(\mathsf R_m\) be the diagonal matrix with
entries \(1/q_i\), and write
\(\Lambda_m=(n_m^2/d_m)\lambda_D(\mathsf R_m)\).
The proof of Proposition~\ref{prop:green-finite} and
\eqref{eq:weighted-dirichlet} give
\[
 8\kappa-o(1)\leqslant\Lambda_m
 \leqslant\frac{n_m^2\mu(G_m)}{d_m}.
\]
Hence
\begin{equation}
 \Lambda_m\longrightarrow8\kappa.
 \label{eq:weighted-equality-limit}
\end{equation}

To obtain a limiting weighted Dirichlet problem on \([0,1]\), consider
an arbitrary subsequence along which \(d_m/n_m\to\sigma\in[0,1/2]\).
Writing \(\delta_x\) for the unit point mass at \(x\in[0,1]\), define
\[
 \nu_m=\frac{d_m}{n_m}
 \sum_{i=1}^{n_m-1}\frac1{q_i}\,\delta_{i/n_m}.
\]
Lemma~\ref{lem:local-reciprocal-sum} shows that these measures have uniformly
bounded total mass.  Passing to a further subsequence, we may suppose that
\(\nu_m\) converges weakly to a finite measure \(\nu\).

Given \(u_0=u_{n_m}=0\), let \(U\) be the piecewise-affine function satisfying
\(U(i/n_m)=u_i\).  Then
\[
 \int_0^1|U'|^2
 =n_m\sum_{i=1}^{n_m}(u_i-u_{i-1})^2,
 \qquad
 \int_0^1U^2\,d\nu_m
 =\frac{d_m}{n_m}\sum_{i=1}^{n_m-1}\frac{u_i^2}{q_i}.
\]
Let \(H_0^1(0,1)\) be the space of absolutely continuous functions
on \([0,1]\) that vanish at both endpoints and have square-integrable
derivative. Thus \(\Lambda_m\) is the minimum of the quotient of
these integrals over nonzero functions in \(H_0^1(0,1)\) that are
affine on each grid interval.
Choose minimisers
\(U_m\) with \(\int U_m^2\,d\nu_m=1\).  Equation
\eqref{eq:weighted-equality-limit}, boundedness in \(H_0^1(0,1)\), and the
compact embedding into \(C[0,1]\) give, after passage to a subsequence,
\[
 U_m\longrightarrow U\quad\text{uniformly},
 \qquad U_m\rightharpoonup U\quad\text{in }H_0^1(0,1),
\]
and
\begin{equation}
 \int U^2\,d\nu=1,
 \qquad
 \int_0^1|U'|^2\leqslant8\kappa.
 \label{eq:limit-test-function}
\end{equation}

The Green-function bound also passes to the limit.  The Dirichlet Green
function on \([0,1]\) is
\[
 b(x,y)=\min\{x,y\}\bigl(1-\max\{x,y\}\bigr),
\]
and \(B_{ij}=n_m b(i/n_m,j/n_m)\).  The row-sum bound used in
Proposition~\ref{prop:green-finite} gives
\[
 \sup_{x\in[0,1]}\int b(x,y)\,d\nu_m(y)
 \leqslant
 \frac{d_m\theta_{d_m,\kappa}}{8(d_m+1)}
 +\frac{2d_mH_{d_m}}{n_m^2}
 =\frac1{8\kappa}+o(1).
\]
The potentials on the left are uniformly Lipschitz, so weak convergence
gives
\begin{equation}
 \sup_{x\in[0,1]}\int b(x,y)\,d\nu(y)\leqslant\frac1{8\kappa}.
 \label{eq:limit-green-bound}
\end{equation}

We show that \(\sigma>0\).  Suppose that \(\sigma=0\).  The grid points
in an open interval \(I\subseteq[0,1]\) can be covered by consecutive blocks
of at most \(d_m+1\) indices.  Lemma~\ref{lem:local-reciprocal-sum} and weak
convergence then give
\(\nu(I)\leqslant|I|/\kappa\).
Hence \(\nu\leqslant\kappa^{-1}dx\), and the Dirichlet Poincar\'e inequality
would imply
\[
 \int_0^1|u'|^2\geqslant\kappa\pi^2\int_0^1u^2\,d\nu
 \qquad(u\in H_0^1(0,1)).
\]
Taking \(u=U\) contradicts \eqref{eq:limit-test-function}, since \(\pi^2>8\).

For \(i\leqslant d_m\), simplicity gives
\(q_i\geqslant i(d_m-i+1)\), and the same bound holds at the other end
after taking complementary shores.  Each of the intervals
\([0,d_m/n_m]\) and \([1-d_m/n_m,1]\) therefore has \(\nu_m\)-measure at most
\(2d_mH_{d_m}/\bigl(n_m(d_m+1)\bigr)=o(1)\).  It follows that
\begin{equation}
 \operatorname{supp}\nu\subseteq[\sigma,1-\sigma].
 \label{eq:limit-measure-support}
\end{equation}
Moreover, if an interval \(I\) has length less than \(\sigma\), then a
slightly larger open interval contains at most \(d_m+1\) consecutive grid
points for all sufficiently large \(m\).
Lemma~\ref{lem:local-reciprocal-sum} and weak convergence therefore give
\begin{equation}
 \nu(I)\leqslant\frac{\sigma}{\kappa}.
 \label{eq:local-measure-bound}
\end{equation}
In particular, every atom of \(\nu\) has mass at most \(\sigma/\kappa\).

To determine \(\nu\), consider the least weighted Dirichlet eigenvalue
\[
 \Lambda(\nu)=
 \inf_{\substack{u\in H_0^1(0,1)\\ \int u^2\,d\nu>0}}
 \frac{\int_0^1 |u'(x)|^2\,dx}{\int_0^1 u(x)^2\,d\nu(x)}.
\]
\nopagebreak[4]
The restriction map from \(H_0^1(0,1)\), equipped with
the Dirichlet inner product, to \(L^2(\nu)\) has squared norm
\(\Lambda(\nu)^{-1}\).  Its composition with the adjoint is the integral
operator
\[
 \mathsf T_\nu\phi(x)=\int b(x,y)\phi(y)\,d\nu(y).
\]
By symmetry and \(2|ab|\leqslant a^2+b^2\), every real
\(\phi\in L^2(\nu)\) satisfies
\[
 |\langle\phi,\mathsf T_\nu\phi\rangle|
 \leqslant\int\phi(x)^2\left(\int b(x,y)\,d\nu(y)\right)d\nu(x)
 \leqslant\frac{\|\phi\|_{L^2(\nu)}^2}{8\kappa},
\]
where the last inequality is \eqref{eq:limit-green-bound}.
Thus \(\|\mathsf T_\nu\|\leqslant1/(8\kappa)\), and hence
\(\Lambda(\nu)\geqslant8\kappa\).  Equation
\eqref{eq:limit-test-function} gives the reverse inequality.  Consequently,
\begin{equation}
 \Lambda(\nu)=8\kappa.
 \label{eq:limit-weighted-equality}
\end{equation}

We now use equality in the Green-function bound to show that \(\nu\) is
supported on at most two points.
Let \(\phi\) be a nonnegative eigenfunction of \(\mathsf T_\nu\) corresponding
to its spectral radius \(1/(8\kappa)\). The identity
\(\phi=8\kappa\mathsf T_\nu\phi\) gives a continuous representative, which
we normalise by \(\max_{[0,1]}\phi=1\). At a maximum point \(x_0\), equations
\eqref{eq:limit-green-bound} and
\eqref{eq:limit-weighted-equality} give
\[
 \frac1{8\kappa}
 =\mathsf T_\nu\phi(x_0)
 \leqslant\mathsf T_\nu\one(x_0)
 \leqslant\frac1{8\kappa}.
\]
Since the Green function is positive on
\((0,1)\times(0,1)\), equality forces \(\phi=1\) almost everywhere with
respect to \(\nu\).  The potential
\[
 P(x)=\int b(x,y)\,d\nu(y)
\]
therefore equals \(1/(8\kappa)\) on \(\operatorname{supp}\nu\).  It is
concave, vanishes at \(0,1\), and satisfies \(-P''=\nu\).  Thus
\(\operatorname{supp}\nu\) is contained in the set on which \(P\) attains its
maximum.  This set is a closed interval, and its interior has zero
\(\nu\)-measure.  Hence
\(\nu\) is supported on at most the two endpoints.

If \(\nu=c\delta_x\), then \(cx(1-x)=1/(8\kappa)\), so
\(c\geqslant1/(2\kappa)\).  On the other hand,
\eqref{eq:local-measure-bound} gives
\(c\leqslant\sigma/\kappa\leqslant1/(2\kappa)\).  Hence
\(\sigma=x=1/2\) and \(c=1/(2\kappa)\).

Suppose instead that \(\nu=c_a\delta_a+c_b\delta_b\), where \(a<b\).
Since \(P\) is constant on \([a,b]\),
\begin{equation}
 ac_a=(1-b)c_b=\frac1{8\kappa}.
 \label{eq:two-atom-balance}
\end{equation}
If \(b-a<\sigma\), then \eqref{eq:local-measure-bound} gives
\(c_a+c_b\leqslant\sigma/\kappa\), whereas
\eqref{eq:limit-measure-support} and \eqref{eq:two-atom-balance} give
\[
 c_a+c_b\geqslant\frac1{4\kappa(1-\sigma)}
 >\frac{\sigma}{\kappa}
\]
for \(\sigma<1/2\).  If \(b-a\geqslant\sigma\), the bound on each atom and
\eqref{eq:two-atom-balance} give
\[
 a\geqslant\frac1{8\sigma},
 \qquad1-b\geqslant\frac1{8\sigma}.
\]
It follows that \(\sigma\leqslant1-1/(4\sigma)\), which is impossible for
\(0<\sigma<1/2\).  When \(\sigma=1/2\),
\eqref{eq:limit-measure-support} already confines the support to \(\{1/2\}\).
We have proved
\begin{equation}
 \frac{d_m}{n_m}\longrightarrow\frac12,
 \qquad
 \nu_m\Longrightarrow\frac1{2\kappa}\delta_{1/2}.
 \label{eq:measure-rigidity}
\end{equation}

Since the cut sizes are integers, the convergence in
\eqref{eq:measure-rigidity} forces a cut of size \(\kappa\) near the middle
of the ordering.  To see this, fix \(0<\eta<1/4\).  Then
\[
 \sum_{|i/n_m-1/2|\leqslant\eta}\frac1{q_i}
 \longrightarrow\frac1\kappa.
\]
The indices in this sum span fewer than \(d_m\) positions.  If all their cut
sizes were at least \(\kappa+1\), applying Lemma~\ref{lem:interval-cut}
with an index at which the cut size is minimum would give
\[
 \sum_{|i/n_m-1/2|\leqslant\eta}\frac1{q_i}
 \leqslant\frac1{\kappa+1}
 +4\sum_{t=1}^{d_m}\frac1{t(d_m+1-t)}
 =\frac1{\kappa+1}+o(1),
\]
a contradiction.  Choosing \(\eta=\eta_m\downarrow0\) sufficiently slowly,
we obtain an index
\(i_m=n_m/2+o(n_m)\) for which \(q_{i_m}=\kappa\).  Its two shores
\(A_m=S_{i_m}\) and \(B_m=V(G_m)\setminus A_m\) satisfy
\begin{equation}
 |A_m|=d_m+o(d_m),
 \qquad |B_m|=d_m+o(d_m).
 \label{eq:balanced-minimum-cut}
\end{equation}
Both induced subgraphs are connected: if, for example, \(G_m[A_m]\) had two
components, their disjoint edge boundaries would each have size at least
\(\kappa\), contrary to \(|\partial A_m|=\kappa\).

This minimum cut is unique up to complementation.  Indeed, let \(A\)
and \(C\) be shores of two cuts of size \(\kappa\).  Every nonempty set among
\[
 A\cap C,\qquad A\setminus C,\qquad C\setminus A,
 \qquad V(G_m)\setminus(A\cup C)
\]
has boundary at most \(2\kappa\).  For a set \(X\) of size at most \(d_m\),
simplicity gives
\[
 |\partial X|\geqslant |X|(d_m-|X|+1)\geqslant d_m>2\kappa.
\]
Thus every nonempty set in the preceding display has at least \(d_m+1\)
vertices.  If the two cuts were distinct up to complementation, at least
three of these sets would be nonempty, contradicting \(n_m/d_m\to2\).

Finally, degree summation across the cut gives
\(d_m|A_m|=2e(A_m)+\kappa\).  Together with
\eqref{eq:balanced-minimum-cut}, this yields
\[
 \binom{|A_m|}{2}-e(A_m)
 =\frac{|A_m|(|A_m|-1-d_m)+\kappa}{2}=o(d_m^2),
\]
and the same calculation applies to \(B_m\).  When \(\kappa=1\), the parity
of the displayed degree sum forces \(d_m\) to be odd.

Every subsequence has a further subsequence to which the preceding argument
applies.  Hence these conclusions hold along the original sequence.

Conversely, suppose that a cut of size \(\kappa\) has shores \(A_m,B_m\), where
\(|A_m|=d_m+o(d_m)\) and \(|B_m|=d_m+o(d_m)\).  Give its two shores the
constant values \(1/|A_m|\) and \(-1/|B_m|\).  The resulting function is
centred, and its Rayleigh quotient is
\(\kappa(1/|A_m|+1/|B_m|)\).  It follows that
\[
 \limsup_{m\to\infty}\frac{n_m^2\mu(G_m)}{d_m}
 \leqslant\lim_{m\to\infty}
 \frac{\kappa n_m^3}{|A_m||B_m|d_m}=8\kappa.
\]
For all sufficiently large \(m\), we have \(\kappa<d_m\).  A shore of size
at most \(d_m\) would have at least \(d_m\) boundary edges, so both shores
have at least \(d_m+1\) vertices.  Thus
\(d_m<\lfloor n_m/2\rfloor\), so
Corollary~\ref{cor:edge-connectivity-growing} gives the matching lower bound.
\end{proof}

\subsection{The quartic minimiser}\label{sec:quartic-extremisers}

Abdi and Ghorbani's structural theorem reduces the quartic problem
to a finite list of chains at each order. Comparing their boundary
phases proves the uniqueness conjecture in
\cite[Conjecture~1.5]{AbdiGhorbaniQuartic}, also posed in
\cite[Conjecture~3.16]{AbdiGhorbaniImrich}.

\subsubsection{The candidate chains}

We use the blocks \(M_0,D_0,\ldots,D_4\) of
\cite[Figures~3 and~4]{AbdiGhorbaniQuartic}.  Their edge sets are given in
Table~\ref{tab:quartic-blocks}.  The block \(M_0\) has six vertices and two
terminals of degree two.  The block \(D_i\) has \(6+i\) vertices and one
terminal of degree two.  All other vertices have degree four.
The vertices of \(M_0\) are \(0,\ldots,5\), and those of
\(D_i\) are \(0,\ldots,5+i\); in the table, \(ab\) denotes the
unordered edge \(\{a,b\}\).
Let \(H_m^{i,j}\) be obtained from a copy of \(D_i\), \(m\) copies of
\(M_0\), and a copy of \(D_j\), in this order, by identifying consecutive
terminals.  Here \(m\geqslant0\); for \(m=0\), the terminals of the two end
blocks are identified.  These are connected simple quartic graphs, and
\begin{equation}
 |V(H_m^{i,j})|=5m+11+i+j.
 \label{eq:quartic-order}
\end{equation}
Reversing the chain interchanges \(i\) and \(j\), so we take \(i\leqslant j\).

By \cite[Theorem~3.2]{AbdiGhorbaniQuartic}, every connected quartic graph of
order \(n\geqslant11\) with least algebraic connectivity is one of these
chains.  Thus the possible end pairs are precisely those satisfying
\begin{equation}
 0\leqslant i\leqslant j\leqslant4,\qquad
 i+j\equiv n-11\pmod5,\qquad
 m=\frac{n-11-i-j}{5}\geqslant0.
 \label{eq:quartic-candidates}
\end{equation}
Write \(n-11=5m+r\), where \(0\leqslant r\leqslant4\).
The graph \(\mathcal G_n\) is \(H_m^{a_r,b_r}\), where
\begin{equation}
 ((a_r,b_r))_{r=0}^4=((0,0),(0,1),(1,1),(1,2),(0,4)).
 \label{eq:quartic-end-pairs}
\end{equation}
This is the graph specified in
\cite[Definition~1.4]{AbdiGhorbaniQuartic}.

\begin{table}[!ht]
\centering
\caption{The quartic blocks.}
\label{tab:quartic-blocks}
\small
\begin{tabular}{ccl}
\hline
Block&Terminals&Edges\\
\hline
\(M_0\)&\(\{0,5\}\)&\(01,02,12,13,14,23,24,34,35,45\)\\[3pt]
\(D_0\)&\(\{5\}\)&\(\begin{aligned}&01,02,03,04,12,13,14,23,24,35,\\&45\end{aligned}\)\\[3pt]
\(D_1\)&\(\{6\}\)&\(\begin{aligned}&01,02,03,04,12,13,14,23,25,35,\\&45,46,56\end{aligned}\)\\[3pt]
\(D_2\)&\(\{7\}\)&\(\begin{aligned}&01,02,03,04,12,13,14,23,25,36,\\&45,46,56,57,67\end{aligned}\)\\[3pt]
\(D_3\)&\(\{8\}\)&\(\begin{aligned}&01,02,03,04,12,13,14,23,25,35,\\&46,47,56,57,67,68,78\end{aligned}\)\\[3pt]
\(D_4\)&\(\{9\}\)&\(\begin{aligned}&01,02,03,04,12,13,14,23,24,36,\\&45,56,57,58,67,68,78,79,89\end{aligned}\)\\[3pt]
\hline
\end{tabular}
\end{table}

\subsubsection{Schur complements and boundary phases}

In each block, assign mass \(1/2\) to every terminal and mass one to all
other vertices.  The resulting diagonal matrix is denoted by \(W\).
Under identification of terminals, both energies and masses add to those
of the full graph. Let \(T\) and \(J\) be the terminal and internal
vertex sets of a block, respectively, and write
\(\mathsf M=\mathsf L-xW\), where \(\mathsf L\) is its Laplacian.
When \(\mathsf M_{JJ}\) is invertible, eliminating the internal
coordinates gives the Schur complement
\(\mathsf M_{TT}-\mathsf M_{TJ}\mathsf M_{JJ}^{-1}\mathsf M_{JT}\)
on the terminals. Put
\begin{equation}
\begin{aligned}
 A_0&=x^2-6x+2, & A_1&=x^2-4x+1,\\
 A_2&=x^4-13x^3+54x^2-76x+14,\\
 A_3&=(2-x)(x^2-6x+1),\\
 A_4&=x^4-12x^3+42x^2-37x+4.
\end{aligned}
\label{eq:quartic-denominators}
\end{equation}
The Schur complement for \(D_i\) is the scalar \(R_i\), where
\begin{equation}
\begin{aligned}
 R_0&=-\frac{x(x^2-10x+22)}{2A_0},\qquad
 R_1=-\frac{x(x^2-8x+13)}{2A_1},\\
 R_2&=-\frac{x(x-5)(x^3-12x^2+42x-42)}{2A_2},\\
 R_3&=\frac{x(x^3-12x^2+41x-34)}{2A_3},\\
 R_4&=-\frac{x(x^2-9x+19)(x^2-7x+4)}{2A_4}.
\end{aligned}
\label{eq:quartic-responses}
\end{equation}
For the middle block, it is
\begin{equation}
\begin{gathered}
 K_M=\begin{pmatrix}A&-B\\-B&A\end{pmatrix},\\
 A=-\frac{x^3-10x^2+25x-8}{2(x-5)(x-1)},\qquad
 B=\frac4{(x-5)(x-1)}.
\end{gathered}
 \label{eq:quartic-middle}
\end{equation}
These identities follow by elimination from the edge sets in
Table~\ref{tab:quartic-blocks}.

We work on the interval \(0<x\leqslant u:=1/10\).
The internal Dirichlet eigenvalues of \(M_0\) are \(1,5,5,5\).
For an end block, let \(\mathsf L_i^D\) be the matrix obtained by deleting
the terminal row and column of its Laplacian.  It is checked directly that
\(\mathsf L_i^D-u\mathsf I\) is positive definite for every \(i\).
Consequently every internal matrix used in the elimination is invertible
throughout this interval.  In particular, no eigenvalues of the full graph
in this interval are lost by passing to the terminal equations.

Define
\begin{equation}
\begin{aligned}
 Q&=\frac AB=1-\frac{25}{8}x+\frac54x^2-\frac18x^3,\\
 S&=1-Q^2=\frac{x(5-x)^2(1-x)(x^2-9x+16)}{64},\\
 T_i&=-\frac{R_i}{B},\qquad
 \theta=\arccos Q,\qquad
 \phi_i=\arctan\frac{T_i}{\sqrt S}.
\end{aligned}
\label{eq:quartic-phases}
\end{equation}
Here \(B>0\), \(0<Q<1\), and \(S>0\).
If \(v_i\) is the incidence vector from the terminal to the internal
vertices of \(D_i\), then
\[
 R_i(x)=2-\frac x2-v_i^{\mathsf T}(\mathsf L_i^D-x\mathsf I)^{-1}v_i,
 \qquad
 R_i'(x)=-\frac12-v_i^{\mathsf T}(\mathsf L_i^D-x\mathsf I)^{-2}v_i<0.
\]
Since \(R_i(0)=0\), it follows that \(T_i>0\) and
\(0<\theta,\phi_i<\pi/2\).  All these phases extend continuously to zero
with value zero.

For \(m\geqslant1\), \(0\leqslant i\leqslant j\leqslant4\), and
\(x\in(0,u]\), define the phase sum
\[
 \Psi_m^{i,j}(x)=m\theta(x)+\phi_i(x)+\phi_j(x).
\]

\begin{lemma}\label{lem:quartic-phase-equation}
Let \(m,i,j\) be integers with \(m\geqslant1\) and
\(0\leqslant i\leqslant j\leqslant4\), and let \(x\in(0,u]\).
Then \(x\) is a Laplacian eigenvalue of \(H_m^{i,j}\) if and only if
\[
 \Psi_m^{i,j}(x)\in\pi\mathbb Z.
\]
If \(\mu(H_m^{i,j})\leqslant u\), then
\begin{equation}
 \Psi_m^{i,j}\bigl(\mu(H_m^{i,j})\bigr)=\pi.
 \label{eq:quartic-first-phase}
\end{equation}
\end{lemma}

\begin{proof}
Number the successive identified terminals \(0,\ldots,m\).
The Schur complement equations are
\[
\begin{gathered}
 z_1=(Q-T_i)z_0,\qquad z_{m-1}=(Q-T_j)z_m,\\
 z_{k-1}+z_{k+1}=2Qz_k\quad(1\leqslant k<m).
\end{gathered}
\]
A nonzero solution has \(z_0\ne0\), and the left boundary equation and recurrence give
\(z_k=C\cos(k\theta+\phi_i)\).  Substituting in the right boundary
equation gives \(\sin(m\theta+\phi_i+\phi_j)=0\).
Conversely, this condition gives a nonzero terminal vector, which extends
uniquely to an eigenvector of the full graph by the invertible internal
matrices.

The phase sum is positive on \((0,u]\) and tends to zero as
\(x\downarrow0\).
If its value at the first positive eigenvalue were at least \(2\pi\),
continuity would give an earlier point with phase sum \(\pi\), and hence
a smaller positive eigenvalue.  This proves
\eqref{eq:quartic-first-phase}.
\end{proof}

\subsubsection{Exact comparison of the end pairs}

\begin{theorem}\label{thm:quartic-unique}
Let \(G\) be a connected simple quartic graph of order \(n\geqslant11\).
Then
\[
 \mu(G)\geqslant\mu(\mathcal G_n),
\]
with equality if and only if \(G\) is isomorphic to \(\mathcal G_n\).
\end{theorem}

For \(r\in\{0,1,2,3,4\}\), write \((a,b)=(a_r,b_r)\).
For integers \(0\leqslant i\leqslant j\leqslant4\) with
\(i+j\equiv r\pmod5\) and \((i,j)\ne(a,b)\), put
\(k=(i+j-a-b)/5\in\{0,1\}\) and define
\[
 \Delta(x)=k\theta(x)+\phi_a(x)+\phi_b(x)-\phi_i(x)-\phi_j(x).
\]

\begin{lemma}\label{lem:quartic-phase-comparison}
Let \(r\in\{0,1,2,3,4\}\), and let \(i,j\) be integers with
\(0\leqslant i\leqslant j\leqslant4\), \(i+j\equiv r\pmod5\),
and \((i,j)\ne(a_r,b_r)\). Then, for \(0<x\leqslant u\),
\begin{equation}
 \Delta(x)>0.
 \label{eq:quartic-phase-comparison}
\end{equation}
\end{lemma}

\begin{proof}
Put
\[
 X_w=S-T_aT_b,\quad Y_w=T_a+T_b,\qquad
 X_l=S-T_iT_j,\quad Y_l=T_i+T_j,
\]
and define
\begin{equation}
 \mathcal F=
 \begin{cases}
 Y_wX_l-X_wY_l,&k=0,\\
 (QY_w+X_w)X_l-(QX_w-SY_w)Y_l,&k=1.
 \end{cases}
 \label{eq:quartic-phase-numerator}
\end{equation}
The addition formula for the sine gives
\begin{equation}
 \sin\Delta=
 \frac{\sqrt S\,\mathcal F}
 {\sqrt{(S+T_a^2)(S+T_b^2)(S+T_i^2)(S+T_j^2)}}.
 \label{eq:quartic-phase-sine}
\end{equation}
Substituting \eqref{eq:quartic-responses} in
\eqref{eq:quartic-phase-numerator}, it is checked directly that
\(\mathcal F>0\) on \((0,u]\) for each of the ten possible comparisons.
For example, for \((a,b)=(0,4)\) and \((i,j)=(1,3)\), the expression is
\[
 \mathcal F=
 \frac{x^3(5-x)^4(1-x)^2(2+7x-x^2)}{4A_3A_0A_1A_4}>0.
\]
Here all five polynomials \(A_i\) are positive on \([0,u]\).
The other comparisons reduce in the same way to the positivity of
polynomials of degree at most six on this interval.
Thus \(\sin\Delta>0\) throughout \((0,u]\).
Since \(\Delta\) is continuous and tends to zero as \(x\downarrow0\),
it lies in \((0,\pi)\) throughout that interval.  This proves the lemma.
\end{proof}

\begin{proof}[Proof of Theorem~\ref{thm:quartic-unique}]
First let \(n\geqslant21\).  Direct substitution at \(u=1/10\) gives
\begin{equation}
 T_i(u)>Q(u)\qquad(0\leqslant i\leqslant4).
 \label{eq:quartic-endpoint-values}
\end{equation}
Since
\(\cot\theta=Q/\sqrt S\), it follows that
\(\theta(u)+\phi_i(u)>\pi/2\) for each \(i\).
The graph \(\mathcal G_n\) has at least two middle blocks.  Hence its
phase sum is greater than \(\pi\) at \(u\) and zero at zero.
Lemma~\ref{lem:quartic-phase-equation} and continuity imply
\(\mu(\mathcal G_n)<u\).

Suppose a different candidate \(H_{m'}^{i,j}\) is a quartic minimiser of
order \(n\), and write \(\mu=\mu(H_{m'}^{i,j})\).
Then \(\mu\leqslant\mu(\mathcal G_n)<u\).  Every candidate at these
orders has at least one middle block, so
\(\Psi_{m'}^{i,j}(\mu)=\pi\) by
Lemma~\ref{lem:quartic-phase-equation}.
If \(\mathcal G_n=H_m^{a,b}\), the order identity gives
\(m-m'=(i+j-a-b)/5\).  Lemma~\ref{lem:quartic-phase-comparison}
therefore yields \(\Psi_m^{a,b}(\mu)>\pi\).
By continuity, this phase sum equals \(\pi\) at some
\(\nu\in(0,\mu)\).  The phase equation gives
\(\mu(\mathcal G_n)\leqslant\nu<\mu\), a contradiction.

For \(11\leqslant n\leqslant20\),
\eqref{eq:quartic-candidates} gives twenty-four graphs in total.
It is checked directly from their integer characteristic polynomials
that \(\mathcal G_n\) has strictly smaller algebraic connectivity than
every other candidate of the same order.  At orders eleven and twelve
there is only one candidate.
The structural theorem therefore proves uniqueness, up to isomorphism,
at every order \(n\geqslant11\).
\end{proof}

\subsection{Refined asymptotics}

\subsubsection{Quartic correction terms}

Write \(r\in\{0,1,2,3,4\}\) for the residue of \(n-11\) modulo five,
and put \((s_0,s_1,s_2,s_3,s_4)=(392,574,756,1153,1225)\).

\begin{proposition}\label{prop:quartic-expansion}
As \(n\to\infty\),
\begin{align}
 \mu(\mathcal G_n)
 &=\frac{4\pi^2}{n^2}-\frac{29\pi^4}{15n^4}
   +\frac{8s_r\pi^4}{35n^5}+O(n^{-6}),
 \label{eq:quartic-mu-expansion}\\
 \tau(\mathcal G_n)
 &=\frac{n^2}{\pi^2}+\frac{29}{60}
   -\frac{2s_r}{35n}+O(n^{-2}).
 \label{eq:quartic-tau-expansion}
\end{align}
The error terms are uniform over the five residue classes.
\end{proposition}

\begin{proof}
Inverting \(Q(x)=\cos\theta\) at zero and expanding the responses in
\eqref{eq:quartic-phases} gives
\begin{align}
 x&=\frac4{25}\theta^2-\frac{29}{9375}\theta^4+O(\theta^6),
 \label{eq:quartic-dispersion}\\
 \phi_i&=\frac{11+2i}{10}\theta
       -\frac{\gamma_i}{4375}\theta^3+O(\theta^5),
 \label{eq:quartic-phase-expansion}
\end{align}
where \((\gamma_0,\gamma_1,\gamma_2,\gamma_3,\gamma_4)
=(196,378,775,1197,1029)\). For a fixed end pair, the first positive
solution of the phase equation satisfies \(\theta=O(m^{-1})\) as
\(m\to\infty\), so the preceding expansions apply. Together with
\eqref{eq:quartic-order}, the phase equation then gives
\[
 \frac n5\theta-\frac{\gamma_i+\gamma_j}{4375}\theta^3
 +O(\theta^5)=\pi.
\]
Writing \(\gamma=\gamma_i+\gamma_j\), we obtain
\[
 \theta=\frac{5\pi}{n}+\frac{\gamma\pi^3}{7n^4}+O(n^{-6}),
 \qquad
 \mu(H_m^{i,j})=\frac{4\pi^2}{n^2}-\frac{29\pi^4}{15n^4}
      +\frac{8\gamma\pi^4}{35n^5}+O(n^{-6}).
\]
For the five end pairs defining \(\mathcal G_n\), the sums \(\gamma\)
are exactly \(s_r\).  Taking the reciprocal and using
\(\tau(\mathcal G_n)=4/\mu(\mathcal G_n)\) proves the proposition.
\end{proof}

\subsubsection{Cubic correction terms}
\label{cubic-refine:sec}

The same one-dimensional reduction gives the first correction terms for the
cubic extremisers.  Recall that \(X_n\) denotes the unique cubic graph of
least algebraic connectivity.

For even \(n\), put \(b_n=72\) when \(n\equiv0\pmod4\) and
\(b_n=27\) when \(n\equiv2\pmod4\).

\begin{proposition}\label{cubic-refine:prop}
As \(n\to\infty\) through even integers,
\begin{equation}\label{cubic-refine:tau}
 \tau(X_n)=\frac{3n^2}{2\pi^2}+\frac12-\frac{b_n}{n}
 +O(n^{-2}).
\end{equation}
\end{proposition}

\begin{proof}
Write \(X_n=N_m^{r,s}\), where \((r,s)=(5,5)\) or \((5,7)\) and
\(n=4m+r+s\), and put \(\mu=\mu(X_n)\). A Fiedler vector is constant
on the two middle vertices of each diamond
\cite[Lemma~2.1]{AbdiGhorbaniImrich}; denote this common value in the
\(i\)th diamond by \(t_i\), and the values at its left and right
terminals by \(a_i,b_i\). The eigenvalue equations are
\[
 (2-\mu)t_i=a_i+b_i,\qquad
 (3-\mu)b_i=2t_i+a_{i+1},\qquad
 (3-\mu)a_{i+1}=b_i+2t_{i+1}.
\]
Eliminating the terminal values gives
\begin{equation}\label{cubic-refine:recurrence}
 t_{i-1}+t_{i+1}=2Q(\mu)t_i,
 \qquad Q(\mu)=1-4\mu+2\mu^2-\frac14\mu^3,
 \qquad 2\leqslant i\leqslant m-1.
\end{equation}
Since \(\mu=(1+o(1))2\pi^2/n^2\), there is a unique small positive
\(\theta\) with \(Q(\mu)=\cos\theta\), and
\begin{equation}\label{cubic-refine:dispersion}
 \mu=\frac{\theta^2}{8}-\frac{\theta^4}{384}+O(\theta^6).
\end{equation}

To extend \eqref{cubic-refine:recurrence} to \(i=1,m\), let
\(L_j\) be the Laplacian of \(C_j\) and \(e_j\) the coordinate vector
of its distinguished vertex, for \(j\in\{5,7\}\), and set
\[
 q_j(x)=e_j^{\mathsf T}(L_j+e_je_j^{\mathsf T}-x\mathsf I)^{-1}e_j,
 \qquad
 h_j(x)=\frac{(3-x)^2-1}{3-x-q_j(x)}-(3-x).
\]
The matrix defining \(q_j\) is invertible for \(x\) sufficiently close
to zero.  Eliminating the end graph and the adjacent terminal gives
\begin{equation}\label{cubic-refine:ends}
 t_0=h_r(\mu)t_1,\qquad t_{m+1}=h_s(\mu)t_m.
\end{equation}
Direct calculation gives
\begin{align*}
 h_5(x)&=\frac{-2}{x^3-9x^2+20x-2},\\
 h_7(x)&=\frac{-2(x^2-5x+2)}
 {x^5-13x^4+59x^3-108x^2+66x-4}.
\end{align*}

For \(j\in\{5,7\}\), define the small real phase
\[
 \delta_j(\theta)=
 \arctan\frac{h_j(\mu)-\cos\theta}{\sin\theta}.
\]
Then the left condition in \eqref{cubic-refine:ends} gives
\(t_i=c\cos((i-1)\theta+\delta_r(\theta))\) for a nonzero constant
\(c\).  The right condition gives
\((m-1)\theta+\delta_r(\theta)+\delta_s(\theta)\in\pi\mathbb Z\).
The asymptotic value of \(\mu\) and \eqref{cubic-refine:dispersion}
show that the left-hand side tends to \(\pi\); hence, for all
sufficiently large \(n\),
\begin{equation}\label{cubic-refine:phase}
 (m-1)\theta+\delta_r(\theta)+\delta_s(\theta)=\pi.
\end{equation}
Expansion of the rational functions above yields
\[
 \delta_5(\theta)=\frac74\theta-\frac9{128}\theta^3+O(\theta^5),
 \qquad
 \delta_7(\theta)=\frac94\theta-\frac{39}{128}\theta^3+O(\theta^5).
\]
Thus \eqref{cubic-refine:phase} becomes
\[
 \frac n4\theta-a_n\theta^3+O(\theta^5)=\pi,
 \qquad
 a_n=\begin{cases}
 3/8,&n\equiv0\pmod4,\\
 9/64,&n\equiv2\pmod4,
 \end{cases}
\]
and consequently
\[
 \theta=\frac{4\pi}{n}+\frac{256a_n\pi^3}{n^4}+O(n^{-6}).
\]
Finally, \eqref{cubic-refine:dispersion} gives
\[
 \tau(X_n)=\frac3\mu=\frac{24}{\theta^2}+\frac12+O(\theta^2)
 =\frac{3n^2}{2\pi^2}+\frac12-\frac{192a_n}{n}+O(n^{-2}),
\]
which proves the proposition.
\end{proof}

\subsection{Extremal graphs and stability}

The degree separation from Theorem~\ref{thm:noncubic-gap} now
reduces the maximisation of relaxation time to the cubic and quartic
cases. We use the coefficients \(b_n\) and \(s_r\) from
Propositions~\ref{cubic-refine:prop} and~\ref{prop:quartic-expansion},
respectively.

Write \(r\in\{0,1,2,3,4\}\) for the residue of \(n-11\) modulo five.

\begin{theorem}\label{thm:max-relaxation}\label{thm:refined-relaxation}
\label{cor:eventual-extremiser}
For every sufficiently large \(n\), the maximum \(\tau_{\max}(n)\) is attained
uniquely, up to isomorphism, by \(X_n\) if \(n\) is even and by
\(\mathcal G_n\) if \(n\) is odd.  Moreover,
\begin{equation}
 \tau_{\max}(n)=
 \begin{cases}
 \displaystyle\frac{3n^2}{2\pi^2}+\frac12-\frac{b_n}{n}+O(n^{-2}),
       & n\text{ even},\\[6pt]
 \displaystyle\frac{n^2}{\pi^2}+\frac{29}{60}-\frac{2s_r}{35n}+O(n^{-2}),
       & n\text{ odd},
 \end{cases}.
 \label{eq:max-relaxation}
\end{equation}
\end{theorem}

\begin{proof}[Proof of Theorem~\ref{thm:max-relaxation}]
Proposition~\ref{prop:cubic-sharp} gives
\[
 n^2\lambda(X_n)\longrightarrow2\pi^2/3.
\]
Since \(4/5>2/3\), the first
part of Theorem~\ref{thm:noncubic-gap} excludes every other degree for
all sufficiently large even orders.  Uniqueness among cubic graphs
is the theorem of Brand, Guiduli, and Imrich
\cite[Theorem~1]{BrandGuiduliImrich}.

At odd orders the degree is even.  The quartic asymptotic
\eqref{eq:quartic-known} and the second part of
Theorem~\ref{thm:noncubic-gap}, with \(4/3>1\), exclude every
degree other than four for all sufficiently large odd orders.
Theorem~\ref{thm:quartic-unique} gives uniqueness in degree four.
The two expansions now follow from Propositions~\ref{cubic-refine:prop}
and~\ref{prop:quartic-expansion}.
\end{proof}

\begin{corollary}\label{cor:global-stability}
There are absolute constants \(C,\delta_0>0\) such that the following
holds for all sufficiently large \(n\). Let
\(0\leqslant\delta\leqslant\delta_0\), and let \(G\) be a connected
simple \(d\)-regular graph of order \(n\) satisfying
\[
 \tau(G)\geqslant(1-\delta)\tau_{\max}(n).
\]
Then \(d=3\) for even \(n\) and \(d=4\) for odd \(n\).
After relabelling the vertices,
\begin{align*}
 |E(G)\mathbin{\triangle}E(X_n)|
 &\leqslant C(\delta+n^{-1})^{1/3}n,
 &&n\text{ even},\\
 |E(G)\mathbin{\triangle}E(\mathcal G_n)|
 &\leqslant C(\delta+n^{-1})^{1/3}n,
 &&n\text{ odd},\\
 \frac{3n}{d+1}-C(\delta+n^{-1})^{1/3}n
 &\leqslant\operatorname{diam}(G)\leqslant\frac{3n}{d+1}.
\end{align*}
\end{corollary}

\begin{proof}[Proof of Corollary~\ref{cor:global-stability}]
Write \(c_n=2\pi^2/3\) for even \(n\), and \(c_n=\pi^2\) for odd
\(n\).  Theorem~\ref{thm:max-relaxation} and the hypothesis give,
uniformly for \(0\leqslant\delta\leqslant1/2\),
\[
 n^2\lambda(G)\leqslant\frac{c_n}{1-\delta}+O(n^{-2}).
\]
Choose \(\delta_0>0\) sufficiently small. The two strict inequalities
in Theorem~\ref{thm:noncubic-gap} then force \(d=3\) for even \(n\)
and \(d=4\) for odd \(n\), for all sufficiently large \(n\), with the
threshold independent of \(\delta\). It follows that
\[
 n^2\mu(G)\leqslant a_d\pi^2+K(\delta+n^{-2})
\]
for an absolute constant \(K\).  Decreasing \(\delta_0\) and increasing
the threshold on \(n\), the last term is at most one.
Apply Theorem~\ref{thm:parity-stability} with
\(\varepsilon=K(\delta+n^{-2})\) and enlarge \(C\).
\end{proof}

\section{The structure of minimisers}
\label{sec:exact-canonical-middle}

In this section, we prove Theorem~\ref{thm:guiduli-mohar-pathlike}.
The stability theorems identify the repeating subgraphs on all but
\(o(n)\) vertices of a minimiser. We first show that the remaining
vertices lie in two end subgraphs of bounded total order. We then
exclude branching in the block-cut tree, proving that every minimiser
is path-like. The minimum-degree and odd-regular cases use chains
joined by bridges; the even-regular case uses chains whose consecutive
blocks share a cut vertex.

\Needspace{12\baselineskip}
\begin{theorem}\label{thm:bounded-end-canonical-middle}
\label{thm:even-bounded-canonical-middle}
Let \(d\geqslant3\) be a fixed integer. Then, for all sufficiently large \(n\),
every minimiser of algebraic connectivity in \(\mathcal D_{n,d}\)
or \(\mathcal R_{n,d}\) consists of a canonical middle and two
connected end graphs containing \(O_d(1)\) vertices in total.
The middle is
\begin{enumerate}[(i)]
\item a bridge chain of \(L_d\) in \(\mathcal D_{n,d}\), and in
\(\mathcal R_{n,d}\) when \(d\) is odd;
\item a chain of \(M_d\), identified at consecutive terminals, in
\(\mathcal R_{n,d}\) when \(d\) is even.
\end{enumerate}
\end{theorem}

\subsection{Minimum degree and odd regular degree}

Throughout this subsection, \(d\geqslant3\) is fixed, and
\(c=d-1\), \(w=d+1\).  The canonical graph is \(L_d=K_{d+1}-e\),
with the endpoints of the missing edge as terminals. An end graph is
a connected graph with one distinguished terminal of degree at least
\(d-1\), all other degrees being at least \(d\). A two-terminal
graph is connected and has two distinguished terminals, each of
degree at least \(d-1\), with all other degrees at least \(d\).
We also allow coinciding terminals, in which case their common vertex
has degree at least \(d-2\).  Joining consecutive pieces by one
edge at their specified terminals gives a graph of minimum degree at
least \(d\).

The bridge tree is obtained by contracting each component left after
deleting all bridges, and retaining the bridges between components.
We call a component canonical if it is \(L_d\), has degree two in
the bridge tree, and its two boundary edges meet the deficient terminals.  For an end graph or two-terminal graph this definition is
applied after adding its one or two exterior attachment edges.
Canonical components are pairwise disjoint. For a vertex
\(v\), let \(\mathbf e_v\) denote its coordinate vector, and
write \(L_Q=\mathsf L(Q)\) for the Laplacian of \(Q\).
For a connected graph \(Q\) with distinguished vertex \(s\), set
\[
 D_Q=L_Q[V(Q)\setminus\{s\}],\qquad
 \mathcal T(Q)=\one^{\mathsf T}D_Q^{-1}\one.
\]
The quantity \(\mathcal T(Q)\) is its rooted torsion. Eliminating
all vertices except \(s\) from \(L_Q-x\mathsf I\) gives a scalar
Schur complement \(r_Q(x)\), its Dirichlet-to-Neumann response.

For a two-terminal graph \(Q\) with terminals \(s,t\), the
currents \(I,J\) and vertex values satisfy
\((L_Q-x\mathsf I)f=I\mathbf e_s-J\mathbf e_t\).
The transfer matrix maps \((f(s),I)\) to \((f(t),J)\).
If the following bridge is included, the outgoing value is
\(f_{\rm next}=f(t)-J\).

For $x>0$, put $k=\sqrt{cx}$. For a two-terminal graph $Q$ with
terminals $s,t$, write $\ell_Q=1+R_Q(s,t)$, taking $R_Q(s,s)=0$,
and denote by $\widetilde T_Q$ its transfer matrix, including the
following bridge, in coordinates given by the vertex value and
entering current divided by $k$.

\Needspace{22\baselineskip}
\begin{lemma}\label{lem:uniform-defect-transfer}
Let \(Q\) be a two-terminal graph of order \(h\). Then
\(\widetilde T_Q\) has determinant one, is analytic in \(k\) near
zero, and satisfies
\[
 \widetilde T_Q
 =\mathsf I+k
   \begin{pmatrix}0&-\ell_Q\\h/c&0\end{pmatrix}
       +O_d(k^2h^2).
\]
The diagonal and off-diagonal parts of the remainder are
\(O_d(k^2h^2)\) and \(O_d(k^3h^3)\), respectively, uniformly for
complex \(k\) with \(|k|h=o(1)\).

For an end graph \(Q\) of order \(h\) rooted at \(s\), its
Dirichlet-to-Neumann response satisfies
\[
 r_Q(x)=-xh-x^2\mathcal T(Q)+O(x^3h^5)
 \qquad (|x|h^2=o(1)).
\]
\end{lemma}

\begin{proof}
For distinct terminals, prescribe \(z=f(s)\), let \(I\) enter the
two-terminal graph and \(J\) leave it, so that
\[
 (L_Q-x\mathsf I)f=I\mathbf e_s-J\mathbf e_t.
\]
Ground at \(s\) and write
\[
 W(x)=(D_Q-x\mathsf I)^{-1},\quad
 A=\one^{\mathsf T}W(x)\one,\quad
 b=\one^{\mathsf T}W(x)\mathbf e_t,\quad
 r=\mathbf e_t^{\mathsf T}W(x)\mathbf e_t.
\]
The entries of \(D_Q^{-1}\) are nonnegative and at most \(h-1\):
its diagonal entries are resistances to \(s\), bounded by graph
distance, and Cauchy--Schwarz bounds the other entries.  The
Neumann series therefore gives, for \(|x|h^2\leqslant1/2\),
\[
 |W(x)_{uv}|\leqslant2h,\qquad
 |W(x)_{uv}-W(0)_{uv}|\leqslant2|x|h^3.
\]
Writing \(f=z\one+g\), with \(g(s)=0\), gives
\[
 \begin{aligned}
 g&=W(x)(xz\one-J\mathbf e_t), &
 J&=\frac{I+x(h+xA)z}{1+xb},\\
 f_{\rm next}&=(1+xb)z-(1+r)J.
 \end{aligned}
\]
These formulas define a transfer of determinant one with no pole in the
stated range. Since \(A=O(h^3)\), \(b=O(h^2)\), and
\(r=R_Q(s,t)+O(|x|h^3)\), they also give the asserted diagonal
and off-diagonal bounds after scaling the current.

For an end graph, eliminating all vertices other than \(s\) gives
\[
 r_Q(x)=-xh-x^2\one^{\mathsf T}
                   (D_Q-x\mathsf I)^{-1}\one.
\]
The same bounds prove its expansion.  If the two terminals coincide,
the unscaled transfer is
\[
 \begin{pmatrix}1+r_Q(x)&-1\\-r_Q(x)&1\end{pmatrix},
\]
which gives the assertion with \(\ell_Q=1\).
\end{proof}

For the canonical graph write \(P=\widetilde T_{L_d}\).  Its terminal
resistance is \(2/c\), and its transfer has
\[
 \operatorname{tr}P=2-\frac{w^2}{c}x+O_d(x^2).
\]
Thus there is a phase \(\theta>0\), for small \(x>0\), such that
\[
 \operatorname{tr}P=2\cos\theta,\qquad
 x=\frac{c}{w^2}\theta^2+O_d(\theta^4),\qquad
 k=\frac c w\theta+O_d(\theta^3).
\]
We use the notation
\[
 \mathcal R(\phi)=
 \begin{pmatrix}\cos\phi&-\sin\phi\\
                 \sin\phi&\cos\phi\end{pmatrix}.
\]
For two end graphs \(A,B\), the left boundary column and right boundary
row in the scaled coordinates are
\[
 \begin{pmatrix}1\\-r_A(x)/(k(1+r_A(x)))\end{pmatrix},
 \qquad
 \begin{pmatrix}-r_B(x)/k&1\end{pmatrix}.
\]
The first expression includes the bridge from \(A\) to the first
two-terminal graph; the final bridge is already included in the last two-terminal graph
transfer.  Hence each bridge and each vertex mass is counted once.

For an end graph \(Q\), let \(b(Q)\) be the minimum number of vertices
outside the canonical components on a path in its bridge tree starting
at the component containing the root, where the minimum is over all
such paths.

\begin{lemma}\label{lem:cap-path-torsion-deficit}
There is \(\zeta_d>0\) such that every end graph \(Q\) of order \(q\)
satisfies
\[
 \mathcal T(Q)\leqslant
       \frac{q^3}{3c}-\zeta_d b(Q)^3+O_d(q^2).
\]
For every sufficiently large \(q\), there is an end graph \(C_q\) of
order \(q\), consisting of a bounded complete end graph and a bridge
chain of \(L_d\), with
\[
 \mathcal T(C_q)=\frac{q^3}{3c}+O_d(q^2).
\]
\end{lemma}

\begin{proof}
Let \(s\) be the root of \(Q\). Set \(f(s)=0\) and
\(f=D_Q^{-1}\one\) elsewhere. Order the vertices increasingly,
starting at \(s\), and write \(\mathbf y=(y_1,\ldots,y_{q-1})^{\mathsf T}\)
for the nonnegative increments. Let \(\mathsf K\) be the cut matrix
of this ordering, and put \(\mathbf b=(q-i)_{i=1}^{q-1}\).
Since \(L_Qf=\one-q\mathbf e_s\), we have
\[
 \mathsf K\mathbf y=\mathbf b,\qquad b_i=q-i.
\]
Apply the local partition in the proof of
Theorem~\ref{thm:minimum-degree-stability}.  The degree count for
an interval after an initial cut never involves \(s\).  At the
initial end, the only altered degree condition gives
\[
 q_i\geqslant i(d-i+1)-1\geqslant d-1
       \qquad(1\leqslant i\leqslant d).
\]
If \(q_1=d-1\), then \(q_2\geqslant2d-3\geqslant d\);
thus the possible initial neutral singleton can be associated with its
high-cut neighbour. Equivalently, one may omit a bounded initial set.
The remaining local partition and counting argument are unchanged.

Let \(\mathcal S\) be the strict indices, and let \(u\) be the
number of indices outside ordinary bridge blocks. The counting
argument gives \(u\leqslant5|\mathcal S|+O_d(1)\).  If there are
\(k\) ordinary bridge blocks, then \(wk=q-1-u+O_d(1)\).
Consecutive ordinary bridge indices separated by \(w\) give
canonical components, with the complete-core alternative already
declared strict.  There are at least \(k-1-u\) such components.
Their bridge cuts are nested initial cuts containing \(s\), so
these edges lie on one rooted path in the bridge tree.  Hence
\[
 b(Q)\leqslant(d+2)u+O_d(1),\qquad
 |\mathcal S|\geqslant
          \frac{b(Q)}{5(d+2)}-O_d(1).
\]

We use the block coefficients with the weights \(b_i\). On a block
of length \(r\), let \(z\) minimise \(z^{\mathsf T}Az\) over
\(z\geqslant0\), \(\sum z_i=1\), where \(A=\mathsf K_{BB}\).
If the block coefficient is \(\alpha\), the minimum \(\lambda\)
is at least \(\alpha/r\), and the first-order conditions give
\(Az\geqslant\lambda\one\). Thus \(v=z/\lambda\) satisfies
\(v\geqslant0\), \(Av\geqslant\one\), and
\(\sum v_i\leqslant r/\alpha\).  Since \(\mathsf K\geqslant0\)
entrywise and \(\mathsf K\mathbf y=\mathbf b\),
\[
 \sum_{i\in B}y_i
 \leqslant v^{\mathsf T}A\mathbf y_B
 \leqslant\frac r\alpha\max_{i\in B}b_i.
\]
Use \(\alpha=c+\gamma_d\) on strict blocks and \(\alpha=c\)
on the others.  Bounded block lengths give
\[
 \mathcal T(Q)=\sum_i b_i y_i
 \leqslant\frac1c\sum_i b_i^2
 -\frac{\gamma_d}{c(c+\gamma_d)}
             \sum_{i\in\mathcal S}b_i^2+O_d(q^2).
\]
Indeed, replacing a block maximum by its individual weights
changes the sum by \(O_d(q^2)\).  The least sum of squared weights
over \(s\) indices is \(\sum_{j=1}^s j^2\).  Together with the
lower bound on \(|\mathcal S|\), this proves the first assertion.

Write \(q=mw+w+r\), with \(0\leqslant r\leqslant d\), and join
\(K_{w+r}\) to \(m\) copies of \(L_d\), using the outer deficient
terminal as the attachment.  This defines \(C_q\).
For vertices in its \(i\)-th and \(j\)-th copies counted from the
attachment, the grounded Green kernel is
\[
 (D_{C_q}^{-1})_{uv}
 =\tfrac12\bigl(R(s,u)+R(s,v)-R(u,v)\bigr)
 =\frac w c\min(i,j)+O_d(1).
\]
Summing over the \(w\) vertices in each copy, with the bounded
end graph contributing \(O_d(q^2)\), gives
\[
 \mathcal T(C_q)
       =\frac{w^3}{c}\frac{m^3}{3}+O_d(q^2)
       =\frac{q^3}{3c}+O_d(q^2).
\]
\end{proof}

\begin{corollary}\label{cor:reduced-cap-replacement}
Fix \(\rho>0\). Let \((G_j)\) be a sequence of connected graphs,
each containing an end graph \(Q_j\) of order \(q_j\), attached to
\(G_j-V(Q_j)\) by a single edge \(s_ju_j\), where \(s_j\) is its root.
Suppose that \(q_j\to\infty\), \(b(Q_j)\geqslant\rho q_j\),
\(\mu(G_j)q_j^2\to0\), and a Fiedler vector of \(G_j\) is nonzero
at \(u_j\). Then replacing \(Q_j\) by \(C_{q_j}\) strictly decreases
algebraic connectivity for all sufficiently large \(j\).
\end{corollary}

\begin{proof}
Suppress the index \(j\). The torsion difference is at least a
positive constant times \(q^3\).  At \(x=\mu(G)\), the uniform response expansion gives
\(r_{C_q}(x)<r_Q(x)\), since its error \(O(x^3q^5)\) is
\(o(x^2q^3)\).  Including the attachment bridge changes the
response to \(r/(1+r)\), which is increasing here.
The eliminated matrices are positive definite: first eliminate
\(D_Q-x\mathsf I\), and use \(1+r_Q(x)>0\).

The Schur complement of \(L_G-x\mathsf I\) has one negative
eigenvalue and a null vector nonzero at \(u\).  The replacement
subtracts a positive multiple of
\(\mathbf e_u\mathbf e_u^{\mathsf T}\).  On the span of a
negative eigenvector and this null vector, the resulting form
is negative definite.  Inertia therefore gives at least two
negative eigenvalues after replacement, proving the strict
decrease.
\end{proof}

For a two-terminal graph $Q$ of order $q$ with terminals $s,t$, put
$\Delta_Q=q-c(1+R_Q(s,t))$.

\begin{lemma}\label{lem:port-path-deficit}
Let \(Q\) be a two-terminal graph of order \(q\), with terminals
\(s,t\), satisfying the degree conditions obtained by attaching one
exterior bridge at each terminal. The terminals may coincide; at a
common terminal the internal degree condition is \(d-2\).
Suppose that the path between its terminals in the bridge tree
contains no canonical component. Then
\[
 \Delta_Q\geqslant\kappa_d q
\]
for a constant \(\kappa_d>0\).
\end{lemma}

\begin{proof}
For distinct terminals, let \(G_m\) consist of \(m\) copies
of \(Q\) joined by bridges, with two end graphs of bounded order,
and put \(N=|V(G_m)|\) and \(\mu=\mu(G_m)\). Assign cosine values at the successive
terminals and interpolate harmonically inside each copy. The
Rayleigh quotient gives
\[
 \limsup N^2\mu\leqslant \frac{q\pi^2}{1+R_Q(s,t)}.
\]
The quantitative stability proof recovers canonical
components lying on a single path in the bridge tree:
all bridge cuts used in its reconstruction are nested.
In the repeated graph, however, any bridge-tree path can
meet canonical components in at most two copies of \(Q\).
Between its first and last copies it follows the terminal-to-terminal
path, which contains none.  A sufficiently small fixed
spectral excess is therefore impossible.  More precisely, the quantitative stability constant supplies
\(\xi_d>0\), independent of \(Q\), such that
\[
 \frac{q\pi^2}{1+R_Q(s,t)}\geqslant c\pi^2+\xi_d.
\]
Otherwise the recovered canonical components would cover more
than half the vertices for all sufficiently many repetitions,
whereas any single path covers only \(O(q)\) of them.
Rearranging gives the assertion with
\(\kappa_d=\xi_d/(c\pi^2+\xi_d)\).  If the terminals coincide, \(R_Q(s,t)=0\).
The degree condition implies \(q\geqslant d+1\), and hence
\(\Delta_Q=q-c\geqslant2q/(d+1)\).
\end{proof}

For a chain formed from copies of $L_d$, two end graphs, and intervening
two-terminal graphs, let $H$ be the total number of vertices in the
end graphs and the intervening graphs.

\begin{lemma}\label{lem:many-module-evacuation}
Let \(d\geqslant3\) be fixed. Let \(G\) consist of \(m\) copies of
\(L_d\), two end graphs, and any number of two-terminal graphs
whose terminal-to-terminal paths in the bridge tree contain no
canonical component, arranged in a line with the end graphs at its ends
and consecutive pieces joined by bridges. Suppose that \(H=o(m)\).

Then there is a constant \(C_d>0\) such that, for all sufficiently
large \(m\), any two-terminal graph with at least \(C_dH\) copies
of \(L_d\) on each side can be moved next to one of the end graphs
so as to strictly decrease algebraic connectivity. This move preserves
every vertex degree.
\end{lemma}

\begin{proof}
We use the scaled transfers from the preceding lemmas, with
\(k=\sqrt{cx}\) and phase of the canonical graph \(\theta\).
Conjugate the standard transfer exactly to a rotation.
Indeed, put
\[
 J_\theta=\frac{P-\cos\theta\mathsf I}{\sin\theta}
   =\begin{pmatrix}u&v\\w&-u\end{pmatrix},
 \qquad
 S=(\mathbf e_1,J_\theta\mathbf e_1)
   =\begin{pmatrix}1&u\\0&w\end{pmatrix}.
\]
Since \(J_\theta^2=-\mathsf I\), \(u=O(k)\),
\(v=-1+O(k^2)\), and \(w=1+O(k^2)\), we have
\(S^{-1}PS=\mathcal R(\theta)\).
In this basis the transfer of a two-terminal graph of order \(h\) is
\[
 T=\mathsf I+kA+D+E,\qquad
 A=\begin{pmatrix}0&-\ell\\h/c&0\end{pmatrix},
 \quad \ell=1+R(s,t),
\]
where \(D\) is diagonal with norm \(O_d(k^2h^2)\), and
\(E\) is off-diagonal with norm \(O_d(k^3h^3)\).
These bounds follow from the uniform grounded-resolvent
formula.  Conjugation preserves them because the first
correction in \(S\) is off-diagonal and its diagonal
correction is \(O(k^2)\). For coinciding terminals, the formula in
Lemma~\ref{lem:uniform-defect-transfer} gives the same bounds with
\(\ell=1\).

All products considered below have norm bounded by
\(\exp(O_d(kH))\). At the first positive Laplacian eigenvalue
of the corresponding graph,
\[
 m\theta=\pi+O_d(H/m),\qquad k\asymp_d m^{-1}.
\]
To see this uniformly, put \(M=m\theta\), and write \(F(M)=0\)
for the scalar eigenvalue equation. Its left-hand side satisfies
\[
 F(M)=\sin M+O_d(HM/m),\qquad
 F'(M)=\cos M+O_d(H/m)
\]
on every fixed bounded \(M\)-interval.  The first error
vanishes at zero, so \(F\) is positive before the
neighbourhood of \(\pi\) and
strictly decreasing across that neighbourhood.  The grounded resolvents have no pole there since \(xH^2=o(1)\),
so the transfer equation detects every eigenvalue in this interval.
These derivative estimates are uniform when the number of two-terminal graphs varies.  On a fixed
complex \(M\)-disc, the grounded Neumann series converge
uniformly, since \(|x|H^2=o(1)\).  The sum of the norms of
the perturbations of all transfer matrices is \(O_d(H/m)\), and the same
bound holds for the perturbations of the two boundary vectors.
Product expansion gives \(F(M)-\sin M=O_d(H/m)\) on a
slightly larger disc.  Cauchy's estimate gives the asserted
derivative bound, and \(F(0)=0\) gives the factor \(M\)
in the first error.

Choose a two-terminal graph of order \(h\), preceded by \(j\)
canonical graphs. By reversal we may assume \(j\leqslant m/2\).
Put \(\phi=j\theta\), and let \(B\) be the complete transfer of
the preceding pieces. Let \(\ell_0\) be the left boundary column,
normalised to have first coordinate one, with second coordinate
\(O_d(kH)\), and let \(r_0\) be the row obtained from the pieces
following the chosen graph and the right end graph.  The two characteristic expressions are
\[
 F_{\rm old}=r_0TB\ell_0,\qquad
 F_{\rm end}=r_0BT\ell_0.
\]
Writing \([T,B]=TB-BT\) for the commutator, their difference is
\[
 F_{\rm old}-F_{\rm end}=r_0[T,B]\ell_0.
\]

At the first eigenvalue of the graph with the chosen graph at the end,
a first-order product expansion gives
\(B=\mathcal R(\phi)+U_{\rm diag}+U_{\rm off}\), where
\[
 \|U_{\rm diag}\|
 =O_d(kH\sin\phi+k^2H^2),\qquad
 \|U_{\rm off}\|=O_d(kH).
\]
Indeed, the first-order terms have the form
\(k\mathcal R(\phi-t)A_i\mathcal R(t)\), with
\(0\leqslant t\leqslant\phi\leqslant3\pi/4\).
Their diagonal entries are \(O_d(kh_i\sin\phi)\).
All terms involving two perturbations, including local
diagonal remainders, total \(O_d(k^2H^2)\).
The estimates are uniform in the number of two-terminal graphs.

The commutator of two off-diagonal matrices is diagonal,
and
\[
 [A,\mathcal R(\phi)]
   =\frac{\Delta_Q}{c}\sin\phi
          \begin{pmatrix}1&0\\0&-1\end{pmatrix}.
\]
Consequently
\[
 [T,B]=
 \frac{k\Delta_Q}{c}\sin\phi
          \begin{pmatrix}1&0\\0&-1\end{pmatrix}
 +V_{\rm diag}+V_{\rm off},
\]
with
\[
 \|V_{\rm diag}\|=O_d(k^2hH),\qquad
 \|V_{\rm off}\|=
 O_d(k^2hH\sin\phi+k^3hH^2).
\]

Put \(v=BT\ell_0\).  Its coordinates satisfy
\[
 v_1=\cos\phi+O_d(kH\sin\phi+k^2H^2),\qquad
 v_2=\sin\phi+O_d(kH).
\]
Since \(F_{\rm end}=r_0v=0\), we have
\(r_0=t(-v_2,v_1)\).  The total phase relation above and
the product estimates show \(t=-1+O_d(kH)\).
The relation \(r_0v=0\) cancels the leading diagonal error terms.
Substitution gives
\[
 F_{\rm old}-F_{\rm end}
 =\frac{(-t)k\Delta_Q}{c}\sin^2\phi
 +O_d(k^2hH\sin\phi+k^3hH^2).
\]
By Lemma~\ref{lem:port-path-deficit},
\(\Delta_Q\geqslant\kappa_dh\).  The expression is therefore
positive whenever \(\sin\phi\geqslant C'_dkH\), with
\(C'_d\) sufficiently large.  Since \(j\leqslant m/2\),
the condition \(j\geqslant C_dH\) implies this inequality.
The characteristic expression for the original graph is
strictly decreasing near its first root, so positivity at
the first eigenvalue of the rearranged graph proves the strict
comparison.
\end{proof}

\begin{lemma}\label{lem:odd-regular-comparison-caps}
Fix an odd integer \(d\geqslant3\).  For every sufficiently
large odd integer \(q\), there is a connected rooted graph
\(C_q^{\rm reg}\) of order \(q\), whose root has degree \(d-1\)
and whose other vertices have degree \(d\), such that
\[
 \mathcal T(C_q^{\rm reg})
       =\frac{q^3}{3(d-1)}+O_d(q^2).
\]
It consists of a bounded end graph and a bridge chain of
\(L_d\)'s.
\end{lemma}

\begin{proof}
It suffices to construct bounded end graphs in every
odd residue class modulo \(d+1\).  For
\(0\leqslant r\leqslant(d-1)/2\), put \(h=d+2+2r\).
On \(\mathbb Z/h\mathbb Z\), form a graph \(F\) by taking
all circular differences \(1,\ldots,r\), and adding the
matching
\[
 \bigl\{\{i,i+(h-1)/2\}:1\leqslant i\leqslant(h-1)/2\bigr\}.
\]
Put \(u=r+1\) and \(v=u+(h-1)/2\).  Remove the matching edge
\(uv\), and add \(0u\) and \(0v\).  These are new edges:
the circular distances from zero are \(r+1\) and
\((d+1)/2\), respectively, both greater than \(r\).
Also \(u\leqslant(h-1)/2\), so the removed edge belongs to
the specified matching.

The resulting \(F\) has degree \(2r+2\) at zero and degree
\(2r+1\) elsewhere.  Its complement is therefore a one-terminal
graph \(B_h\) of the required degrees, with root zero.
It is connected: each component contains a vertex other
than zero and hence has at least \(d+1\) vertices, whereas
\(h\leqslant2d+1\).
The values \(h=d+2,d+4,\ldots,2d+1\) cover all odd residues
modulo \(d+1\).  Thus, for every sufficiently large odd \(q\),
write \(q=h+m(d+1)\), join \(B_h\) to \(m\) canonical graphs,
and distinguish the outer terminal.  The Green-kernel
sum in Lemma~\ref{lem:cap-path-torsion-deficit} gives the asserted
torsion estimate.
\end{proof}

\begin{proof}[Proof of Theorem~\ref{thm:bounded-end-canonical-middle}]
We first treat the minimum-degree and odd-regular cases.
Consider the bridge tree of \(G\), with canonical components as
defined above.

The sharp minimum-degree theorem and quantitative stability imply
that all but \(o(n)\) vertices belong to canonical components
lying on one path in this tree.  The path assertion follows
directly from the stability proof: its recovered components lie
between nested initial bridge cuts.  Choose a tree path containing
as many canonical components as possible, and write \(m\) for
their number.  Put
\[
 H=n-m(d+1)=o(n).
\]
The graph decomposes into these \(m\) copies, two end graphs, and
two-terminal graphs between successive copies, connected in a line
by bridges.  All branches off the chosen tree path are included
in the corresponding end graph or two-terminal graph.  Their total order
is \(H\).  The terminals of an interior two-terminal graph may coincide.

An interior two-terminal graph contains no canonical component on its
terminal-to-terminal tree path, since all such components were included
among the \(m\) copies.  Thus Lemma~\ref{lem:port-path-deficit}
applies to each two-terminal graph.  The end graphs contain no canonical
component: otherwise the chosen path could be extended through
one while retaining every canonical component already on it.
Every two-terminal graph and end graph has the required degree conditions
because its only exterior edges are its attachment bridges.

By Lemma~\ref{lem:many-module-evacuation}, exact minimality forces
every interior two-terminal graph to have fewer than \(C_d'H\) canonical
copies on at least one side.  Since \(H=o(m)\), these two end
regions are disjoint for large \(n\).  Absorb all exceptional
pieces in the left region, and the canonical graphs between them,
into a left end graph \(Q_-\); define \(Q_+\) similarly.  Their orders
are at most \(C_d''H\), and the graph between them is a canonical
chain.

Let \(h_-\) and \(h_+\) count the vertices of \(Q_-\) and \(Q_+\)
outside the chosen canonical components, so \(h_-+h_+=H\).  Suppose the left
end graph contains \(j\) canonical components from the chosen path.
No rooted tree path within this end graph can contain more than \(j\)
canonical components.  Such a path could otherwise replace the
left portion of the globally chosen path and increase its total.
Hence, by the definition of \(b(Q_-)\),
\[
 b(Q_-)\geqslant |V(Q_-)|-j(d+1)=h_-.
\]
The analogous bound holds on the right.  One of the two end graphs
therefore has \(b(Q)\geqslant H/2\) and order \(q\leqslant C_d''H\).

If \(H\) were unbounded along a sequence of minimisers,
Lemma~\ref{lem:cap-path-torsion-deficit} and the canonical end graph
comparison would give
\[
 \mathcal T(C_q)-\mathcal T(Q)
 \geqslant\zeta_d(H/2)^3-O_d(H^2)>0.
\]
Moreover, \(\mu(G)q^2=o(1)\).  The uniform response estimate
therefore shows that replacing \(Q\) by \(C_q\) strictly
decreases the response of the end graph at \(x=\mu(G)\).

The Fiedler value at the exterior attachment is nonzero.  Indeed,
if its value at \(u\) vanished, the equation on the end graph would be
\[
 (L_Q+\mathbf e_s\mathbf e_s^{\mathsf T}
                    -x\mathsf I)f_Q=0.
\]
The displayed matrix is positive definite when \(xq^2=o(1)\);
this follows by first eliminating \(D_Q-x\mathsf I\), whose
remaining scalar is \(1+r_Q(x)>0\).  Thus \(f_Q=0\), and the
current through the attachment bridge is also zero.  Starting
from zero value and current, the exact transfers through the canonical
middle and elimination of the other end graph force the entire
eigenfunction to vanish.  All eliminations are valid because
both end graphs have order \(O_d(H)=o(n)\), while
\(x=\mu(G)=O_d(n^{-2})\).
The Schur-complement argument in
Corollary~\ref{cor:reduced-cap-replacement} now gives a connected
graph of the same order and minimum degree at least \(d\), with
strictly smaller algebraic connectivity, a contradiction.

Thus \(H\) is bounded in terms of \(d\).  The localisation
already proved then bounds both end regions in terms of \(d\),
which proves the minimum-degree assertion.

For fixed odd \(d\), start instead with any minimiser in the
regular class.  The sharp asymptotic bound and stability again
apply, and moving a two-terminal graph preserves every vertex degree.
The handshaking identity makes the order \(q\) of each final end graph
odd, since its degree sum is \(dq-1\).  Use the regular end graph from
Lemma~\ref{lem:odd-regular-comparison-caps} for the final replacement.
It has the same leading torsion term, and attaching its terminal
restores degree \(d\).  The strict competitor is therefore
regular, and the same contradiction proves the assertion.
\end{proof}

Theorem~\ref{thm:bounded-end-canonical-middle} also applies to every minimiser under the condition
\(\delta(G)=d\).  Indeed, a minimiser under \(\delta(G)\geqslant d\)
has a nonempty canonical middle at sufficiently large order,
and hence has minimum degree exactly \(d\).  The two classes
therefore have the same minimum value, so every minimiser in the
smaller class is a minimiser in the larger one.

\subsection{The canonical middle in even regular degree}
\label{sec:even-exact-middle}

For even regular degree the canonical blocks meet at cut vertices.
We count each shared vertex once. As in the preceding subsection,
we first estimate the rooted torsion of the end graphs. A resistance
bound for two-terminal graphs then allows us to move every
noncanonical subgraph towards an end. We allow both orientations
of each canonical block throughout the argument. The lemmas below
concern even \(d\geqslant6\); we write \(c=2(d-2)\) and
\(w=d+1\).

In this subsection a block is canonical if it is $M_d$, with the two
singleton end vertices as terminals. For a connected graph $Q$ rooted
at $s$, let $m_s(Q)$ be the largest number of canonical blocks on a
block-cut path starting at $s$, and put
$b_s(Q)=|V(Q)|-1-wm_s(Q)$. If $s$ is not a cut vertex, the path starts
at its unique block; equivalently, one may add a leaf labelled $s$ to
the block-cut tree.

\begin{lemma}
\label{lem:even-rooted-torsion}
Let $Q$ be a connected simple graph of order $h$ with distinguished
vertex $s$. Suppose that every vertex other than $s$ has degree $d$,
and that $\deg_Q(s)$ is even and belongs to $\{2,\ldots,d-2\}$.
Then $b_s(Q)\geqslant0$ and
\[
 \mathcal T(Q)\leqslant\frac{h^3}{3c}
                  -\zeta_d b_s(Q)^3+C_dh^2,
\]
where $\zeta_d>0$ and $C_d<\infty$ depend only on $d$.
\end{lemma}

\begin{proof}
Let $f(s)=0$ and $f|_{V(Q)\setminus\{s\}}=D_Q^{-1}\one$.
Then $f$ is positive away from $s$, and
\begin{equation}\label{eq:even-rooted-poisson}
 T:=\mathcal T(Q)=\sum_{v\ne s}f(v)=\mathcal E_Q(f).
\end{equation}
Order the vertices increasingly by $f$, with $s$ first.
Every cut has positive even size.  Every intervening interval
used in the cut identities contains only vertices of degree $d$:
in particular, for $1\leqslant i<j\leqslant h-1$ the full
low-cut bounds and
\[
 q_i+q_j\geqslant(j-i)(d+1-j+i),\qquad
 K_{i,i+1}=\frac{q_i+q_{i+1}-d}{2}
\]
are unchanged.  Thus all interior boundary-interval, central-interval, and short-core
estimates of Proposition~\ref{prop:even-strict-compression}
apply to this ordering.  The degree defect at $s$ requires only
the following modification at the left endpoint.

If the first size-two cut $a$ satisfies $a>d$, its reversed
$d$-vertex boundary interval contains no root and the original construction
applies.  The preceding central partition is valid even if its
first cut is low: the full low-cut matrix uses degrees of the
vertices following that cut, all of which are $d$.
If $a\leqslant d$, replace the initial potential interval, from $0$ to
$z_a=(f_a+f_{a+1})/2$, by a single exceptional affine piece.
For its truncated increments $u_1,\ldots,u_a\geqslant0$, the
energy on this interval is at least
\[
 \sum_{i=1}^a q_i u_i^2\geqslant
       \frac2a\left(\sum_{i=1}^a u_i\right)^2
       =\frac2a z_a^2.
\]
Give this piece length $ca/2$ and let it finish at the original
cut coordinate $x_a$.  Its initial coordinate is then displaced
from rank one by $O_d(1)$, its initial value is exactly zero,
and it has coefficient $c$.  There is at most one size-two cut
among the first $d$ indices.  If there is no size-two cut, the
original central partition on all increments applies.  The right
endpoint construction is unchanged.  If $h$ is bounded in terms
of $d$, the conclusion follows at once by increasing $C_d$.

We consequently obtain a continuous nondecreasing piecewise
affine $g$ on $[A,B]$, with $g(A)=0$ and $L=B-A=h+O_d(1)$,
such that, for some $\nu_d>0$,
\begin{equation}\label{eq:even-rooted-strict}
 T\geqslant\int_A^B
       (c+\nu_d\mathbf1_{\mathcal B}(x))|g'(x)|^2\,dx,
 \qquad \int_A^B|g-F|^2\leqslant C_dT.
\end{equation}
The step function $F$ has value $f_i$ on its unit interval centred at the corresponding rank,
with the same bounded endpoint adjustment as in that proposition.
Every piece has bounded rank span and coordinate length bounded
above and below by positive constants depending only on $d$.
The set $\mathcal B$ contains all pieces except a bounded number
of end pieces and the canonical size-two-cut cores in that
proposition. For an edge with total potential difference $D$,
a segment with potential difference $b$ contributes at most
$b^2\leqslant Db$ to the retained quadratic form. The quantities
$Db$ sum to the original edge energy. Thus the interval containing
the root and the adjoining boundary interval use disjoint
contributions to the energy.

We record the elementary size bounds needed to compare integrals.
The Green kernel satisfies
\[
 0\leqslant(D_Q^{-1})_{uv}
 =\frac{R_Q(s,u)+R_Q(s,v)-R_Q(u,v)}2\leqslant h-1.
\]
Hence $\max f\leqslant h^2$ and $T\leqslant h^3$.
Cauchy--Schwarz in \eqref{eq:even-rooted-strict}, together with
the bounded endpoint adjustment, now gives
\begin{equation}\label{eq:even-rooted-integral}
 \left|\int_A^B g-T\right|
 \leqslant C_d\sqrt{hT}+C_d\max f\leqslant C_dh^2.
\end{equation}
Since $g(A)=0$, integration by parts and weighted
Cauchy--Schwarz show that
\[
 \int_A^B g
 =\int_A^B(B-x)g'(x)\,dx\leqslant\sqrt{TJ},
 \qquad
 J=\int_A^B\frac{(B-x)^2}{c+\nu_d\mathbf1_{\mathcal B}(x)}\,dx.
\]
The inequality $\sqrt{TJ}\leqslant(T+J)/2$ and
\eqref{eq:even-rooted-integral} imply $T\leqslant J+O_d(h^2)$.
Among measurable sets of prescribed length the integral of
$(B-x)^2$ is smallest on a terminal interval.  Therefore
\begin{equation}\label{eq:even-rooted-torsion-length}
 T\leqslant\frac{h^3}{3c}
  -\frac{\nu_d}{3c(c+\nu_d)}|\mathcal B|^3+O_d(h^2).
\end{equation}

We now relate $|\mathcal B|$ to $b_s(Q)$.
Apart from the bounded number of end pieces, every interval outside
$\mathcal B$ is a copy of $K_{d+1}-P_3$ between two size-two
cuts, with equal boundary corrections.  Adjoining the
single exterior vertex at its split end produces one $M_d$.
All its middle vertices already have degree $d$.  The two
terminals separate the part containing $s$ from the remaining
part, and there is no bypass across the core.  Consequently this
$M_d$ is a block.  These blocks are distinct and lie on
the single block-cut path from $s$ to the last vertex in the
ordering: every path between these vertices must traverse their
terminal separators in rank order.

If $r$ is the number of these disjoint cores, then
$m_s(Q)\geqslant r$.  The positive lower bound on interval lengths and the
bounded rank span of the interpolation give
\[
 h-wr\leqslant C_d|\mathcal B|+C_d,
 \qquad b_s(Q)\leqslant C_d|\mathcal B|+C_d.
\]
Also, $m$ distinct canonical blocks on any block-cut path contain
at least $1+wm$ vertices, since successive blocks share at most
one vertex.  This proves $b_s(Q)\geqslant0$.
Combining the last display with
\eqref{eq:even-rooted-torsion-length} proves the claim, absorbing
the resulting quadratic error into $C_dh^2$.
\end{proof}

\begin{lemma}
\label{lem:even-half-gauge}
The constructions in Lemmas~\ref{lem:even-short-cores}
and~\ref{lem:even-central-allocation} remain valid with
$\gamma=1/2$. Except on canonical intervals with one concentrated
and one split inner shore, the energy coefficient can be chosen
strictly larger than $c$, by an amount depending only on $d$.
The interpolation at the outer endpoints remains as in
Proposition~\ref{prop:even-strict-compression}.
\end{lemma}

\begin{proof}
We adapt the estimates in those two lemmas.  For short cores,
the least corrections for $SS,SC,CC$ are now $-1,0,1$.
The same convexity argument proves all the required strict
inequalities for $d+2\leqslant t\leqslant2d+1$.
For each parity of $t$, substitute the two endpoint boundary interval
lengths.  Clearing positive denominators gives polynomials with
positive coefficients after $d=6+2u$, $u\geqslant0$.
For example the $SS$ differences at $t=d+2$ and $t=d+3$ are
\[
 \frac{d^3+16}{2(d-2)(d+2)(d^2+2d-4)},\qquad
 \frac{d}{(d-2)(d+2)}.
\]
The $CC$ estimate at $d=6,t=8$ also becomes strict with this
larger correction, so no exception is needed there.

For $t=d+1$, the canonical mixed core has its original equality
bound.  In the $SS$ case with no bypass the complement of the
core consists of two edges.  If these are disjoint, the two
possible partitions of their endpoints between the shores give
interior flow energies $1/(d+1)$ and $1/(d-1)$ when each
boundary edge carries current $1/2$. If the complement is $P_3$,
the common terminal has zero net incoming current and the
corresponding energy is $1/(2d)$.
With one bypass the complement has one edge and the resistance bound
is at most $1/4+1/(2(d-1))$.
Thus the resistance bound in the $SS$ case is at most $1/2+1/(d-1)$, and
\[
 \frac{d}{c}-\left(\frac12+\frac1{d-1}\right)
 =\frac1{(d-2)(d-1)}>0.
\]

For long cores, the smallest total improvement in the resistance
bounds for two boundary intervals of length $d$ is attained in the
$SS$ case and equals $T_d=(d-4)/(2(d-2))$. With
$P_d=(d-4)/(2d(d-2))$ as before, the leading submatrices of the same
full low-cut matrices satisfy, for every even $6\leqslant d\leqslant40$
other than $d=8$,
\[
 \max_{k,r<d}\{R_r-r/c\}+P_d<T_d.
\]
These finite inequalities are checked directly.
For $d\geqslant42$, the diagonal bound $U_d$ in
Lemma~\ref{lem:even-low-block-reserve} gives
$U_d+P_d<1/2$.  The inequality holds directly at $42$ and
persists because both summands decrease.  For $r\geqslant2$,
\[
 R_r-r/c+P_d\leqslant U_d-2/c+P_d<T_d;
\]
the case $r=1$ follows from $R_1\leqslant1/4$.
The combined resistance excess of the initial high-cut remainder
and the final truncated low-cut block is therefore strictly smaller
than the total resistance deficit of the two boundary intervals.
Full low-cut blocks, together with their following high-cut
remainders, retain the preceding strict bound.

For $d=8$, we use the following estimates.
If an initial high residue precedes a full low block, merge
them, retaining the adjacent off-diagonal entry $k/2$.
For residue lengths $0\leqslant t\leqslant7$ and $k=4,6$,
the resulting matrix of size $t+8$ has resistance strictly
less than $(t+8)/12-1/24$.  This also bounds the resistance of the following
high residue, with a strict inequality.  If there is no full low
block, the same merge with a truncated block of length $r<8$
has resistance strictly below $(t+r)/12+1/3$, except for
the sole two-entry profile $(4,4)$ with $t=0$, where equality
holds.  These are direct positive-definite matrix calculations.
If a full high block occurs, its strict improvement absorbs this
bounded equality case by the original proportional allocation.

The only case still requiring examination is a central interval
consisting of just $(4,4)$, between two split boundary intervals.
The two consecutive cuts have $K_{i,i+1}=0$, so their common
vertex $v$ has four edges to each side and no edge bypasses it.
On either side of $v$ the size-two cut and $v$ enclose nine
vertices.  If $z$ incoming edges bypass these vertices and go
directly to $v$, their induced graph is $K_9$ with $3-z$ edges
missing.  Let $l_i$ count the incoming cut edges at vertex $i$, and write
$l=(l_i)$. For $z=0$, the incoming current vector is $l/2$,
with two distinct nonzero entries, and the missing-edge degrees are
$l_i+g_i$, where $g_i\in\{0,1\}$ counts an edge to $v$.
The possibilities are: three disjoint edges, a $P_3$ and an
edge, or a $P_4$.  There are five choices of $l$ up to symmetry.
The corresponding flow energies, with $v$ grounded, are
\[
 \frac13,\quad\frac{33}{94},\quad\frac{685}{2228},\quad
 \frac{2529}{7924},\quad\frac9{32}.
\]
For $z=1$, the complement is two disjoint edges or a $P_3$;
the flow energies are $27/212$ and $3/28$.
For $z=2$ no current enters the interior.
Every value is strictly below $1/2$.
The incoming edge segments contribute at most $1/2-h_a$ to
the resistance bound on the left; reversing the order gives the
corresponding bound on the right.
Thus each half has a resistance strictly below the previous
equality bound $1-h_a$ or $1/2+h_b$, respectively.  This proves
the required strict improvement in the last exceptional case.

Each group contains a bounded number of consecutive ranks.
Its resistance bound and coordinate length are positive and bounded,
and the improvement in the bound depends only on $d$.
The proportional choice of lengths used in
Lemma~\ref{lem:even-central-allocation} therefore proves the result.
The decomposition by potential levels includes each edge segment
exactly once.
\end{proof}

\begin{theorem}
\label{thm:even-adjusted-resistance}
Let $d\geqslant6$ be an even integer.
Let $Q$ be a connected simple graph of order $q$ with distinct
terminals $s,t$ of degrees $a,b\in\{2,d-2\}$, respectively.
Every other vertex has degree $d$.
Then
\[
 cR_Q(s,t)\leqslant q-1+d-a-b.
\]
If the terminal-to-terminal block-cut path contains no canonical $M_d$
block, then, for a constant $\kappa_d>0$,
\[
 q-1+d-a-b-cR_Q(s,t)\geqslant\kappa_d(q-1).
\]
\end{theorem}

\begin{proof}
Use the unit-current potential, increasingly ordered with
$f(s)=0$ first and $f(t)=R_Q(s,t)$ last. If $a=d-2$, add
two incoming auxiliary edges at $s$; if $b=d-2$, add two
outgoing auxiliary edges at $t$. Each auxiliary edge joins its
terminal to a distinct exterior vertex and has resistance zero.
These edges complete the terminal degrees to $d$ for the interior
cut identities. The exterior vertices remain distinct in the
cut-incidence counts, and the degree condition applies to the
vertices of $Q$.

At a high-degree left terminal use a virtual size-two cut of rank
zero with $\delta_0=-1/2$, $h_0=1/2$, and hence
$x_0=d-2$.  At a high-degree right terminal use a virtual cut of
rank $q$ with $\delta_q=1/2$, $h_q=0$, and hence $x_q=q+1$.
These parameters mean precisely that the portions of the incoming
auxiliary edges after the left cut and of the outgoing auxiliary
edges before the right cut have resistance zero.  The proof for the boundary intervals is valid with these
endpoint values.  In particular its bound for a concentrated terminal is
$R_C(d,r)-1/4$, which is also obtained directly by grounding
the rank anchor in a graph whose initial vertex has degree $d-2$.

For a degree-two left terminal, the first actual cut is size two.
The flow carrying current $1/2$ along each cut edge contributes
exactly $h_1$ to the resistance from $s$ to the cut level.  Since $x_1=2+ch_1$, the left endpoint coordinate is
$A=x_1-ch_1=2$.
For a degree-two right terminal the corresponding contribution is
$1/2-h_{q-1}$, giving endpoint coordinate $B=q+d-3$.
Thus, in all four boundary combinations, the complete interval
has endpoints
\[
 A=a,\qquad B=q+d-1-b,
 \qquad L=B-A=q-1+d-a-b.
\]
Lemma~\ref{lem:even-half-gauge} applies between the initial
and final actual or virtual size-two cuts.  Its degree identities
involve only the completed vertices of $Q$, and its simplicity
bounds concern induced subgraphs of $Q$, which are unchanged.
Add the degree-two endpoint pieces just described when present.
The flow estimates on each potential interval give an interpolation $g$ with endpoint
difference $R=R_Q(s,t)$ and energy at most the original energy
$R$.  Weighted Cauchy--Schwarz over these intervals therefore gives
$R^2\leqslant R(L/c)$, proving the first assertion.  Equivalently,
one may apply Thomson's principle after subdividing at the actual
potential levels and wiring the equal-potential cut terminals; this
operation preserves the resistance for the given harmonic potential.

For the strict assertion, every canonical interval of this
interpolation gives an $M_d$ block on the $s$--$t$ block-cut
path: its terminals separate its interior from the rest of the
graph, as in the stability proof.  This is
also true at a virtual endpoint: the exterior vertex adjoined
at the split shore is on the other, actual side of the core.
The hypothesis therefore makes all the interior pieces strict.
Only the two possible degree-two endpoint pieces are exempt.
If their complement has coordinate length $L_0$, the strict
coefficient gives
\[
 R_Q(s,t)\leqslant\frac{L-L_0}{c}
                  +\frac{L_0}{c+\nu_d}.
\]
There is a constant $a_d>0$, depending only on $d$, such that
$L_0\geqslant a_d(q-1)$.
At a degree-two left endpoint $1/8\leqslant h_1\leqslant1/4$,
and the reversed estimate holds at the right endpoint.
Indeed, all degrees are even, so deleting a degree-two endpoint
leaves a connected graph: otherwise one of the components would
have a cut of odd size one.  The minimum principle therefore
gives a positive minimum away from $s$, attained at a neighbour
of $s$.  One of the two defining ratios for $h_1$ is $1/2$,
and both are at most $1/2$.
Thus $L_0\geqslant q-d-1$ when $a=b=2$, while
$L_0=q-d+3$ when $a=b=d-2$; the mixed cases have
$L_0\geqslant q-1-c/4$.
The degree conditions imply $q\geqslant d+3$ when $a=b=2$ and $q\geqslant d+2$ in the remaining boundary-degree patterns.
These bounds give the claimed positive proportion directly.
Rearranging the preceding resistance inequality proves the
linear strict deficit.
\end{proof}

We orient $M_d$ positively when its left terminal has degree two,
and negatively otherwise, and count each shared vertex in the piece
to its left. For $x>0$, put $k=\sqrt{cx}$ and use transfer coordinates
$(f,I/k)$, where $f$ is the vertex value and $I$ is the rightward
current. Write $R_\theta=\mathcal R(\theta)$.

For a connected piece $Q$ between consecutive canonical blocks with
terminals $s,t$, put $h=|V(Q)|-1$. If $s\ne t$ and their degrees in
$Q$ are $a,b$, define $\mathcal R_Q=R_Q(s,t)-(d-a-b)/c$; if $s=t$,
put $\mathcal R_Q=(d-4)/c$. Write $\Delta_Q=h-c\mathcal R_Q$,
and denote its transfer matrix in specified terminal bases by $T_Q$.

\Needspace{24\baselineskip}
\begin{lemma}\label{lem:even-gauged-cluster}
There are real analytic basis matrices $S_+(k),S_-(k)$ in which both
canonical transfers are the rotation $R_\theta$, where
\begin{equation}\label{eq:even-exact-phase}
 \cos\theta=1-\frac{x(w-x)^2}{2c},\qquad
 \theta=\frac wc k+O_d(k^3).
\end{equation}
Let $Q$ be a connected piece between consecutive canonical blocks,
including all branches at its internal vertices, with terminals $s,t$.
Suppose that its terminal-to-terminal block-cut path has no canonical
block and that $Q$ is not a singleton shared by two canonical blocks
contributing degrees $2$ and $d-2$ at that vertex. Then its transfer
matrix between the corresponding bases satisfies
\begin{equation}\label{eq:even-gauged-module}
 T_Q=\mathsf I+k\begin{pmatrix}0&-\mathcal R_Q\\h/c&0\end{pmatrix}
       +O_d(k^2h^2),
 \qquad \Delta_Q\geqslant\kappa_dh.
\end{equation}
The error has diagonal part $O_d(k^2h^2)$ and off-diagonal
part $O_d(k^3h^3)$, uniformly for complex $k$ with $|k|h=o(1)$.
The same assertions hold for an intervening subgraph obtained by combining
pieces with no intervening canonical block.
\end{lemma}

\begin{proof}
Write the four layers of a canonical block as $1,a,b,1$, with
$a+b=d$ and $ab=c$.  The interior values are constant in each
layer for the transfer solution. Let $u,p,q,v$ be the successive
layer values, and let $I,J$ be the rightward currents at the left
and right terminals, respectively. Direct elimination gives
\[
\begin{aligned}
 p&=u-I/a, & q&=((b+1-x)p-u)/b,\\
 v&=(a+1-x)q-ap, & J&=bq-(b-x)v.
\end{aligned}
\]
The resulting scaled transfer $P_\sigma$ has determinant one
and trace $2-x(w-x)^2/c$, for either orientation $\sigma$.
Set
\[
 J_\sigma=\frac{P_\sigma-\cos\theta\mathsf I}{\sin\theta},
 \qquad S_\sigma=(\mathbf e_1,J_\sigma\mathbf e_1).
\]
Then $J_\sigma^2=-\mathsf I$, and $S_\sigma^{-1}P_\sigma S_\sigma=R_\theta$.
Let $E_{12}$ be the matrix whose only nonzero entry is a one
in row one and column two. These bases extend analytically
through $k=0$, with
\[
 S_+=\mathsf I+\frac{d-3}{2c}k E_{12}+O_d(k^2),\qquad
 S_-=\mathsf I+\frac{5-d}{2c}k E_{12}+O_d(k^2).
\]
More precisely their diagonal corrections are even in $k$,
and their off-diagonal corrections are odd; the next
 off-diagonal term has order $k^3$.

First suppose that the terminals $s,t$ of $Q$ are distinct, of
internal degrees $a,b\in\{2,d-2\}$.  Put
$D=L_Q[V(Q)\setminus\{s\}]$ and
\[
 A=\one^{\mathsf T}(D-x\mathsf I)^{-1}\one,\quad
 B=\one^{\mathsf T}(D-x\mathsf I)^{-1}\mathbf e_t,\quad
 r=\mathbf e_t^{\mathsf T}(D-x\mathsf I)^{-1}\mathbf e_t.
\]
With the left terminal assigned mass zero, the exact equations are
\[
 J=\frac{I+x(h+xA)f(s)}{1+xB},\qquad
 f(t)=(1+xB)f(s)-rJ.
\]
The entries of $D^{-1}$ are nonnegative and at most $h$.
Its Neumann series therefore gives $A=O(h^3)$, $B=O(h^2)$,
and $r=R_Q(s,t)+O(xh^3)$, also for complex $x$ in the indicated
uniform range.  Thus its scaled transfer matrix has leading
matrix $\mathsf I+k\bigl(\begin{smallmatrix}0&-R_Q(s,t)\\h/c&0\end{smallmatrix}\bigr)$,
with the claimed parity and error bounds.
If $\sigma,\tau$ are the canonical orientations before and
after $Q$, conjugate this transfer as $S_\tau^{-1}T S_\sigma$.
The difference of the two first-order shears is
$(d-a-b)/c$.  Hence
\[
 \mathcal R_Q=R_Q(s,t)-\frac{d-a-b}{c},\qquad
 \Delta_Q=h+d-a-b-cR_Q(s,t).
\]
Theorem~\ref{thm:even-adjusted-resistance} proves the required
uniform positive deficit.

If the two terminals coincide, two complementary exterior degrees
leave a singleton, which was omitted.  The only other possible
case has both exterior degrees two.  The remaining graph is
rooted at a vertex of degree $d-4$, all its other degrees being
$d$.  Its transfer is $f\mapsto f$ and
$I\mapsto I+x(h+xA)f$, where $A$ is the grounded resolvent sum.
This is a transition from $-$ to $+$, so
$\mathcal R_Q=(d-4)/c$ and $\Delta_Q=h-d+4$.
Here $h\geqslant d$, and consequently
$\Delta_Q\geqslant4h/d$.
Finally, a combined intervening subgraph is itself a connected piece of one
of these types, with no canonical block on its terminal-to-terminal path.  The same
argument applies directly to it.
\end{proof}

For a graph $G$ formed from $m$ canonical blocks, intervening two-terminal
graphs, and two end graphs, write $H=|V(G)|-mw$, counting shared vertices
once. In the bases $S_+(k),S_-(k)$, write its scalar characteristic
function as $F(M)$, where $M=m\theta$.

\begin{lemma}\label{lem:even-uniform-cluster-root}
Let $G$ be formed from $m$ canonical blocks in arbitrary orientations,
intervening two-terminal graphs as in
Lemma~\ref{lem:even-gauged-cluster}, and two connected end graphs.
Suppose that $H=o(m)$. Then $F$ can be normalised so that
\[
 F(M)=\sin M+O_d((H/m)M),\qquad
 F'(M)=\cos M+O_d(H/m),
\]
uniformly on bounded real intervals. Its first positive root is
$\pi+O_d(H/m)$ and is simple. The same estimates hold after moving
an intervening subgraph past one neighbouring canonical block.
\end{lemma}

\begin{proof}
If an end graph $C$ of order $q$ counts its root with weight
$\alpha\in\{0,1\}$, its scalar response is
\[
 r_C^\alpha(x)=-x(q-1+\alpha)
 -x^2\one^{\mathsf T}(D_C-x\mathsf I)^{-1}\one.
\]
Take weight one for the left end graph and zero for the right end graph.
Their boundary vector and covector are respectively
$\mathbf e_1+O_d(kH)$ and $\mathbf e_2^{\mathsf T}+O_d(kH)$
in the corresponding bases.  For an intervening subgraph with $h+1$ vertices, its transfer
matrix differs from the identity by $O_d(|k|h)$. The sum of
these parameters $h$ and the two end-graph orders is $O(H)$, so telescoping the transfer product gives the
first displayed estimate.  It holds on every fixed complex
$M$-disc: $k=O_d(M/m)$, all grounded resolvents have
$|x|q^2=o(1)$, and the product norms are bounded by
$\exp(O_d(|M|H/m))$.  The error vanishes at $M=0$.
Cauchy's estimate on a slightly larger disc gives the derivative
estimate.  In particular $F(M)/M=1+o(1)$ near zero; on compact
subintervals of $(0,\pi)$ it is positive; and near $\pi$ it is
strictly decreasing and has one zero.
Every eigenvalue in this range is detected by this scalar
transfer: the grounded eliminations are nonsingular, and zero
entry value and current would propagate to zero throughout the
graph.  Thus this is the first positive eigenvalue.  Moving one canonical block changes neither $m$ nor $H$, so all estimates remain uniform.
\end{proof}

\begin{lemma}
\label{lem:even-local-cluster-shift}
Let $G$ consist of a chain of $m$ canonical blocks $M_d$, two connected end
graphs, and intervening two-terminal graphs whose terminal-to-terminal
block-cut paths contain no canonical block. Suppose that $H=o(m)$. If $G$ minimises algebraic connectivity among connected
simple $d$-regular graphs of its order, then every nontrivial
intervening subgraph lies within $C_dH$ canonical blocks of one of
the two ends, for every choice of the canonical orientations.
\end{lemma}

\begin{proof}
Merge any pieces separated by zero canonical blocks.
Each intervening subgraph then has a canonical block immediately
on each side.  If its incoming and outgoing orientations are
$\sigma,\tau$, the replacements
\[
 P_\sigma Q P_\tau\longmapsto P_\sigma^2Q,
 \qquad
 P_\sigma Q P_\tau\longmapsto QP_\tau^2
\]
preserve the graph class. They move one canonical block across $Q$, reversing
that copy if $\sigma\ne\tau$.  The degree contributions at every shared vertex again sum to
$d$; the construction preserves order, simplicity,
connectedness, and every vertex degree.  It does not move one
exceptional intervening subgraph across another.

Evaluate at the old gap $x$, and abbreviate $R=R_\theta$.
Let $l$ be the left boundary vector propagated to the left
terminal of $Q$, and let $r$ be the right boundary covector
propagated to its right terminal. Writing $T$ for the transfer
matrix of $Q$ in the chosen bases, the original characteristic
equation is $F_0=rTl=0$. The two competitors have values
\[
 F_+=rRTR^{-1}l,\qquad F_-=rR^{-1}TRl.
\]
Their first positive roots are at least the old one, by
minimality.  Lemma~\ref{lem:even-uniform-cluster-root} shows that
all three roots lie in the same interval where the normalised
characteristic functions decrease.  Hence $F_+,F_-\geqslant0$.

Put $J=\bigl(\begin{smallmatrix}0&-1\\1&0\end{smallmatrix}\bigr)$ and
$Z=(T+T^{\mathsf T}-\operatorname{tr}(T)\mathsf I)/2$.
The following identities hold for every real two-by-two
matrix $T$:
\[
 RTR^{-1}+R^{-1}TR-2T=-4\sin^2\theta\,Z,
 \qquad
 RTR^{-1}-R^{-1}TR=2\sin(2\theta)JZ.
\]
Let $j$ be the number of canonical blocks before $Q$, put
$\phi=j\theta$, and write $h=|V(Q)|-1$.  The uniform product
bounds and $m\theta=\pi+O_d(kH)$ give
\[
 l=\binom{\cos\phi}{\sin\phi}+O_d(kH),\qquad
 r=(\sin\phi,-\cos\phi)+O_d(kH).
\]
Writing $X=\bigl(\begin{smallmatrix}0&1\\1&0\end{smallmatrix}\bigr)$,
Lemma~\ref{lem:even-gauged-cluster} yields
\[
 Z=\frac{k\Delta_Q}{2c}X+O_d(k^2h^2).
\]
Since $h\leqslant H$ and $\Delta_Q\geqslant\kappa_dh$, it follows that
\begin{align*}
 F_++F_-&=\frac{2k\Delta_Q}{c}\sin^2\theta\cos(2\phi)
              +O_d(k^2hH\sin^2\theta),\\
 F_+-F_-&=-\frac{k\Delta_Q}{c}\sin(2\theta)\sin(2\phi)
              +O_d(k^2hH|\sin(2\theta)|).
\end{align*}
Nonnegativity of both competitors implies that the first line
is nonnegative and the absolute value of the second does not
exceed the first.  As $\theta\asymp_d k$, division gives
\[
 \cos(2\phi)\geqslant-O_d(kH),\qquad
 |\sin(2\phi)|=O_d(kH+\theta)=O_d(kH).
\]
Here $H\geqslant1$.  The first inequality excludes the possible
neighbourhood of $\pi/2$ allowed by the second.  Since
$0\leqslant\phi\leqslant m\theta=\pi+O_d(kH)$, we obtain
$\operatorname{dist}(\phi,\{0,\pi\})=O_d(kH)$.
Dividing by $\theta\asymp_d k$ proves
$\min\{j,m-j\}=O_d(H)$, as required.
\end{proof}

\begin{proof}[Proof of Theorem~\ref{thm:bounded-end-canonical-middle}
in the even-regular case]
For \(d=4\), Theorem~\ref{thm:quartic-unique} and the
definition of \(\mathcal G_n\) give the assertion; its block
\(M_0\) is \(M_4\). Suppose henceforth that \(d\geqslant6\).
Put $w=d+1$ and $c=2(d-2)$.
The sharp fixed-degree theorem and quantitative stability give
$n/w-o(n)$ canonical $M_d$ blocks on a single block-cut path.
The stability proof also shows that the terminals of these blocks
separate them in their order along the chain, so the blocks lie on
a single block-cut path. Choose a block-cut path
containing the largest possible number $m$ of canonical blocks.
Retain all those blocks, and include every branch in the
corresponding end graph or intervening two-terminal graph.
Writing $a,b$ for the two end graph orders and $q_i$ for the orders of
the nontrivial two-terminal graphs gives the exact order identity
\begin{equation}\label{eq:even-middle-net-mass}
 n=mw+H,\qquad H=a+b-1+\sum_i(q_i-1)=o(n).
\end{equation}
Since the two terminals of a canonical block have degrees $2,d-2$
inside that block, each intervening graph with distinct terminals has
terminal degrees in $\{2,d-2\}$ and all other degrees $d$.
Moreover, its terminal-to-terminal block-cut path contains no canonical
block, since all canonical blocks on the chosen path were retained.

There is only one nontrivial common-terminal possibility: the two
neighbouring canonical blocks each contribute degree two at
that vertex, leaving root degree $d-4$ in the two-terminal graph.  A common
terminal with complementary exterior contributions has internal
degree zero and is only a singleton; we omit it.  Two exterior
contributions of $d-2$ cannot meet at one vertex for $d\geqslant6$.
For a nontrivial common-terminal two-terminal graph $Q$ of order $q$, all other
vertices have degree $d$, so $q\geqslant d+1$ and
\[
 R_Q(s,s)=0,\qquad
 \Delta_Q=(q-1)-(d-4)=q-d+3\geqslant\frac4d(q-1).
\]
For distinct terminals, the uniform positive adjusted defect is
Theorem~\ref{thm:even-adjusted-resistance}.

Apply Lemma~\ref{lem:even-local-cluster-shift} to this decomposition.
Each nontrivial interior two-terminal graph is bordered by a canonical block
on either side.  If several exceptional pieces have no canonical
block between them, first regard their connected union as one
two-terminal graph; the same degree conditions and resistance theorem apply.
The lemma compares only the two single-block shifts of each two-terminal graph.
When the two neighbouring canonical blocks have different
orientations, the moved block is reversed. Thus the degree at each
shared vertex is preserved, and the order of the intervening
subgraphs is unchanged.
Exact minimality forces each two-terminal graph to lie within $C_dH$
canonical blocks of one of the two global ends.

Since $H=o(m)$, these two regions are disjoint for large $n$.
Absorb them into end graphs $Q_-,Q_+$, of orders $O_d(H)$, leaving
a canonical middle.  With $j_-,j_+$ the numbers of absorbed
canonical blocks, global maximality of the original path gives
$m_s(Q_\pm)\leqslant j_\pm$: a rooted path with a larger count
could replace that end of the original path and increase its
total count.  Consequently
\begin{align*}
 b_s(Q_-)+b_s(Q_+)
 &\geqslant (|Q_-|-1-wj_-)+(|Q_+|-1-wj_+)\\
 &=H-1.
\end{align*}
Thus at least one end graph $Q$ has $b_s(Q)\geqslant(H-1)/2$,
root degree $2$ or $d-2$, and all other degrees $d$.

We construct comparison end graphs of the same order and root degree.
For each $r\in\{2,d-2\}$ and all sufficiently large $q$, an
end graph $C_{q,r}$ exists with
\begin{equation}\label{eq:even-comparison-cap-torsion}
 \mathcal T(C_{q,r})=\frac{q^3}{3c}+O_d(q^2).
\end{equation}
To construct it, take a long consistently oriented canonical
chain with initial terminal degree two.  If $r>2$, identify that terminal with the root of a bounded
graph whose root has degree $r-2$; for $r=2$ use the terminal
alone. At the far terminal attach a bounded rooted graph whose
root has degree two.  Such bounded rooted graphs are obtained by deleting
disjoint offset-two edges from a sufficiently large $d$-regular
circulant of offsets $\pm1,\ldots,\pm d/2$ and joining their
endpoints to a new root.  The original cycle remains, and all
old degrees are restored to $d$.  Varying the bounded far-end graph
order over $w$ consecutive values covers every residue modulo
$w$.  Let $t$ be the coordinate along the chain, starting at zero
at its first terminal and increasing by $w$ across each block.
The trial function $(qt-t^2/2)/c$, harmonically interpolated
inside each canonical block, proves
the lower estimate in \eqref{eq:even-comparison-cap-torsion};
its sum and energy have errors $O_d(q^2)$.
Lemma~\ref{lem:even-rooted-torsion} proves the matching upper bound.

Suppose now that $H$ is unbounded along a sequence of minimisers.
For the end graph just selected, of order $q=O_d(H)$, the rooted torsion
deficit yields
\[
 \mathcal T(C_{q,r})-\mathcal T(Q)
 \geqslant\zeta_d((H-1)/2)^3-O_d(H^2)>0.
\]
Count a root of an end graph with weight
$\alpha\in\{0,1\}$ according to whether the neighbouring
piece has already counted it.  Its scalar Schur response is
\[
 r_Q^{\alpha}(x)
 =-x(q-1+\alpha)-x^2\mathcal T(Q)+O_d(x^3q^5),
 \qquad xq^2=o(1).
\]
This follows directly by eliminating
$D_Q-x\mathsf I$: the exact second term is
$-x^2\one^{\mathsf T}(D_Q-x\mathsf I)^{-1}\one$.
The left end graph counts its root once and the right end graph excludes its
root, which is already counted by the final canonical block.
Using the same convention in an end graph replacement preserves the
ordinary unweighted vertex mass of the graph.

At $x=\mu(G)=O_d(n^{-2})$, we have $xH^2=o(1)$, so the
torsion difference times $x^2$ dominates the response remainder.
Replacing $Q$ by $C_{q,r}$ strictly decreases its scalar Schur
response at $x$.  The Fiedler value at the shared attachment is
nonzero: if it vanished, positivity of $D_Q-x\mathsf I$ would force
the end graph interior to vanish.  The root equation would give zero
current into the canonical middle; its exact transfers and
elimination of the other end graph would then force the entire
eigenfunction to vanish.  Thus the strict Schur-complement
comparison gives a connected simple regular graph of the same
order and strictly smaller algebraic connectivity, a contradiction.

It follows that $H$ is bounded in terms of $d$.
The localisation by single-block moves already proved then bounds both end
regions by a constant depending only on $d$, as required.
\end{proof}

\subsection{The block structure of the ends}

It remains to rule out branching at both kinds of node in the
block-cut tree. We use edge switches in the regular classes and
also a grafting in the minimum-degree class.

\begin{lemma}\label{lem:finite-end-fiedler}
Let \(d\geqslant3\) and \(C\) be fixed. Let \(G_m\) consist of a
chain of \(m\) copies of \(L_d\), joined by bridges to connected
end graphs containing at most \(C\) vertices in total. Suppose that
\(G_m\) has minimum degree at least \(d\). Then, for all sufficiently
large \(m\), its algebraic connectivity is simple, and every Fiedler
vector is nonzero at the two exterior attachment vertices of the end
graphs and is either strictly positive or strictly negative on each
end graph.

Moreover, if \(f\) is a Fiedler vector positive on one end graph,
\(u\) is a cut vertex in that end graph, and \(B\) is a component
of \(G_m-u\) not containing the canonical middle, then
\[
 f(v)>f(u)>0\qquad(v\in B).
\]
\end{lemma}

\begin{proof}
Write \(x=\mu(G_m)=O_{d,C}(m^{-2})\).  Every connected
subgraph of an end graph, grounded at its exterior boundary, has
smallest grounded eigenvalue bounded below by a positive
constant depending only on \(C\). This follows by grounding one additional vertex, bounding effective
resistances by distance, and obtaining an inverse with
operator norm at most \((C+1)^2\).

If a Fiedler vector vanished at an exterior attachment,
positive definite elimination would force it to vanish on
that end graph and give zero current through the attachment bridge.
Vanishing value and current then imply vanishing throughout every
canonical graph.  Indeed,
its interior coordinates are equal when \(x<1\), by subtracting
their eigenvalue equations.  The equation at its zero incoming
terminal forces this common value to vanish; the interior equation
then forces the other terminal to vanish, and the outgoing current
is zero.  Elimination of the other end graph consequently forces
the whole vector to vanish.  Evaluation at one exterior attachment
is consequently an injective linear map from the Fiedler
eigenspace to \(\mathbb R\), proving simplicity.

For an entire end graph \(Q\) with terminal \(s\), the matrix
\(L_Q+\mathbf e_s\mathbf e_s^{\mathsf T}-x\mathsf I\)
is irreducible and has a strictly positive inverse.  Its
response to the exterior attachment therefore gives values on the end graph with the same strict sign.  For the last assertion, let
\(D_B=L_{G_m}[B]\).  Its only boundary vertices are attached
to \(u\), so \(D_B\one\) is their boundary-degree vector.
The eigenvalue equation gives
\[
 f_B=f(u)\bigl(\one+x(D_B-x\mathsf I)^{-1}\one\bigr).
\]
The inverse is nonnegative with strictly positive row sums.
Since \(x>0\), this proves the strict inequalities.
\end{proof}

\begin{lemma}\label{lem:finite-end-slack-grafting}
Fix \(d\geqslant3\) and \(C\). Let \(G\) be a bridge chain of
\(m\) copies of \(L_d\) with a connected end graph at each end,
the two end graphs having total order at most \(C\). Suppose that
\(G\) minimises algebraic connectivity among connected graphs of
its order and minimum degree at least \(d\). Let \(u\) be a cut vertex
in an end graph with at least two components of \(G-u\) not containing
the middle. Then, for all sufficiently large \(m\), every such
component \(B\) satisfies
\[
 \deg_G(u)-|N_G(u)\cap B|\leqslant d-1.
\]
In particular, \(\deg_G(u)\leqslant2d-3\).
\end{lemma}

\begin{proof}
Take a unit, mean-zero Fiedler vector which is positive on
the end graph.  Choose a vertex \(w\) in another component away
from the middle.  By the preceding lemma,
\(a=f(w)-f(u)>0\), and \(f\) is positive on \(B\).
If \(r=|N_G(u)\cap B|\) and \(\deg_G(u)-r\geqslant d\),
replace every edge \(uv\), \(v\in N_G(u)\cap B\), by
\(wv\).  All the new edges were absent, and the new graph
\(G'\) is connected and has minimum degree at least \(d\).

Put \(g=f+a\mathbf1_B\), and subtract its mean.  Internal
edge differences are unchanged, while on each moved edge
the new difference is
\(g(w)-g(v)=f(u)-f(v)\).  Thus the energy is unchanged.
Writing \(b=|B|\), the centred squared norm is
\[
 1+2a\sum_{v\in B}f(v)+a^2b(1-b/n)>1.
\]
The Rayleigh principle gives \(\mu(G')<\mu(G)\), a
contradiction.  This proves the first assertion.

For two distinct components \(B_1,B_2\), let
\(r_i=|N_G(u)\cap B_i|\).  The component towards the
middle contains at least one neighbour, so
\(r_1+r_2\leqslant\deg_G(u)-1\).  The first assertion gives
\(r_i\geqslant\deg_G(u)-d+1\).  Combining these inequalities
yields \(\deg_G(u)\leqslant2d-3\).
\end{proof}

\Needspace{14\baselineskip}
\begin{lemma}\label{lem:even-finite-end-fiedler}
Fix an even integer \(d\geqslant4\) and \(C\). Let \(G\) be a connected
\(d\)-regular graph consisting of a chain of \(m\) copies of \(M_d\),
identified at consecutive terminals, and two connected end graphs of
total order at most \(C\). For all sufficiently large \(m\), the algebraic
connectivity of \(G\) is simple, every Fiedler vector is nonzero at both
roots of the end graphs, and its values on each end graph are either
all positive or all negative.

Let \(f\) be a Fiedler vector positive on one end graph. If \(u\) is
a cut vertex there and \(B\) is a component of \(G-u\) away from the
middle, then
\[
 f(v)>f(u)>0\qquad(v\in B).
\]
\end{lemma}

\begin{proof}
Write \(x=\mu(G)=O_{d,C}(m^{-2})\), where \(m\) is the
number of middle blocks.  Grounding a root in any connected
end graph leaves a positive definite matrix whose smallest eigenvalue
is bounded below uniformly over the finitely many end graphs of order
at most \(C\).

We first show that vanishing value and current at one terminal
force vanishing throughout the chain.  In a canonical
block, write \(s,t\) for its terminal values and \(\alpha,\beta\)
for the common values on its two interior layers.  The two vertices
of the first layer have equal values, and so do the \(d-2\) vertices
of the second layer: subtraction of their equations gives the
factor \(d+1-x\), which is nonzero when \(x<1\).
If the incoming terminal has value zero and the current from the
preceding piece is zero, its vertex equation gives \(\alpha=0\)
when the incoming terminal has degree two in this block.
The two interior equations are
\[
 (d-1-x)\alpha=s+(d-2)\beta,\qquad
 (3-x)\beta=2\alpha+t.
\]
They imply \(\beta=t=0\), and the outgoing current is zero.
For the reversed block the same argument starts with \(\beta=0\)
and then gives \(\alpha=s=0\).  This also works at a terminal
shared by consecutive blocks, because the preceding block contributes
zero current there.

If a Fiedler vector vanished at a root of an end graph, grounded elimination
would make it vanish on that end graph.  The root equation supplies zero
incoming current to the middle.  The preceding propagation and the
elimination of the other end graph would make the entire vector zero.
Evaluation at a root of an end graph is therefore an injective linear map from
the eigenspace to \(\mathbb R\), proving simplicity and nonvanishing.

For any component \(B\) attached only at a vertex \(u\), put
\(D_B=L_G[B]\).  Its boundary-degree vector is \(D_B\one\).
For sufficiently small \(x\),
\[
 f_B=f(u)\bigl(\one+x(D_B-x\mathsf I)^{-1}\one\bigr).
\]
The inverse is nonnegative with strictly positive row sums.
Applied to the components of an end graph minus its root, this shows
that its Fiedler values are either all positive or all negative.  Applied at any other cut vertex,
it gives the asserted strict increase away from the middle.
\end{proof}

The inequalities on the end graphs in Lemmas~\ref{lem:finite-end-fiedler}
and~\ref{lem:even-finite-end-fiedler} are the only properties of
the canonical middle used by the following switches. Thus the
regular-class arguments apply in both degree parities.

\begin{lemma}\label{lem:no-branching-cutvertex}
Let \(G\) be a minimiser in one of the classes of
Theorem~\ref{thm:bounded-end-canonical-middle}, with sufficiently
large order.  No cut vertex in either bounded end graph separates two
components away from the canonical middle.  Thus every
cut-vertex node of its block-cut tree has degree at most two.
\end{lemma}

\begin{proof}
Take a unit, mean-zero Fiedler vector, oriented to be positive
on the end graph.  Write \(x=\mu(G)\).  Suppose that a cut vertex
\(u\) has two components \(B,C\) away from the middle.
All their values exceed \(f(u)>0\), by
Lemma~\ref{lem:finite-end-fiedler} or
Lemma~\ref{lem:even-finite-end-fiedler}, as appropriate.  Put
\[
 m_B=\min_B f,\quad M_B=\max_B f,
 \qquad m_C=\min_C f,\quad M_C=\max_C f.
\]

We use two observations.  At a maximum vertex
of \(B\), an incident bridge cannot lead to a component away
from \(u\): summing the eigenvalue equation on that component
would make the value at its other endpoint strictly larger.
There is at most one bridge towards \(u\).  Since the degree
is at least three and at most one neighbour is \(u\), some
nonbridge edge incident with this maximum has both endpoints in \(B\).

Second, if \(B\) contains a vertex not adjacent to \(u\),
then it contains an internal nonbridge edge \(yz\) with
\(y\notin N(u)\).  Otherwise every edge incident with
\(B\setminus N(u)\) would be a bridge.  Each component
induced by these vertices would be a tree attached by exactly
one edge to the connected graph on \(\{u\}\cup N(u)\).
Such a tree has a vertex of degree one in \(G\), a contradiction.

Suppose first that the value intervals overlap, and relabel
them so that \(M_B\leqslant M_C\).  Unless \(C\) is constant
at \(M_B\), choose an internal edge \(cd\) of \(C\) with
\(f(c)\leqslant M_B\leqslant f(d)\) and \(f(c)<f(d)\).
Such an edge is found on a path between a minimum and a
maximum of \(C\); at an endpoint of the interval use an edge
entering or leaving the corresponding level set.  Choose an
internal nonbridge edge \(ab\) of \(B\) with \(f(b)=M_B\).

Replace \(ab,cd\) by \(ac,bd\).  The new edges were absent
because \(B,C\) are distinct components of \(G-u\).
Deleting \(ab\) does not disconnect \(B\cup\{u\}\): every
cycle containing it lies in that subgraph.  If \(cd\) is a
bridge, its higher endpoint \(d\) lies on its side away from
\(u\), so the new edge \(bd\) reconnects that side.  If it
is not a bridge, its cycle lies in \(C\cup\{u\}\) and is
unaffected by deleting \(ab\).  Thus the switched graph is
connected and preserves every degree.

Its energy change at \(f\) is
\[
 2(f(a)-f(d))(f(b)-f(c))\leqslant0.
\]
The four values are not all equal.  If the energy change is
zero, the changes in the Laplacian coordinates at \(a,b\)
are \(f(b)-f(c)\) and \(f(a)-f(d)\), at least one of which
is nonzero.  Equality in the Rayleigh principle is therefore
impossible, so the new algebraic connectivity is strictly
smaller in all cases.

If \(C\) is constant at \(M_B\) but \(B\) is not, use
an internal nonbridge edge of \(C\) and an internal edge of
\(B\) entering its maximum level.  The same argument applies.
The case when both components are constant is treated below.

Now suppose that the intervals are disjoint, say
\(M_B<m_C\).  A minimum vertex \(v\) of \(C\) must be
adjacent to \(u\), since otherwise its eigenvalue equation
would give \(xf(v)\leqslant0\).  If \(B\) has a vertex
not adjacent to \(u\), choose its nonbridge edge \(yz\)
as in the second observation and replace \(uv,yz\) by
\(uy,vz\).  These are new edges.  Deleting \(yz\) leaves
\(B\cup\{u\}\) connected; deleting \(uv\) either leaves
the rest connected or separates \(C\), which is reattached
by \(vz\).  Every degree is preserved, and the energy change
\[
 2(f(u)-f(z))(f(v)-f(y))<0
\]
contradicts minimality.

The remaining possibility is that all vertices of one
component, say \(B\), are adjacent to \(u\).  Their equations
give
\[
 \bigl(L_{G[B]}+(1-x)\mathsf I\bigr)f_B=f(u)\one,
 \qquad f_B=\frac{f(u)}{1-x}\one.
\]
Conversely, a positive constant component must have this
property, by its vertex equations.  Thus this case includes
the previously deferred pair of constant components.
Its order is at least \(d\), because each of its vertices
has degree at most \(|B|\).  In the regular class this is
impossible: \(u\) also has a neighbour towards the middle
and a neighbour in \(C\), giving degree greater than \(d\).
In the minimum-degree class,
\[
 \deg_G(u)-|N_G(u)\cap C|\geqslant |B|+1>d.
\]
The strict grafting of Lemma~\ref{lem:finite-end-slack-grafting}
therefore supplies a competitor.  This final contradiction
proves the lemma.
\end{proof}

\begin{lemma}\label{lem:no-branching-block}
Let \(G\) be a minimiser in one of the classes of
Theorem~\ref{thm:bounded-end-canonical-middle}, with sufficiently
large order.  Every block of \(G\) contains at most two cut vertices.
\end{lemma}

\begin{proof}
Only an end graph could contain an exceptional block.  Orient
a unit Fiedler vector to be positive on that end graph, and write
\(x=\mu(G)\).  Suppose that a block \(D\) has a cut vertex
\(p\) towards the canonical middle and two distinct cut
vertices \(u_1,u_2\) away from it.  Let \(H_i\) be a
component of \(G-u_i\) away from \(D\).  These components
are disjoint, and their only external neighbours are
\(u_1,u_2\), respectively.

By Lemma~\ref{lem:finite-end-fiedler} or
Lemma~\ref{lem:even-finite-end-fiedler},
\[
 f(p)<f(u_i)<\min_{H_i}f.
\]
Write \(m_i=\min_{H_i}f\), \(M_i=\max_{H_i}f\), and
relabel so that \(M_1\leqslant M_2\).  Each \(H_i\) has
an internal nonbridge edge incident with a maximum vertex,
by the observation in Lemma~\ref{lem:no-branching-cutvertex}.
Every cycle containing an internal edge of \(H_i\) lies
in \(H_i\cup\{u_i\}\).

If \([m_1,M_1]\) and \([m_2,M_2]\) overlap, and the two
components are not both constant at the same value, the
overlapping-interval switch in that lemma gives a connected
degree-preserving competitor with strictly smaller algebraic
connectivity.  That switch uses internal edges from the two
components, so it does not require their attachment vertices
to coincide.  The cycles used for connectivity are still
unaffected by the edge deletion in the other component.

If both components are constant at the same value \(M\),
every vertex of \(H_i\) is adjacent to \(u_i\), and
\(f(u_i)=(1-x)M<M\).  Choose an internal nonbridge edge
\(ab\) of \(H_1\) and an attachment edge \(u_2v\), with
\(v\in H_2\).  Replace them by \(u_2a,vb\).  The new
edges were absent, the graph remains connected, and all
degrees are preserved.  The energy change is zero, but the
Laplacian coordinate of \(f\) at \(v\) changes by
\(f(u_2)-M\ne0\).  Equality in the Rayleigh principle is
therefore impossible, giving a strict decrease again.

We may consequently assume that \(M_1<m_2\).  Choose a
minimum vertex \(v_2\) of \(H_2\).  It is adjacent to
\(u_2\), since otherwise its eigenvalue equation would
give \(xf(v_2)\leqslant0\).  Choose an internal nonbridge
edge \(ab\) of \(H_1\) with \(f(b)=M_1\).

If \(f(u_2)<M_1\), replace \(u_2v_2,ab\) by
\(u_2a,v_2b\).  Both new edges were absent.  The deletion
of \(ab\) leaves \(H_1\cup\{u_1\}\) connected; if
\(u_2v_2\) is a bridge, its detached component is reattached
by \(v_2b\), and otherwise its cycle is unaffected.
Every degree is preserved, and the energy change is
\[
 2(f(u_2)-f(b))(f(v_2)-f(a))<0.
\]
This contradicts minimality.

Finally suppose that \(f(u_2)\geqslant M_1\).  Since
\(D\) has at least three cut vertices, it is not a bridge
block, and \(D-u_1\) is connected.  Choose a path in it
from \(p\) to \(u_2\).  Its initial value is less than
\(M_1\) and its final value is at least \(M_1\), so it
contains an edge \(yz\) with
\[
 f(y)<M_1\leqslant f(z).
\]
Both endpoints differ from \(u_1\).  Replace \(ab,yz\)
by \(ay,bz\).  These cross edges were absent because the
only exterior neighbour of \(H_1\) is \(u_1\).
The edge \(yz\) lies on a cycle in \(D\), while the cycle
for \(ab\) lies in \(H_1\cup\{u_1\}\); hence the two
deletions leave the graph connected.  All degrees are
preserved, and the energy change is
\[
 2(f(a)-f(z))(f(b)-f(y))\leqslant0.
\]
If it is zero, the Laplacian coordinate at \(a\) changes
by \(f(b)-f(y)>0\).  As before the new algebraic
connectivity is strictly smaller.  This final contradiction
proves the lemma.
\end{proof}

\begin{proof}[Proof of Theorem~\ref{thm:guiduli-mohar-pathlike}]
By Theorem~\ref{thm:bounded-end-canonical-middle}, every minimiser
in the classes under consideration consists of a canonical middle and
two end graphs of bounded total order. Lemma~\ref{lem:finite-end-fiedler} supplies the end graph
inequalities for the minimum-degree and odd-regular classes;
Lemma~\ref{lem:even-finite-end-fiedler} supplies them for even
regular degree, including degree four.
Lemma~\ref{lem:no-branching-cutvertex} gives degree at most two
at every cut-vertex node of the block-cut tree, and
Lemma~\ref{lem:no-branching-block} gives the same at every block
node. This connected tree is therefore a path, with all
noncanonical blocks confined to its bounded ends.

For sufficiently large order the canonical middle is nonempty
and contains vertices of degree \(d\). Thus the least values
under \(\delta(G)\geqslant d\) and under \(\delta(G)=d\) agree.
Every minimiser in the latter class is consequently also a
minimiser in the former, proving the asserted version.
\end{proof}
\section{Adjacency eigenvalues and random walks}\label{sec:applications}

In this section, we derive the consequences for adjacency spectral
radius and for hitting and commute times. The adjacency bound follows
by adjoining a bounded number of vertices to two copies of an extremal
nonregular graph and applying the fixed-degree spectral bound.
We then deduce an effective-resistance bound from the minimum-degree
theorem. This gives the commute-time estimate; the hitting-time
estimate follows by applying the cut inequalities to the corresponding
electrical potential.

\subsection{The extremal adjacency spectral radius}
\label{sec:liu-spectral-radius}

\begin{proof}[Proof of Theorem~\ref{thm:liu-spectral-radius}]
Let $G$ attain $\rho(n,d)$. It has at most $d$ vertices of degree
less than $d$. Indeed, if there were more, two of these vertices
would be nonadjacent. Adding the edge between them would preserve
maximum degree $d$ and leave another vertex of degree less than $d$, so the graph
would remain nonregular. Its spectral radius would strictly increase,
a contradiction. Consequently the total degree deficiency satisfies
\[
 1\leq D:=\sum_v(d-\deg_G(v))\leq d(d-1).
\]

Choose a connected simple $d$-regular graph $R$ of bounded order
with a matching of size $D$ whose deletion leaves it connected.
For an explicit construction, put $M=4D+2d+2$ and label the
vertices by $0,\ldots,M-1$, with differences taken modulo $M$. If $d=2k+1$, join vertices at cyclic differences
$\pm1,\ldots,\pm k$ and opposite vertices; delete $D$ edges of
the opposite-vertex matching. If $d=2k$, use differences
$\pm1,\ldots,\pm k$ and delete the disjoint edges
$(4j,4j+2)$ for $0\leq j<D$. In both cases the spanning cycle
of difference one remains, and the deletion leaves $2D$ distinct
vertices of degree $d-1$.

Take two disjoint copies of $G$ and this modified graph $R$.
For each vertex $v$ in either copy, add $d-\deg_G(v)$ edges
from $v$ to distinct vertices of degree $d-1$ in $R$, using each
of these $2D$ vertices exactly once. The resulting graph
$\widehat G$ is connected, simple, and $d$-regular, with
$N=|V(\widehat G)|=2n+O_d(1)$ vertices. Let $f$ be a positive
eigenvector of $\mathsf A(G)$ for $\rho(G)$, normalised by
$\|f\|_2=1$, and assign the values $f$ and $-f$ to the two copies
and zero to every vertex of $R$. This gives a mean-zero vector
with squared norm two and Laplacian Rayleigh quotient
\[
 \sum_{uv\in E(G)}(f(u)-f(v))^2
 +\sum_v(d-\deg_G(v))f(v)^2
 =d-\rho(n,d).
\]
The sharp fixed-degree bound of
Theorem~\ref{thm:fixed-regular-sharp} therefore gives
\[
 d-\rho(n,d)\geq\mu(\widehat G)
 \geq(c_d-o(1))\frac{\pi^2}{N^2}
 =(c_d-o(1))\frac{\pi^2}{4n^2}.
\]
The matching upper bound is supplied by the constructions in
\cite[Theorems~6.1 and~6.2]{LiuSpectralRadius2024}.
\end{proof}

We next determine the sharp constant in Cioab\u{a}'s diameter bound.

\Needspace{12\baselineskip}
\begin{proposition}\label{prop:sharp-diameter-spectral-radius-constant}
The largest universal constant $c$ for which
\[
 \Delta(G)-\rho(G)>\frac{c}{|V(G)|\operatorname{diam}(G)}
\]
holds for every finite connected nonregular simple graph is $c=1$.
There is a sequence $(G_h)$, each with one vertex of degree two and
all other vertices of degree three, such that
\[
 |V(G_h)|\operatorname{diam}(G_h)(3-\rho(G_h))\longrightarrow1.
\]
\end{proposition}

\begin{proof}
Cioab\u{a}'s inequality \cite{Cioaba2007} gives the lower bound
with $c=1$. The following sequence proves that this constant is sharp.

Let $h\ge2$, and let $R_h$ be a complete rooted binary tree of height
$h$ with a cycle added through its $2^h$ leaves. It has order
$N_h=2^{h+1}-1$ and diameter at most $2h$, with its root of degree
two and all other vertices of degree three. Let $B$ be obtained by
subdividing one edge of $K_4$, and denote the new vertex by $r$.
This graph has five vertices, all within distance two of $r$;
$r$ has degree two and the remaining vertices have degree three.

Start with a path $v_1,\ldots,v_L$, where $L\ge h$.  Attach to each $v_i$
a disjoint copy $B_i$ of $B$ by the edge $v_ir_i$, where $r_i$
corresponds to $r$, and attach $v_L$ to
the root of a disjoint copy of $R_h$ by one further edge. The resulting
graph $G_{h,L}$ is connected and simple, with $v_1$ of degree two and
all other vertices of degree three, as is checked directly. Its order
and diameter are
\[
 n=N_h+6L,\qquad D:=\operatorname{diam}(G_{h,L})=L+h+3.
\]
Indeed, two vertices in the path and its attached copies of $B$ are
at distance at most $L+5$, a vertex there and a vertex of $R_h$ are
at distance at most $L+h+3$, and two vertices of $R_h$ are at distance
at most $2h$.  Equality is attained by a vertex of $B_1$ at distance
two from $r_1$ and any leaf of $R_h$.

For any real vector $f$ on this graph,
\[
 f^{\mathsf T}(3\mathsf I-\mathsf A(G_{h,L}))f
 =\sum_{uv\in E(G_{h,L})}(f(u)-f(v))^2+f(v_1)^2.
\]
Set $f=1$ on $R_h$ and $f=i/(L+1)$ on
$\{v_i\}\cup V(B_i)$. The only nonzero contributions are
$(L+1)^{-2}$ from each of the $L-1$ path edges, the edge from
$v_L$ to $R_h$, and the displayed deficiency term. Consequently
\[
 3-\rho(G_{h,L})
 \le\frac{1/(L+1)}{\sum_v f(v)^2}
 \le\frac{1}{N_h(L+1)}.
\]
It follows that
\[
 1<nD(3-\rho(G_{h,L}))
 \le\left(1+\frac{6L}{N_h}\right)
       \frac{L+h+3}{L+1}.
\]
Taking $G_h=G_{h,h^2}$ and letting $h\to\infty$ proves the result.
\end{proof}

\subsection{A sharp resistance bound}
\label{sec:two-terminal-resistance}

We next derive a resistance bound from the minimum-degree spectral
bound and its stability theorem.

Recall that \(L_d=K_{d+1}-e\), with the endpoints of the missing
edge as its terminals.

\Needspace{12\baselineskip}
\begin{proposition}\label{prop:two-terminal-resistance}
Let \(d\geqslant3\) be an integer, and let \(H\) be a connected graph
of order \(h\) with distinct vertices \(s,t\). Suppose that
\(\deg_H(s),\deg_H(t)\geqslant d-1\) and every other vertex has
degree at least \(d\). Then
\[
 1+R_H(s,t)\leqslant\frac{h}{d-1}.
\]
Equality holds if and only if \(H\) is a chain of copies of
\(L_d\), joined by bridges, with \(s,t\) as its two
outer terminals.
\end{proposition}

\begin{proof}
Write \(R=R_H(s,t)\), and for a positive integer \(q\) take copies
\(H_1,\ldots,H_q\) of \(H\), with distinguished vertices \(s_j,t_j\)
corresponding to \(s,t\). Joining consecutive copies by the edges
\(t_js_{j+1}\) and attaching a copy of \(K_{d+1}\) by a single edge
at each outer terminal gives a graph \(G_q\) of minimum degree at
least \(d\) and order \(N=qh+2(d+1)\).

We first show that
\begin{equation}\label{eq:periodic-resistance-trial}
 \limsup_{q\to\infty}N^2\mu(G_q)
 \leqslant\frac{h\pi^2}{1+R}.
\end{equation}
Put \(a_j=\cos((j-\tfrac12)\pi/q)\) and define \(f\) by extending
the values \(a_j,a_{j+1}\) at \(s_j,s_{j+1}\) harmonically across
\(H_j\) and the following bridge for \(1\leqslant j<q\), and by
taking the corresponding constant endpoint values on \(H_q\) and
each end graph. Each harmonic part has energy
\((a_j-a_{j+1})^2/(1+R)\), so
\[
 \cE_{G_q}(f)=\frac{1}{1+R}
      \sum_{j=1}^{q-1}(a_j-a_{j+1})^2.
\]
The harmonic extension lies between its boundary values, so
\(|a_j-a_{j+1}|=O(q^{-1})\) implies that each value on \(H_j\)
differs from \(a_j\) by \(O(q^{-1})\). Since \(H\) is fixed and
\(\sum_j a_j=0\),
\[
 \sum_v f(v)^2=h\sum_{j=1}^q a_j^2+O_H(1)
              =\frac{hq}{2}+O_H(1),
 \qquad \sum_v f(v)=O_H(1).
\]
Subtracting the mean and using the first nonconstant eigenvector
of \(P_q\) proves \eqref{eq:periodic-resistance-trial}.
Theorem~\ref{thm:minimum-degree-bound} now gives
\((d-1)\pi^2\leqslant h\pi^2/(1+R)\), as required.

Suppose that equality holds. Then \((G_q)\) is an asymptotically
extremal sequence. By Theorem~\ref{thm:minimum-degree-stability},
all but \(o(N)\) vertices of \(G_q\) lie in disjoint induced copies
of \(L_d\), each having exactly two boundary edges, both bridges.
Since \(L_d\) has no bridge, such a copy cannot cross an edge
joining two copies of \(H\), or an edge joining an end graph.
Consequently, for sufficiently large \(q\), some \(H_j\) is
entirely covered by these copies of \(L_d\); otherwise there
would be at least \(q\) uncovered vertices.

Contract the copies covering this \(H_j\). Each contracted vertex
has two incident boundary edges, counting the two edges leading
outside \(H_j\); since the quotient is connected and its internal
edges are bridges, it is a path with the external edges at its
endpoints. The attachment vertices of each copy are its two deficient
terminals, so \(H\), with the specified vertices \(s,t\), is the
chain in the statement.

Conversely, the resistance between the terminals of \(L_d\) is
\(2/(d-1)\).  A chain of \(m\) such copies has order \(m(d+1)\)
and terminal resistance \(2m/(d-1)+m-1\), giving equality.
\end{proof}

\subsection{Maximal hitting and commute times}
\label{sec:hitting-commute}

We now prove the sharp bounds on hitting and commute times
conjectured by Aldous and Fill \cite[Open Problem~6.14]{AldousFill}.
Recall that \(T_t\) is the first
hitting time of \(t\), and that
\[
 t_{\mathrm{hit}}(G)=\max_{s,t\in V(G)}\mathbb E_sT_t,
 \qquad
 \tau^*(G)=\max_{s,t\in V(G)}
       (\mathbb E_sT_t+\mathbb E_tT_s).
\]
\begin{theorem}\label{thm:extremal-times}
\label{thm:hitting-time}\label{cor:commute-time}
Let \(G\) be a connected \(d\)-regular graph of order \(n\), where
\(d\geqslant3\). Then
\begin{align}
 t_{\mathrm{hit}}(G)
 &\leqslant\min\left\{
 \frac{d(n-1)(n+d)}{2(d-1)},\,
 dn\left(\frac{n}{d-1}-1\right)\right\}
 \leqslant\frac34(n-1)(n+3),
 \label{eq:hitting-time-finite}\\
 \tau^*(G)&\leqslant dn\left(\frac{n}{d-1}-1\right)
 \leqslant\frac32n^2-3n.
 \label{eq:commute-time-finite}
\end{align}
The bounds independent of \(d\) also hold when \(d=2\) and
\(n\geqslant3\). As \(n\) tends to infinity through
even integers, the maxima over all connected regular graphs of order
\(n\) satisfy
\[
 \max_G t_{\mathrm{hit}}(G)=\frac34n^2+O(n),
 \qquad
 \max_G\tau^*(G)=\frac32n^2+O(n).
\]
\end{theorem}

The resistance bound gives the commute-time assertion immediately.
For hitting times, we also use the local cut estimates to control
the sum of the voltages.

\begin{proof}[Proof of Theorem~\ref{thm:extremal-times}: commute times]
The commute-time identity gives
\[
 \mathbb E_sT_t+\mathbb E_tT_s=2|E(G)|R_G(s,t)=dnR_G(s,t).
\]
Applying Proposition~\ref{prop:two-terminal-resistance} to \(G\)
with the specified vertices \(s,t\) gives the first inequality,
and the second follows because \(dn(n/(d-1)-1)\) decreases with
\(d\geqslant3\).
A connected \(2\)-regular graph of order \(n\) is the cycle \(C_n\), for which
\(\tau^*(C_n)=2\lfloor n^2/4\rfloor\leqslant3n^2/2-3n\)
when \(n\geqslant3\).
Finally, the cubic chains \(X_n\) have bounded end
graphs and \(n/4+O(1)\) copies of \(L_3\). Each copy has
terminal resistance one, and consecutive copies are joined
by a unit-resistance bridge.  Thus their end-to-end resistance
is \(n/2+O(1)\), and their commute time is
\(3n^2/2+O(n)\).
\end{proof}

\begin{proof}[Proof of Theorem~\ref{thm:extremal-times}: hitting times]
Fix distinct vertices \(s,t\), and let \(V\) be the voltage for
unit current from \(s\) to \(t\), normalised by \(V(t)=0\), so that
\(\mathsf L(G)V=\mathbf e_s-\mathbf e_t\). Since \(V\) is harmonic
away from \(s,t\), with its minimum at \(t\) and maximum at \(s\),
we may order the vertices so that
\[
 t=v_1,\qquad V(v_1)\leqslant\cdots\leqslant V(v_n),
 \qquad v_n=s.
\]
Put \(y_i=V(v_{i+1})-V(v_i)\geqslant0\) and
\(\mathbf y=(y_1,\ldots,y_{n-1})^{\mathsf T}\).
For this vertex order, use the cut notation \(S_i,q_i,\mathsf K\)
from Section~\ref{sec:preliminaries}. The current across every initial
cut is one, so
\begin{equation}\label{eq:voltage-cut-equation}
 \mathsf K\mathbf y=\one.
\end{equation}
This remains true when some voltage values coincide.

Put \(c=d-1\).  The low-cut partition in the proof of
Theorem~\ref{thm:minimum-degree-bound} applies to any vertex
ordering: its construction uses only cut sizes and the degree
condition.  It partitions the increment indices into intervals
\(B\) of lengths \(r_B\leqslant d+1\), on each of which
Lemma~\ref{lem:general-low-cut}, or the singleton bound
\(q_i\geqslant c\), gives
\[
 \mathbf y_B^{\mathsf T}\mathsf K_{BB}\mathbf y_B
 \geqslant\frac{c}{r_B}
       \left(\sum_{i\in B}y_i\right)^2.
\]
All entries of \(\mathsf K\) and \(\mathbf y\) are nonnegative.
Writing \(Y_B=\sum_{i\in B}y_i\), equation
\eqref{eq:voltage-cut-equation} therefore implies
\[
 \frac{c}{r_B}Y_B^2
 \leqslant\mathbf y_B^{\mathsf T}\mathsf K_{BB}\mathbf y_B
 \leqslant\sum_{i\in B}y_i(\mathsf K\mathbf y)_i=Y_B,
 \qquad Y_B\leqslant\frac{r_B}{c}.
\]

The hitting-time identity for regular graphs is
\(\mathbb E_sT_t=d\sum_vV(v)\).  Indeed, if
\(h(v)=\mathbb E_vT_t\), then
\(\mathsf L(G)h=d\one-dn\mathbf e_t\), and taking its inner product
with \(V\) proves the identity.  If \(a_B\) is the first
index in \(B\), it follows that
\begin{align*}
 \frac1d\mathbb E_sT_t
 &=\sum_{i=1}^{n-1}(n-i)y_i
 \leqslant\sum_B(n-a_B)Y_B\\
 &\leqslant\frac1c\sum_B r_B(n-a_B)
 =\frac{n(n-1)+\sum_Br_B(r_B-1)}{2c}
 \leqslant\frac{(n-1)(n+d)}{2c}.
\end{align*}
The second bound in the minimum follows from the commute-time
identity and Proposition~\ref{prop:two-terminal-resistance}.

To obtain the uniform bound, put \(B_n=3(n-1)(n+3)/4\).
For \(3\leqslant d\leqslant(n+3)/2\),
\[
 \frac{d(n-1)(n+d)}{2(d-1)}-B_n
 =\frac{(d-3)(n-1)(2d-n-3)}{4(d-1)}\leqslant0.
\]
For \(d\geqslant(n+3)/2\), the second bound decreases with
\(d\), and its value at \((n+3)/2\) is
\(n(n-1)(n+3)/(2(n+1))\leqslant B_n\).
Cycles satisfy
\(t_{\mathrm{hit}}(C_n)=\lfloor n^2/4\rfloor\), so they also
satisfy the uniform bound.

Finally, the cubic chains \(X_n\) have end-to-end commute time
\(3n^2/2+O(n)\), so at least one of the two directed hitting times
is at least half this value. Together with
\eqref{eq:hitting-time-finite}, this proves the last assertion.
\end{proof}


\begin{thebibliography}{99}

\bibitem{AbdiGhorbaniPartI}
M. Abdi and E. Ghorbani,
\emph{Graphs with minimum algebraic connectivity I:
Proofs of Aldous--Fill and Guiduli--Mohar conjectures},
arXiv:2609.26699v1, 2026,
\url{https://arxiv.org/abs/2609.26699v1}.

\bibitem{AbdiGhorbaniPartII}
M. Abdi and E. Ghorbani,
\emph{Graphs with minimum algebraic connectivity II:
Regular graphs of even degree},
arXiv:2609.26700v1, 2026,
\url{https://arxiv.org/abs/2609.26700v1}.

\bibitem{AbdiGhorbaniQuartic}
M. Abdi and E. Ghorbani,
\emph{Quartic graphs with minimum spectral gap},
J. Graph Theory 102 (2023), 205--233.

\bibitem{AbdiGhorbaniMinDegree}
M. Abdi and E. Ghorbani,
\emph{Graphs of degree at least \(3\) with minimum algebraic connectivity},
SIAM J. Discrete Math. 38 (2024), 2447--2467.

\bibitem{AbdiGhorbaniDiameter}
M. Abdi and E. Ghorbani,
\emph{Minimum algebraic connectivity and maximum diameter:
Aldous--Fill and Guiduli--Mohar conjectures},
J. Combin. Theory Ser. B 167 (2024), 164--188.

\bibitem{AbdiGhorbaniImrich}
M. Abdi, E. Ghorbani, and W. Imrich,
\emph{Regular graphs with minimum spectral gap},
European J. Combin. 95 (2021), 103328, 18 pp.

\bibitem{AksoyChungTaitTobin}
S. G. Aksoy, F. R. K. Chung, M. Tait, and J. Tobin,
\emph{The maximum relaxation time of a random walk},
Adv. in Appl. Math. 101 (2018), 1--14.

\bibitem{AldousFill}
D. Aldous and J. A. Fill,
\emph{Reversible Markov Chains and Random Walks on Graphs},
unfinished monograph, 2002; recompiled 2014,
\url{https://www.stat.berkeley.edu/~aldous/RWG/book.html}.

\bibitem{AlonMilman1985}
N. Alon and V. D. Milman,
\emph{\(\lambda_1\), isoperimetric inequalities for graphs, and superconcentrators},
J. Combin. Theory Ser. B 38 (1985), 73--88.

\bibitem{BauerKellerWojciechowski2015}
F. Bauer, M. Keller, and R. K. Wojciechowski,
\emph{Cheeger inequalities for unbounded graph Laplacians},
J. Eur. Math. Soc. 17 (2015), 259--271.

\bibitem{BrandGuiduliImrich}
C. Brand, B. Guiduli, and W. Imrich,
\emph{Characterization of trivalent graphs with minimal eigenvalue gap},
Croat. Chem. Acta 80 (2007), 193--201.

\bibitem{BussemakerEtAl}
F. C. Bussemaker, S. \v{C}obelji\'c, D. M. Cvetkovi\'c, and J. J. Seidel,
\emph{Cubic graphs on \(\leqslant14\) vertices},
J. Combin. Theory Ser. B 23 (1977), 234--235.

\bibitem{Chung1989}
F. R. K. Chung,
\emph{Diameters and eigenvalues},
J. Amer. Math. Soc. 2 (1989), 187--196.

\bibitem{Chung1997}
F. R. K. Chung,
\emph{Spectral Graph Theory},
CBMS Regional Conference Series in Mathematics 92,
American Mathematical Society, Providence, RI, 1997.

\bibitem{Cioaba2007}
S. M. Cioab\u{a},
\emph{The spectral radius and the maximum degree of irregular graphs},
Electron. J. Combin. 14 (2007), Research Paper 38, 10 pp.

\bibitem{CioabaGregoryNikiforov2007}
S. M. Cioab\u{a}, D. A. Gregory, and V. Nikiforov,
\emph{Extreme eigenvalues of nonregular graphs},
J. Combin. Theory Ser. B 97 (2007), 483--486.

\bibitem{Dodziuk1984}
J. Dodziuk,
\emph{Difference equations, isoperimetric inequality and transience of certain random walks},
Trans. Amer. Math. Soc. 284 (1984), 787--794.

\bibitem{Fiedler1973}
M. Fiedler,
\emph{Algebraic connectivity of graphs},
Czechoslovak Math. J. 23 (1973), 298--305.

\bibitem{GuiduliThesis}
B. Guiduli,
\emph{Spectral Extrema for Graphs},
Ph.D. thesis, University of Chicago, 1996.

\bibitem{Guiduli1997}
B. Guiduli,
\emph{The structure of trivalent graphs with minimal eigenvalue gap},
J. Algebraic Combin. 6 (1997), 321--329.

\bibitem{HoffmanKahlePaquette2021}
C. Hoffman, M. Kahle, and E. Paquette,
\emph{Spectral gaps of random graphs and applications},
Int. Math. Res. Not. IMRN 2021, no. 11, 8353--8404.

\bibitem{HornJohnson}
R. A. Horn and C. R. Johnson,
\emph{Matrix Analysis}, 2nd ed.,
Cambridge University Press, Cambridge, 2013.

\bibitem{LalPatraSahoo}
A. K. Lal, K. L. Patra, and B. K. Sahoo,
\emph{Algebraic connectivity of connected graphs with fixed number of pendant vertices},
Graphs Combin. 27 (2011), 215--229.
\href{https://doi.org/10.1007/s00373-010-0975-0}{doi:10.1007/s00373-010-0975-0}.

\bibitem{LawlerSokal1988}
G. F. Lawler and A. D. Sokal,
\emph{Bounds on the \(L^2\) spectrum for Markov chains and Markov processes:
a generalization of Cheeger's inequality},
Trans. Amer. Math. Soc. 309 (1988), 557--580.

\bibitem{LevinPeresWilmer}
D. A. Levin and Y. Peres,
\emph{Markov Chains and Mixing Times}, 2nd ed.,
with contributions by E. L. Wilmer,
American Mathematical Society, Providence, RI, 2017.

\bibitem{LiuSpectralRadius2024}
L. Liu,
\emph{Extremal spectral radius of nonregular graphs with prescribed maximum degree},
J. Combin. Theory Ser. B 169 (2024), 430--479.

\bibitem{LiuXue}
R. Liu and J. Xue,
\emph{Cubic bipartite graphs with minimum spectral gap},
Adv. in Appl. Math. 180 (2026), 103134.

\bibitem{LiuDualCheeger2015}
S. Liu,
\emph{Multi-way dual Cheeger constants and spectral bounds of graphs},
Adv. Math. 268 (2015), 306--338.

\bibitem{Lovasz1996}
L. Lov\'asz,
\emph{Random walks on graphs: a survey},
in \emph{Combinatorics, Paul Erd\H{o}s is Eighty, Vol.~2}
(Keszthely, 1993), Bolyai Soc. Math. Stud. 2,
J\'anos Bolyai Math. Soc., Budapest, 1996, 353--397.

\bibitem{LubetzkySly2010}
E. Lubetzky and A. Sly,
\emph{Cutoff phenomena for random walks on random regular graphs},
Duke Math. J. 153 (2010), 475--510.

\bibitem{MarcusSpielmanSrivastava2015}
A. W. Marcus, D. A. Spielman, and N. Srivastava,
\emph{Interlacing families I: Bipartite Ramanujan graphs of all degrees},
Ann. of Math. (2) 182 (2015), 307--325.

\bibitem{Mohar1991}
B. Mohar,
\emph{Eigenvalues, diameter, and mean distance in graphs},
Graphs Combin. 7 (1991), 53--64.

\bibitem{Schrijver}
A. Schrijver,
\emph{Theory of Linear and Integer Programming},
John Wiley \& Sons, Chichester, 1986.

\end{thebibliography}
\end{document}